\documentclass[12pt,a4paper,leqno]{amsart}

\usepackage[top=1.1in, bottom=1.3in, left=1.5in, right=1.5in]{geometry}

\usepackage[ansinew]{inputenc}
\usepackage[T1]{fontenc}
\usepackage[english]{babel}
\usepackage{amsthm}
\usepackage{amsfonts}         
\usepackage{amsmath}
\usepackage{amssymb}
\usepackage{breqn}
\usepackage{bm}
\usepackage{url}
\usepackage{esint}
\usepackage{fancyhdr}

\newcommand{\R}{\mathbb{R}}
\newcommand{\C}{\mathbb{C}}

\newcommand{\N}{\mathbb{N}}

\newcommand{\Z}{\mathbb{Z}}

\newcommand{\A}{\mathcal{A}}
\newcommand{\B}{\mathcal{B}}
\newcommand{\CC}{\mathcal{C}}

\theoremstyle{plain}
\newtheorem{theorem}{Theorem}
\newtheorem{lemma}[theorem]{Lemma}
\newtheorem{prop}[theorem]{Proposition}

\theoremstyle{remark}
\newtheorem{remark}{Remark}

\numberwithin{equation}{section}
\begin{document}

\title{Large prime factors on short intervals}

\author{Jori Merikoski}

\address{Department of Mathematics and Statistics, University of Turku, FI-20014 University of Turku,
Finland}
\email{jori.e.merikoski@utu.fi}

\begin{abstract}
We show that for all large enough $x$ the interval $[x,x+x^{1/2}\log^{1.39}x]$ contains numbers with a prime factor $p > x^{18/19}.$ Our work builds on the previous works of Heath-Brown and Jia (1998) and Jia and Liu (2000) concerning the same problem for the longer intervals $[x,x+x^{1/2+\epsilon}].$ We also incorporate some ideas from Harman's book \emph{Prime-detecting sieves} (2007). The main new ingredient that we use is the iterative argument of Matom\"aki and Radziwi\l\l \, (2016) for bounding Dirichlet polynomial mean values, which is applied to obtain Type II information. This allows us to take shorter intervals than in the above-mentioned previous works. We have also had to develop ideas to avoid losing any powers of $\log x$ when applying Harman's sieve method.
\end{abstract}

\maketitle
\tableofcontents

\section{Introduction and results}
The current best result for prime numbers in short intervals is the theorem of Baker, Harman and Pintz \cite{BHP} from 2001, which states that for all large enough $x$ the interval $[x,x+x^{1/2+1/40}]$ contains prime numbers. It is well known that conditionally on the Lindel\"of Hypothesis, for instance, there are prime numbers in the intervals $[x,x+x^{1/2+\epsilon}]$ for any $\epsilon>0$ for all large enough $x$.

Since showing that the short intervals $[x,x+x^{1/2+\epsilon}]$ contain prime numbers seems to be beyond the current methods, it is sensible to consider the easier problem of finding numbers with a large prime factor. Consider intervals $[x,x+x^{1/2+\epsilon}],$ where the aim is to show that the interval contains an integer with a prime factor $p>x^{1-\gamma}$ for as small $\gamma >0$ as possible;  this problem has attracted the attention of many authors. In 1973 Jutila \cite{Jut} obtained this for $\gamma=1/3+\epsilon$ by considering numbers $pn,$  where $n\asymp x^{\gamma}$ is very smooth. This was then improved to $\gamma=0.18$ by Balog, Harman and Pintz \cite{BHP4} in 1983 (improving the earlier works \cite{Bal},\cite{Bal2},\cite{BHP3} of the same authors). In 1996, Heath-Brown \cite{HB} combined Jutila's method with sieve arguments to obtain $\gamma=1/12 +\epsilon,$ which was then improved by Heath-Brown and Jia \cite{HJ} to $\gamma=1/18+\epsilon$ in 1998. Harman in an unpuplished manuscript got $\gamma=1/20.$ The current record is $\gamma=1/26+\epsilon$ by Jia and Liu \cite{JL} from 2000.

In comparison, for slightly shorter intervals $[x,x+x^{1/2}]$ the best exponent is $\gamma=0.2572$ by Baker and Harman \cite{BH}, which is much larger. A natural question then is that at what point does this significant change in the exponent  $\gamma$ become neccesary. This is interesting especially in light of the recent result of Matom\"aki and Radziwi\l\l \, \cite{MR} (given there as a corollary of their much more general theorem on multiplicative functions) that for all $\epsilon >0$ there is a constant $C=C(\epsilon)$ such that the intervals $[x,x+C\sqrt{x}]$ contain $x^{\epsilon}$-smooth numbers  (this was previously known only for intervals $[x,x+x^{1/2}\log^{7/3+\delta}x]$, cf. \cite{Mat}). The main idea in the current manuscript is to combine their argument with the methods used for finding numbers with large prime factors in $[x,x+x^{1/2+\epsilon}],$ so that we can reduce the length of the interval as much as possible. Unfortunately, it appears that the intervals $[x,x+C\sqrt{x}]$ remain out of our reach; our main theorem is
\begin{theorem} \label{main} Let $\beta:=1.388 \dots$ denote the minimum of the function
\begin{align*}
r \mapsto \frac{\log(1-\log(r-1)) -\log(-\log(r-1)) + \log 2}{2\log r} -\frac{1}{2}
\end{align*}
for $1<r<2,$ which is obtained at $r:=1.625\dots$ Then, for all $\delta>0$ and for all large enough $x,$ the interval $[x,x+x^{1/2} \log^{\beta+\delta} x]$ contains numbers with a prime factor $p>x^{1-\gamma}$ for $\gamma=1/19.$ 
\end{theorem}

We now sketch the main ideas in the proof. The argument is based on that of Heath-Brown and Jia \cite{HJ} (also described in Chapter 5 of Harman's book \cite{Har}), so we first describe the argument they use for finding numbers with large prime factors on intervals $I:=[x,x+x^{1/2+\epsilon}];$ the aim is to show that $pn \in I$ for some prime $p$ and for some integer $n \asymp x^{\gamma}.$ Heath-Brown and Jia consider $n$ which are very smooth, of the form $n=q_1\cdots q_K,$ for primes $q_i \sim Q:= x^{\gamma/K}$ and  $K \geq 4/\epsilon$ (this idea is originally from Jutila's work \cite{Jut}). The task is then to obtain a lower bound for
\begin{align*}
\sum_{pq_1 \cdots q_K \in I, \, q_i \sim Q} 1_{p \in \mathbb{P}}.
\end{align*}
By applying Harman's sieve we can obtain a lower bound for this sum if we have an asymptotic formula for sums over $x^{\gamma}$-almost-primes of the form
\begin{align*}
\sum_{uvnq_1 \cdots q_K \in I, \, q_i \sim Q} a_u b_v 1_{(n,P(x^{\gamma}))=1},
\end{align*}
for any (say, bounded) coefficients $a_u,b_v$ supported on $ u < x^{1/4},$ $v < x^{1/2-\gamma}$ (cf. Chapter 5 of Harman's book \cite{Har}, for instance). This in turn can be reduced to obtaining asymptotic formulae for the so-called Type I/II and Type II sums (in the language of Harman's book \cite{Har})
\begin{align*}
\text{Type I/II:} \sum_{uvnq_1 \cdots q_K \in I, \, q_i \sim Q} a_u b_v, \quad \quad \text{Type II:} \sum_{uvq_1 \cdots q_K \in I, \, q_i \sim Q} a_u b_v.
\end{align*}
For the Type I/II sum the coefficients $a_u,b_v$ are supported on  $ u < x^{1/4},$ $v < x^{1/2-\gamma}$, and for the Type II sum the coefficients are supported on the Type II range $u,v \in [x^{1/2-\gamma}, x^{1/2}].$ The name for the Type II sum comes from the bilinear structure of the sum; the idea is that at some suitable stage we can use Cauchy-Schwarz to separate the variables $u$ and $v.$ In the name Type I/II, the `I' refers to the fact that we have a long smooth variable $n$, and the `II' again refers to the bilinear structure which permits the use of Cauchy-Schwarz at some point.

In the Type I/II sum we can apply Perron's formula to reduce matters to obtaining a mean value estimate for Dirichlet polynomials. Since we have a long smooth variable $n,$ we can use methods from the theory of the zeta function (e.g. reflection principle or the fourth moment estimate) to handle these sums; see the beginning of Section \ref{typei} below for a sketch of an argument of this type.

We now describe the argument used for the Type II sums in more detail (for intervals $[x,x+x^{1/2+\epsilon}]$), since obtaining the Type II estimate sets the restriction for the length of the interval in our main theorem; by applying Perron's formula, this can be reduced to obtaining a mean value estimate of the form
\begin{align*}
\int_{T_0}^{x^{1/2-\epsilon}} |A(1+it)B(1+it)Q(1+it)^{K}| \, dt \, \ll \, \log^{-C} x,
\end{align*}
where $T_0=\log^{100C} x,$ and
\begin{align*}
A(s)= \sum_{u\sim U}a_u u^{-s}, \quad B(s)=\sum_{v\sim V} b_v v^{-s}, \quad Q(s)=\sum_{q\sim Q, \, q \in \mathbb{P}} q^{-s}
\end{align*}
for $U,V \in [x^{1/2-\gamma}, x^{1/2}],$ $UV \asymp x^{1-\gamma},$ $Q\asymp x^{\gamma/K},$ $K \geq 4/\epsilon.$ By Vinogradov's zero-free region, we see that the integral is bounded by.
\begin{align*}
(\log^{-2C} x)\int_{T_0}^{x^{1/2-\epsilon}} |A(1+it)B(1+it)Q(1+it)^{K-1}| \, dt.
\end{align*}
Since $K$ is large enough, we can find $L \leq K-1$ such that $UQ^L,VQ^{K-L-1} \gg x^{1/2-\epsilon}.$
Hence, by Cauchy-Schwarz and the Mean value theorem for Dirichlet polynomials (cf. Lemma \ref{mvt1} below)
\begin{align} 
\int_{T_0}^{x^{1/2-\epsilon}} &|A(1+it)B(1+it)Q(1+it)^{K-1}| \, dt \, \leq \nonumber \\
&\left(\int_{T_0}^{x^{1/2-\epsilon}} |A(1+it)Q(1+it)^L|^2 \, dt \right)^{1/2} \left(\int_{T_0}^{x^{1/2-\epsilon}} |B(1+it)Q(1+it)^{K-L-1}|^2 \, dt \right)^{1/2} \nonumber \\ \label{typebound}
&  \ll   \left(\frac{x^{1/2-\epsilon}}{UQ^L} +1 \right)^{1/2} \left(\frac{x^{1/2-\epsilon}}{VQ^{K-L-1} } +1 \right)^{1/2} \log^C x \ll \log^C x,
 \end{align}
which is sufficient.

Consider then the shorter intervals $[x,x+x^{1/2}\log^{\epsilon} x].$ Here the upper integration limit in the above mean value becomes $x^{1/2}\log^{-\epsilon} x$ instead of $x^{1/2-\epsilon}$. Hence, to make the above argument work, we now need to factor the product $A(1+it)B(1+it)Q(1+it)^{K-1}$ into almost equally long parts, with a much greater accuracy of $\log^{\epsilon} x$; this is because we now need both of the polynomials in the factorization to have length greater than $x^{1/2}\log^{-\epsilon} x.$  This means that $K$ must be very large, which in turn means that our set $\{pq_1 \dots q_K\}$ becomes very sparse. This causes losses in the Mean value theorem for Dirichlet polynomials, which now need to be gained back. For short intervals $[x,x+x^{1/2}\log^{\epsilon} x]$ we must apply the method of Matom\"aki and Radziwi\l\l, which can only give a small saving of $\log^{-\epsilon/5} x$ over the mean value estimate, which is insufficient to gain back the losses. To put it simply, we have a situation of two competing requirements, a high density versus a strong factorization property, which forces a compromise.  By taking interval of the form $y=x^{1/2}\log^{B} x,$ we can work with a set of density of some power of $\log x,$ which gives us just enough room to obtain the factorization property.

To maximize the density, we must take our small prime factors on intervals longer than dyadic intervals (cf. intervals $I_j$ below). This makes the computations in the Type II estimate much trickier. We must also exercise great care in every step so that we do not lose any additional powers of $\log x;$ for this purpose we have had to develop some new ideas, especially for the Type I/II estimtate (cf. Section \ref{typei} below) and in the framework of Harman's sieve (cf. Sections \ref{funprop} and \ref{buchs} below). Optimizing the set-up we find that the argument works for $y=(\log^{\beta+\delta}x)\sqrt{x}$ with $\beta$ as in Theorem \ref{main}. The value of $\gamma$ in the theorem is not necessarily the best that one can obtain, but we do not pursue this issue further here since our main focus is on the length of the interval.

The paper is structured as follows: in Section \ref{setup} we give the set-up, and in Section \ref{preli} we have collected some basic lemmata which will be used in the proofs.  Section \ref{typei} contains the proof of our Type I/II estimate (Proposition \ref{typeiestimate}). In Section \ref{typeii} we prove our Type II estimate (Proposition \ref{typeiiestimate}), which is the heart of the matter; Lemma \ref{trivial} there gives the restriction for the length of the interval (cf. line (\ref{length}) in particular, which yields the function of $r$ in Theorem \ref{main} which is minimized to optimize the result). In Section \ref{funprop} we prove the so-called Fundamental proposition (Proposition \ref{ft}), which combines the Type I/II and Type II estimates to give an asymptotic formula for certain sums over almost-primes (this corresponds to Lemma 5.3 in Chapter 5 of Harman's book \cite{Har}); we note that the proofs in Sections \ref{typei}, \ref{typeii}, and \ref{funprop} work for any $0<\gamma<1/2.$ In Section \ref{buchs} we use Buchstab's identity along with the Fundamental proposition and the Type II estimate to give a lower bound for a sum over numbers with large prime factors; there we encounter sums for which we cannot obtain an asymptotic formula. The contribution from these sums is bounded by numerical computations, which ultimately determines the exponent $\gamma$ in Theorem \ref{main}. These numerical computations are performed using Python 3.7; for the code see the codepad links at the end of the paper.

While it is not a prerequisite, the reader will find the proofs easier to digest if he is familiar with the contents of Chapters 3 and 5 of Harman's book \cite{Har}. However, we have tried to sketch the relevant ideas before giving the full proofs in each situation. It may also be helpful at the first pass to read only the statements of the Type I/II and Type II estimates (Propositions \ref{typeiestimate} and \ref{typeiiestimate}), and after reading Section \ref{funprop} return to the proofs of these estimates; this should help to motivate the exact form of these propositions.

\subsection*{Notations}
We use the following asymptotic notations: For positive functions $f,g,$ we write $f \ll g$ or $f= \mathcal{O}(g)$ if there is a constant $C$ such that $f  \leq C g.$ $f \asymp g$ means $g \ll f \ll g.$ The constant may depend on some parameter, which is indicated in the subscript (e.g. $\ll_{\epsilon}$).
We write $f=o(g)$ if $f/g \to 0$ for large values of the variable.

It is useful for us to introduce the following unconventional notation: $f \lesssim g$ means that there is some positive function $\psi = \log^{o(1)} x$ so that $f \leq \psi g$ (e.g. a common estimate we use is $(\log \log x)^{\mathcal{O}(1)} \lesssim 1$). A recurring theme is that we are interested in estimates only up to factors $\lesssim 1.$

In general, $C$ stands for some large constant, which may not be the same from place to place. For variables we write $n \sim N$ meaning $N<n \leq eN$ (an $e$-adic interval), and $n \asymp N$ meaning $N/C < n < CN$ (a $C^2$-adic interval) for some constant $C>1$ which is large enough depending on the situation; for example, we write
\begin{align*}
\sum_{m \asymp M} \sum_{n \asymp N} |a_m||b_n| \, \ll \, \sum_{mn \asymp MN} |a_m||b_n|
\end{align*}
meaning that the implied constants are $C,$ $C'$ and $CC'.$  If not otherwise stated the symbols $p,q,r$  denote primes and $c,d,k,l,m,n$ denote integers.

For a statement $E$ we denote by $1_E$ the characteristic function of that statement. For complex numbers we use the notation $s=\sigma+it,$ $\sigma,t\in \R.$

We define $P(w):= \prod_{p\leq w} p,$ and for any integer $d$ we write $P^-(d):= \min \{p: \, p | d\},$ $P^+(d):= \max \{p: \, p | d\}.$ The divisor function is denoted by $\tau(d).$ We denote by $\tau^{(k)}$  the $k$-fold divisor function (i.e. $\tau^{(k+1)}=1 \ast \tau^{(k)},$ $\tau^{(1)} \equiv 1$, where $\ast$ is the Dirichlet convolution). We set $\tau_{w}(d) := (1_{(\cdot, P(w))=1} \ast 1)(d) ,$ which equals one plus the number of divisors whose prime factors are $> w.$

\subsection*{Acknowledgements} The author is grateful for his supervisor Kaisa Matom\"aki for suggesting the topic as well as for many helpful discussions. The author also wishes to thank Joni Ter\"av\"ainen for supplying a Python code for computing Buchstab integrals, on which the code we use is built upon. The author is also grateful for both of the above for reading a preliminary version of this manuscript; their comments have greatly improved the article. During the work the author was supported by grants from the Finnish Cultural Foundation's Central Fund and the Magnus Ehrnrooth Foundation.

\section{Set-up and conventions} \label{setup}

Let $K,L>0$ denote large and $\epsilon >0$ small constants, with $K,L \, \gg \, 1/\epsilon.$ We will abuse the notation so that we write $K^{-1},L^{-1},\epsilon = o(1)$ meaning that we will eventually choose each constant to be large or small enough. In addition, we let $\delta > 0$ denote a fixed small constant.  We choose $\epsilon$ small and $K,L$ large enough so that  $E(\epsilon,K,L) = o(\delta)$ for certain error terms $E(\epsilon,K,L)$ that occur below.

We now give our basic set-up; to collect all of the definitions in one place, we postpone the motivation for this construction to the paragraphs following Proposition \ref{mainp} below.  Let $r=1.625\dots$ and $\beta=1.388\dots$ be as in Theorem \ref{main}.  Set
\begin{align*}
& \theta := r-1 + \epsilon, \quad \quad \omega := x^{\gamma(r-1)}, \quad \quad w := x^{1/(\log \log x)^2},\\ & J :=  \left \lceil \frac{1}{\log r} \log \left( \frac{\log \omega}{\log K}\right) \right \rceil, \quad \quad  H:= \left \lceil \frac{(\log \log x)^{1/2} }{ 10\delta} \right \rceil,
\\ & Q_1 := \log^{10\delta} x,  \quad \quad Q_2 := Q_1^H, \quad \quad Q_3 := \exp (2 \lfloor \log^{9/10} x \rfloor ), 
\end{align*}
and define intervals 
\begin{align*}
 I &:= [1,x^{\epsilon}], \quad \, I_j := \begin{cases} (\omega^{(1-2\epsilon) r^{-j}}, \omega^{(1-\epsilon)r^{-j}}], \, &j=1,2, \dots, K \\
  (\omega^{\theta r^{-j}}, \omega^{(1-\epsilon)r^{-j}}], \, & j= K+1, \dots, J. \end{cases}
\end{align*}
Note that since $r > (1+\sqrt{5})/2 =1.618...,$ we have $\theta > 1/r$ so that the intervals $I_j$ are disjoint. Note also that $ \omega^{r^{-J}} \in [K^{1/r},K]$ by the choice of $J$.   Similarly as in Section 10 of \cite{MR}, we define the piecewise linear smoothing of the indicator function of $[1-\eta,1+\eta]$ by
\begin{align} \label{ffun}
f_{\eta,\xi}(z) : = \begin{cases} 1, & 1-\eta \leq z \leq 1+\eta, \\
(1-z+\eta + \xi )/\xi, & 1+\eta < z \leq 1+\eta + \xi, \\
(z-1+\eta + \xi )/\xi, & 1-\eta -\xi \leq z \leq 1-\eta, \\
0 & \text{otherwise.}
\end{cases}
\end{align}
Define (this definition is made so that the bound (\ref{length}) in Section \ref{typeii} is satisfied)
\begin{align} \label{y}
y:&=x^{1/2}( \log^{-1/2 + 10\delta} x) 2^{J/2} (1+\log 1/\theta)^{J/2} (\log 1/\theta)^{-J/2} \\ \nonumber
& = x^{1/2}\exp\left(\left( \frac{\log(1+\log(1/\theta)) -\log(\log(1/\theta)) + \log 2}{2\log r} -\frac{1}{2} + 10\delta + o(1)\right)\log\log x \right)  \\ \nonumber
&=  x^{1/2} \log^{\beta+10\delta+o(1)} x ,
\end{align}
and
\begin{align*}
\eta_\A &:=  y/x, \quad \quad \quad \quad\xi_\A := (\log^{-\epsilon}x) y/x, \quad\quad \quad \quad  f_{\mathcal{A}}(z) := f_{\eta_A, \xi_\A}(z),\\
\eta_B &:= \log^{-100} x, \quad\quad\quad  \xi_\B := \log^{-100-\epsilon} x, \quad \quad \quad \quad f_{\mathcal{B}}(z):= f_{\eta_\B, \xi_\B}(z),\\
W_{\mathcal{A}}(n) &:= \sum_{\substack{ncc'q_1q_2q_3 p_1 \cdots p_J r_1 \cdots r_L \asymp x }} f_{\mathcal{A}}(ncc' q_1q_2q_3p_1 \cdots p_J r_1 \cdots r_L/x), \\
W_{\mathcal{B}}(n)&:= \sum_{\substack{ncc'q_1q_2q_3 p_1 \cdots p_J r_1 \cdots r_L \asymp x }} f_{\mathcal{B}}(ncc' q_1q_2q_3p_1 \cdots p_J r_1 \cdots r_L/x), \\
 S(\mathcal{A},z) &:= \sum_{\substack{(n,P(z))=1}} W_{\mathcal{A}}(n), \quad \quad \quad 
S(\mathcal{B},z) := \sum_{\substack{(n,P(z))=1}} W_{\mathcal{B}}(n),
\end{align*}
where in the summations $c,c'$ are $w$-smooth integers (that is, $P^+ (c),P^+ (c') \leq w$),  $q_i,p_j, r_l$ are primes, $ q_1 \sim Q_1,$ $ q_2 \sim Q_2,$  $Q_3^{1/2} < q_3 \leq Q_3,$ and
\begin{align} \label{summation}
c,c' &\in I, \quad \quad 
 p_{j} \in I_j \,\, \text{for} \,\, j=1,2,\dots, J, \quad \quad r_l \sim \log^\epsilon x \,\,\text{for} \,\, l=1,2, \dots, L. 
\end{align}
From here on until Section \ref{funprop}, the above conditions will always apply to the corresponding variables and will usually be suppressed in the notation. Same applies to $q'_i, p'_j, r'_l.$ Note that by definitions $p_J \in I_J\subset [K^{(r-1)/r},K].$ We also note that since the intervals $I_j$ are short for $j \leq K$, we have 
\begin{align} \label{hash}
cc'q_1q_2q_3 p_1 \cdots p_J r_1 \cdots r_L = x^{\gamma(r-1)(r^{-1} + r^{-2} + \cdots)+o(1)} = x^{\gamma +o(1)}.
\end{align}
We remark that by the Prime Number Theorem
\begin{align}
 \frac{\log^{100}x }{x} \sum_{p  } W_{\mathcal{B}}(p) & =\frac{2+o(1)}{\log (x^{1-\gamma+o(1)})} \sum_{\substack{c,c',q_1,q_2,q_3,\\ p_1, \cdots ,p_J, \\ r_1, \cdots ,r_L}}  (cc'q_1q_2q_3 p_1 \cdots p_J r_1 \cdots r_L)^{-1} \nonumber \\ \label{omega}
 &=\frac{(2 +o(1))\Omega}{(1-\gamma )\log x} ,
\end{align}
where $\Omega$ is 
\begin{align*}
\left( \sum_{c \leq D, \, P^+(c) \leq w} \hspace{-15pt} c^{-1}\right)^2 & \left( \sum_{q_1 \sim Q_1} q_1^{-1}\right)\left(  \sum_{q_2 \sim Q_2} q_2^{-1}\right) \left( \sum_{ Q_3^{1/2} \leq q_3 \leq Q_3} \hspace{-15pt} q_3^{-1}\right)     \left(\sum_{r \sim \log^\epsilon x} \hspace{-5pt}r^{-1} \right)^L \prod_{j=1}^J \left( \sum_{p \in I_j} p^{-1}\right)  \\
&= (\log^{2+o(1)} x)  \prod_{j=K+1}^J  \left( (1+\mathcal{O}(\log^{-1} K) ) \log((1-\epsilon)/\theta)  \right) \\
  &=(\log^{2+o(1)} x)  (\log 1/\theta)^{J},
\end{align*}
since for $j\leq K$ we have $\sum_{p \in I_j} p^{-1} \asymp 1,$ and for $j >K$ by the Prime Number Theorem (recall that $ \omega^{r^{-J}} \in [K^{1/r},K]$)
\begin{align*}
\sum_{p \in I_j} p^{-1} &= (1+\mathcal{O}(\log^{-1} (\omega^{\theta r^{-j}})))(\log\log(\omega^{(1-\epsilon)r^{-j}}) - \log\log(\omega^{\theta r^{-j}})) \\
&= (1+\mathcal{O}(\log^{-1} K))\log((1-\epsilon)/\theta).
\end{align*}
Since $z \mapsto f_\A(z/x)$ is supported on $[x-2y,x+2y],$ and  $S(\mathcal{C},x^{(1-\gamma )/2+\delta})=\sum_{p  } W_{\mathcal{C}}(p)$ for $\CC=\A$ and $\CC=\B$, Theorem \ref{main} follows once we prove
\begin{prop} \label{mainp} For every $\gamma \geq 1/19$ there exists $\mathfrak{C}(\gamma) > 0,$ such that for all large enough $x$
 \begin{align*}
 S(\mathcal{A},x^{(1-\gamma)/2 +\delta}) > \mathfrak{C}(\gamma) \cdot \frac{ y \log^{100}x }{x} S(\mathcal{B},x^{(1-\gamma )/2+\delta}),
 \end{align*}
We can take $\mathfrak{C}(1/19) = 0.007.$
\end{prop}

The above definitions are tailored with the Type II estimate in mind. To obtain our Type II estimate, we roughly speaking must require that for any $x^{1/2-\gamma+\delta}<u,v<x^{1/2-\delta}$ such that $uvcc'q_1q_2q_3 p_1 \cdots p_J r_1 \cdots r_L \sim x,$ we can form a partition of the product $q_1q_2q_3 p_1 \cdots p_J = (\pi)\cdot(\tau)$ in such a way that $uc\pi ,vc'\tau  \asymp \sqrt{x}/\log^{L\epsilon/2} x.$  The $w$-smooth parameters $c,c'$ are added to boost the density; the restriction to $w$-smooth numbers will be useful in the Type I/II estimate (cf. Lemma \ref{implies}, this is a kind of `arithmetic smoothing' of $1_{[1,x^{\epsilon}]}$). We will choose $L$ large to make $L\epsilon$  sufficiently large and then use the primes $r_l$ to balance the partition suitably (with an accuracy $\log^\epsilon x$). The primes $q_1,q_2,q_3$ will be used in the Type II estimate to bound a Dirichlet polynomial mean value by using the method of Matom\"aki and Radziwi\l\l \, \cite{MR}. It is technically easier to include them separately, even though the primes $p_j$ could in principle be used to the same effect. Note that the ranges $Q_1,Q_2$ are small enough so that
\begin{align} \label{smallprimes}
 \sum_{q \sim Q_1} q^{-1}, \, \sum_{q \sim Q_2} q^{-1}, \, \left( \sum_{q \sim Q_1} q^{-1} \right)^H \hspace{-5pt}, \, \sum_{r \sim \log^{\epsilon} x} r^{-1} \, = \, \log^{o(1)} x,
\end{align}
that is, using $e$-adic intervals does not cause significant losses. The range for $q_3$ is large (in particular, $\log Q_3 \neq \log^{o(1)}x$), which forces us to take a longer interval $q_3 \in (Q_3^{1/2},Q_3].$

\begin{remark} Those familiar with the well-factorability of weights in the linear sieve will note a similarity with our construction (cf. Chapter 12 of Friedlander and Iwaniec's book \cite{FI}). An integer $d$ is said to be well-factorable of level $D$ if for any $D_1D_2 = D,$ $D_1 \geq 1, D_2 \geq 1,$ there are $d_1 \leq D_1,$ $d_2 \leq D_2$ such that $d=d_1d_2.$ If $d=p_1\cdots p_J,$ $p_1 > p_2 > \dots > p_J,$ then a sufficient condition is that for all $j\leq J$
\begin{align} \label{wellf}
p_1\cdots p_{j-1} p_j^2 \leq D
\end{align}
(see the proof of Lemma 12.16 in \cite{FI}). This becomes stricter as $D$ decreases, and if $D \leq C p_1 \cdots p_J,$ then $d$ has to have a factor on every $C$-adic interval $[z,Cz] \subset [1,d]$. We require a very strong level of well-factorability, that is, of level $D \ll p_1 \cdots p_J.$ Thus, the criteria (\ref{wellf}) becomes
\begin{align*}
p_{j+1} \cdots p_J \gg p_j,
\end{align*}
which motivates the definition of the intervals $I_j$ with $\theta=r-1+\epsilon$ in our situation.
\end{remark}

\section{Preliminaries} \label{preli} 
We have gathered here some basic results for reference. The first two lemmata are mean value estimates for Dirichlet polynomials. The proof of the first can be found in Chapter 9 of the book \cite{IK} by Iwaniec and Kowalski, for instance.

\begin{lemma} \label{mvt1} Let $N \geq 1$ and $F(s) = \sum_{n \sim N} a_n n^{-s}$ for some $a_n\in \C.$  Then
\begin{align*}
\int_0^T | F(it) |^2 \, dt \ll (T+N) \sum_{n \sim N} |a_n|^2.
\end{align*}
\end{lemma}

The following mean value theorem improves the above bound for sparse sequences, as is noted in the work of Ter\"av\"ainen \cite{Ter}. Similarly as in there, we note that the lemma follows from Lemma 7.1 of Chapter 7 in \cite{IK} by taking $Y=10T$ there.

\begin{lemma}\emph{\textbf{(Improved mean value theorem).}}\label{mvt2}  Let $N \geq 1$ and $F(s) = \sum_{n \sim N} a_n n^{-s}$ for some $a_n\in \C.$  Then
\begin{align*}
\int_0^T | F(it) |^2 \, dt \ll T \sum_{n \sim N} |a_n|^2 + T \sum_{1 \leq h \leq N/T} \sum_{\substack{ m-n =h  \\ m,n \sim N}} |a_m||a_n|.
\end{align*}
\end{lemma}

\begin{remark} Suppose that $a_n$ is the indicator function of some well-behaved sparse set with a density $\rho$ around $N.$ Then we expect that $\sum_{n \sim N} |a_n||a_{n+h}| \ll_h \rho^2 N,$ so that the second term is $\ll \rho^2 N,$ saving a factor of $\rho$ in the second term compared to Lemma \ref{mvt1}. 
\end{remark}

We will also need the following large values result for Dirichlet polynomials supported on primes (cf. Lemma 8 of \cite{MR} for the proof). We say that $\mathcal{T} \subset \R$ is well-spaced if for all distinct $t,u \in \mathcal{T}$ we have $|t-u| \geq 1.$
\begin{lemma} \label{large} Let $P(s)=\sum_{p\sim P} a_p p^{-s}$ with $|a_p| \leq 1.$ Let $\mathcal{T} \subset [-T,T]$ be a set of well-spaced points such that $|P(1+it)| \geq P^{-\alpha}$ for all $t \in \mathcal{T}.$ Then
\begin{align*}
|\mathcal{T}| \ll T^{2\alpha}P^{2\alpha}\exp \left \{2\log T \frac{\log \log T}{\log P}\right \}.
\end{align*} 
\end{lemma}

Similarly as in \cite{MR}, the above lemma will be used in co-operation with the Hal\'asz-Montgomery inequality below (cf. \cite{IK} Theorem 9.6 for the proof, for instance).
\begin{lemma}\emph{\textbf{(Hal\'asz-Montgomery inequality).}} \label{halasz} Let $F(s) = \sum_{n \sim N} a_n n^{-s}$ and let $\mathcal{T} \subset [-T,T]$ be a set of well-spaced points. Then
\begin{align*}
\sum_{t\in \mathcal{T}} |F(it)|^2 \ll (N + |\mathcal{T}|\sqrt{T})\log(2T) \sum_{n \sim N} |a_n|^2.
\end{align*}
\end{lemma}

For any compactly supported $g:\R \to \C$ of bounded variation, define the Mellin transform
\begin{align*}
\hat{g}(s) := - \int_0^{\infty} z^s dg(z). 
\end{align*}
For any such $g$ we have for all $z >0$ the Mellin inversion formula
\begin{align*}
g(z) = \frac{1}{2\pi i} \int_{\sigma-i\infty}^{\sigma+i\infty} \frac{z^{-s}}{s} \hat{g}(s)\, ds
\end{align*}
for any $\sigma$ such that the integrand $z^{-s} \hat{g}(s)/s$ is analytic for all $s=\sigma+it,$ and the integral converges absolutely. We give here properties of the Mellin transform of $f_{\eta,\xi}(z)$ as defined in (\ref{ffun}). The proof is a straightforward computation.
\begin{lemma} \label{mellin} Suppose that $\eta,\xi >0,$ $1-\eta-\xi >0.$ Denote $s=\sigma +it.$ Then
\begin{align*}
\widehat{f}_{\eta,\xi}(s) = \frac{(1+\eta+\xi)^{s+1}-(1+\eta)^{s+1}+(1-\eta-\xi)^{s+1}-(1-\eta)^{s+1}}{\xi(s+1)}.
\end{align*}
For $\sigma \geq 1/2$ we also have the asymptotics (uniformly for all $\xi,\eta$) 
\begin{align*}
\widehat{f}_{\eta,\xi}(s)/s &= (2\eta+\xi) + \mathcal{O}\left((1+|s|)(\eta^3+\xi^3)/\xi\right),
\\
| \widehat{f}_{\eta,\xi}(s)/s| &\ll \eta + \xi, \quad \quad \quad | \widehat{f}_{\eta,\xi}(s)/s| \ll \xi^{-1} |s|^{-1}|1+s|^{-1}.
\end{align*}
\end{lemma}

We also require the following lemma, which follows from the Vinogradov zero-free region by using Perron's formula (cf. Harman \cite{Har} Chapter 1).
\begin{lemma} \label{vino} For all large enough $T,P,$ and  $s= \sigma + it,$ $|t| \sim T,$ we have
\begin{align*}
\left| \sum_{p \sim P} p^{-s} \right| \ll P^{1-\sigma} \exp \left( -\frac{\log P}{\log^{7/10} T}\right) + \frac{P^{1-\sigma} \log^3 P}{T}.
\end{align*}
\end{lemma}

We will need the approximate functional equation for $\zeta(s)$ for the Type I/II estimate. See Tao's blog post \cite{Tao} (Theorem 38) for a proof of the result in this form.
\begin{lemma} \label{approx} \emph{\textbf{(Approximate Functional Equation).}} Let $g: \R \rightarrow \C$ be $C^{\infty}$-smooth, bounded, and compactly supported. Then
\begin{align*}
\sum_{n } g( \log (n/N)) n^{-s} = \chi (s) \sum_{m} m^{s-1} g(\log (M /m))+ \mathcal{O}_{g, \epsilon} \left( |t|^{-3/4 + \epsilon}\right)
\end{align*}
for $s=1/2 + it,$ $M,N \gg 1,$ $2\pi MN = |t|,$ where $\chi:\C \to \C$ is such that $|\chi(1/2 + it)| =1.$
\end{lemma}

We also require a smoothing of the characteristic funtion of $[N, N^{1+\delta}];$ fix a function $g:\R \to [0,1]$ which is $C^{\infty}$-smooth and such that $g(x) \equiv 0$ for $x < 1,$ $g(x)\equiv1$ for $x >2.$  For $N \gg_\delta 1,$ define
\begin{align} \label{phi}
\phi_N(x) := \begin{cases} g(x/N), &N<x \leq 2N, \\
1, &2N<x<N^{1+\delta}/2 \\
1-g(2x/N^{1+\delta}), &N^{1+\delta}/2 \leq x < N^{1+\delta} \\
0, &\text{otherwise.}
 \end{cases}
\end{align} 
Notice that $\phi_N$ is also $C^{\infty}$-smooth and satisfies $\phi_N^{(k)}(x) \ll_{k,g} x^{-k}.$ We have gathered some of the properties of $\phi_N(x)$ in the following
\begin{lemma} \label{smooth} Let $0< \delta <1, \, N \gg_\delta 1.$
Then

\textbf{\emph{(i):}} For $|t| \, \gg 1$ the Mellin transform of $\phi_N$ satisfies
\begin{align*}
|\hat{\phi}_N (s) | \, \ll N^{\sigma(1+\delta)}|t|^{-1}.
\end{align*} 

\textbf{\emph{(ii):}} We have $|\hat{\phi}_N(iz)| \, \ll \min \{1, |z| \log N\} $ and
\begin{align*} 
\int_{-\infty}^{\infty} | \hat{\phi}_N(iz)| \frac{dz}{z}  \, \lesssim \, 1.
\end{align*}

\textbf{\emph{(iii):}} For  $|t|\, \leq N$ we have
 \begin{align*}
\sum_{n} \phi_N(n)n^{-s} = \frac{\hat{\phi}_N(1-s)}{1-s} + \mathcal{O}(N^{-\sigma + \mathcal{O}(\delta)}).
\end{align*} 

\textbf{\emph{(iv):}} For $s=1/2+it,$ $|t| \,> N$ we have
\begin{align*}
 \sum_{n} \phi_N(n)n^{-s} =  \chi(s) \sum_{n} \phi_N\left( \frac{|t|}{2\pi n}\right)n^{s-1} +  E(N,|t|) ,
\end{align*}
where $\chi$ is as in Lemma \ref{approx} so that $|\chi(1/2 + it)| =1$, and
\begin{align*}
|E(N,|t|)| \, \ll N^{-1/2 +  \mathcal{O}(\delta)}1_{|t| \leq N^{1+\delta}}+ |t|^{-2}N^{1/2+ \mathcal{O}(\delta)} +|t|^{-3/4 + o(1)}.
\end{align*}
\end{lemma}
\begin{proof}
\textbf{(i):} We integrate by parts to get (for $|t|\gg 1$)
\begin{align*}
| \hat{\phi}_N (s)|  =  \left |\frac{1}{1+s} \int_0^{\infty} z^{s+1} \phi_N''(z) \,  dz  \right | \ll |t|^{-1} \left(\int_{N}^{2N} +\int_{N^{1+\delta}/2}^{N^{1+\delta}} \right) z^{\sigma-1} \,dz \, \ll \, N^{\sigma (1+ \delta)} |t|^{-1}.
\end{align*} 

\textbf{(ii):} We have 
\begin{align*}
\hat{\phi}_N (s)  = -\int_0^{\infty} z^s \phi_N'(z) \,  dz = s\int_0^{\infty} z^{s-1} \phi_N(z) \, dz,
\end{align*}
where clearly the first integral gives $\ll 1$ and the second gives $ \ll |s| \log N$ for $\Re s = 0.$ Thus, by part (i) applied to the large $z$
 \begin{align*}
\int_{-\infty}^{\infty} | \hat{\phi}_N(iz)| \frac{dz}{z} \ll \int_0^{\log^{-1} N} \log N \, dz + \int_{\log^{-1} N}^{1} \frac{dz}{z} + \int_1^{\infty} \frac{dz}{z^2} \, \ll \log \log N.
\end{align*}

\textbf{(iii):} This follows directly from Lemma 8.8 of \cite{IK}.

\textbf{(iv):}  We partition $\phi_N(x)$ smoothly $e$-adically into $\sum_{k}\psi_k(x),$ where each $\psi_k$ is of the form $x \mapsto g_k(\log x - k)$ for some $C^{\infty}$-smooth, bounded and compactly supported $g_k.$ Clearly we can choose $\psi_k(x)$ so that $|\psi_k''(x)| \, \ll_k \, e^{-2k}$ (by a similar smoothing as in (\ref{phi}) but for $e$-adic intervals).  If $N<|t| \leq  e^k/100,$ we have again by Lemma 8.8 of \cite{IK}
\begin{align*}
\sum_{n} \psi_k(n)n^{-s} = \frac{\hat{\psi}_k(1-s)}{1-s} + \mathcal{O}(N^{-\sigma + \mathcal{O}(\delta)}),
\end{align*}
where integration by parts yields
\begin{align*}
\left |\frac{\hat{\psi}_k(1-s)}{1-s} \right | = \left |\frac{1}{(1+s)(1-s)}\int \psi_k''(z) z^{s+1} \, dz \right | \, \ll \, |t|^{-2}N^{\sigma + \mathcal{O}(\delta)}.
\end{align*}
For $|t| \, > e^k/100$ we apply Lemma \ref{approx} to each $\psi_k,$ and recombine the functions $\psi_k (|t|/(2 \pi n))$ to get the sum over $\phi_N(|t|/(2 \pi n)).$
\end{proof}

We also require the following bound for exceptionally smooth numbers (for the proof, see Theorem 1 in Chapter III.5 of \cite{Ten}, for instance):
\begin{lemma} \label{hrdd} For $2 \leq Z \leq X$ we have
\begin{align*}
\sum_{n \sim X, \, P^+(n) < Z} 1 \, \ll \, X e^{-u/2},
\end{align*}
where $u:= \log X/\log Z.$
\end{lemma}
 
We will make use of the following result of Shiu \cite{Shiu}. Note that most of the cases where we apply Shiu's bound could also be handled by direct computations, not unlike some which we will have to carry out (cf. proof of Lemma \ref{trivial} below for instance); we use the more general result to sidestep these calculations whenever possible.
\begin{lemma} \label{shiu} Let $\eta > 0,$ and let $g$ be a non-negative multiplicative function such that there exists a constant $C >0$ such that
\begin{align*}
g(p^k) \, \leq C^k, \quad \text{and} \quad g(n) \, \ll_{\epsilon} \, n^{\epsilon}, \, \forall \epsilon >0
\end{align*} 
Then, for $X^{\eta} \ll  Y  \ll X$,
\begin{align*}
\sum_{ X -Y \leq n \leq X} g(n) \, \ll \, Y \prod_{p \leq X} \left( 1+\frac{g(p) -1}{p} \right).
\end{align*}
\end{lemma}

\section{Type I/II estimate} \label{typei}
In this section we prove our Type I/II estimate. Before this we  briefly discuss the strategy used in Chapter 5 of Harman's book \cite{Har} (used in the case of intervals $[x,x+x^{1/2+\epsilon}]$). There one considers Type I/II sums of the form
\begin{align*}
\sum_{\substack{uv nq_1\cdots q_K  \in [x,x+x^{1/2+\epsilon}] \\
u \sim U, \, v \sim V, \, q_i \sim Q }} a_u b_v,
\end{align*}
where $V \leq x^{1/2-\gamma},$ $U \leq x^{1/4},$ and $Q=x^{\gamma/K}$ (the condition $U \leq x^{1/4}$  can actually be loosenend to $VU^2 \leq x^{1-\gamma},$ which is the form given in Harman's book, but the estimate is needed only for $U \leq x^{1/4},$ cf. Lemma 5.3 in Chapter 5 of \cite{Har}).  By applying Perron's formula, these sums have an asymptotic formula if we can obtain a mean value estimate of the form
\begin{align*}
\int_{1/2+iN}^{1/2+ix^{1/2-\epsilon}} |N(s)A(s)B(s)Q(s)^K| |ds| \, \ll x^{1/2}\log^{-C}x,
\end{align*}
where for $N \asymp x^{1-\gamma}/(UV)\geq x^{1/4}$
\begin{align*}
N(s)= \sum_{n \sim N} n^{-s}, \quad A(s)= \sum_{u\sim U}a_u u^{-s}, \quad B(s)=\sum_{v\sim V} b_v v^{-s}, \quad Q(s)=\sum_{q\sim Q, \, q \in \mathbb{P}} q^{-s}.
\end{align*}
By applying the approximate functional equation of $\zeta(s)$ (a variant of Lemma \ref{approx}), the polynomial $N(s)$ can essentially be replaced by $N_t(s)=\sum_{n \sim |t|/(2 \pi N)} n^{-s}.$ By applying Perron's formula to remove the cross-condition between $n$ and $|t|,$ this can be replaced by 
\begin{align*}
M(s)=\sum_{n \leq x^{1/2-\epsilon}/N} c_n n^{-s}.
\end{align*}  
We now note that $Ux^{1/2-\epsilon}/N \ll x^{1/2-\epsilon},$ and $VQ^K \ll x^{1/2}.$ Hence, by Cauchy-Schwarz and Lemma \ref{mvt1}
\begin{align*}
\int_{1/2+iN}^{1/2+ix^{1/2-\epsilon}} & |M(s)A(s)B(s)Q(s)^K| |ds| \,\\
& \ll \left( \int_{1/2+iN}^{1/2+ix^{1/2-\epsilon}} |B(s)Q(s)^K|^2 |ds|\right)^{1/2}\left( \int_{1/2+iN}^{1/2+ix^{1/2-\epsilon}} |M(s)A(s)|^2 |ds|\right)^{1/2} \\
& \ll \left(x^{1/2-\epsilon} + VQ^K \right)^{1/2}\left(x^{1/2-\epsilon}+ Ux^{1/2-\epsilon}/N \right)^{1/2} \log^C x \ll x^{1/2-\epsilon/2}\log^C x \\
&\ll x^{1/2} \log^{-C} x,
\end{align*}
as was required.

In our case we must tread more carefully to avoid losing of powers of $\log x$; for instance, we cannot divide the variables $u$ and $v$ into dyadic ranges $U,V,$ as this would cause the density to drop too much. Instead, we divide the variables into longer ranges $[U^{1-\epsilon}, U]$ and $[V^{1-\epsilon}, V].$ This means that $N(s)$ must also be replaced by a longer polynomial; by choosing $\epsilon$ small enough in terms of $\delta,$ we can replace $N(s)$ by $\sum_{n} \phi_N(n)n^{-s},$ where $\phi_N(n)$ is as in (\ref{phi}). Notice that in Harman's argument, applying Perron's formula to remove the cross-condition $n\sim |t|/(2\pi N)$ causes a loss of size $\log x.$ By using the smooth function $\phi_N(n)$ and Lemma \ref{smooth}, this cross-condition is replaced by a smoothed cross-condition $\phi_N(|t|/(2 \pi n)),$ which can be removed with losses bounded by $\lesssim 1.$ 

Having the variables $u,v$ in longer ranges instead of dyadic means that the condition $U \leq x^{1/4}$ must be strengthened slightly to $U \leq x^{1/4-10\delta}.$ This also allows us to weaken the condition $V < x^{1/2-\gamma}$ to  $V < x^{1/2-\gamma+\delta};$ this will be important as the Type II information we can obtain in the next section  covers only coefficients supported on $[x^{1/2-\gamma+\delta/2}, x^{1/2-\delta/2}]$ instead of the full range $[x^{1/2-\gamma}, x^{1/2}].$ Precisely, our Type I/II estimate takes the form

\begin{prop} \emph{\textbf{(Type I/II estimate).}} \label{typeiestimate} Let $1 \leq U\leq x^{1/4-10\delta}$ and $1 \leq V\leq x^{1/2-\gamma+\delta}$. Let $a_u, b_v$ be complex coefficients satisfying $|a_u| \lesssim 1_{(u,P(w))=1},$ $|b_v|   \lesssim 1_{(v,P(w))=1}$. Then
\begin{align*}
\left | \frac{1}{y}\sum_{\substack{u,v,n \\ u \in [U^{1-\epsilon},U], \, v \in [V^{1-\epsilon},V]\\ (n,P(w))=1 }} W_{\mathcal{A}}(uvn) a_ub_v - \frac{\log^{100}x}{x} \sum_{\substack{u,v,n \\ u \in [U^{1-\epsilon},U], \, v \in [V^{1-\epsilon},V] \\ (n,P(w))=1 }} W_{\mathcal{B}}(uvn)a_ub_v \right|  \\ \ll (\log^{-\delta} x) \frac{\Omega}{\log x}, \hspace{-15pt}
\end{align*}
where $\Omega$ is as in (\ref{omega}).
\end{prop}

The reason we study sums with the additional condition $(n,P(w))=1$ (instead of $n$ smooth) is that we require the coefficients $a_u,b_v$ in the Type I/II and Type II sums to be supported on $w$-almost-primes (cf. beginning of Section \ref{funprop}). However, recall that in the weight $W_\A$ we sum over a $w$-smooth variable $c.$ This means that we can obtain a long smooth variable by writing $m=cn$ (cf. proof of Lemma \ref{implies} below).

With this in mind, define for $\CC=\A$ or $\CC=\B$ the modified weights (without $c$)
\begin{align*}
\widetilde{W}_{\mathcal{C}}(n) &:= \sum_{\substack{nc'q_1q_2q_3 p_1 \cdots p_J r_1 \cdots r_L \asymp x }} f_{\mathcal{C}}(nc' q_1q_2q_3p_1 \cdots p_J r_1 \cdots r_L/x),
\end{align*}
where the summation runs over the same ranges as before (cf. (\ref{summation})).  Then Lemma \ref{implies} below reduces Proposition \ref{typeiestimate} to the following
\begin{prop} \label{typeiestimate2} Suppose that the assumptions of Proposition \ref{typeiestimate} hold. Then
\begin{align*}
\left | \frac{1}{y}\sum_{\substack{u,v,n \\ u \in [U^{1-\epsilon},U], \, v \in [V^{1-\epsilon},V] }} \widetilde{W}_{\mathcal{A}}(uvn) a_ub_v - \frac{\log^{100}x}{x} \sum_{\substack{u,v,n \\ u \in [U^{1-\epsilon},U], \, v \in [V^{1-\epsilon},V] }} \widetilde{W}_{\mathcal{B}}(uvn)a_ub_v \right|  \\ \ll (\log^{-\delta} x) \frac{\Omega}{\log x},
\end{align*}
where $\Omega$ is as in (\ref{omega}).
\end{prop}
\begin{lemma} \label{implies} Proposition \ref{typeiestimate2} implies Proposition \ref{typeiestimate}.
\end{lemma}
\begin{proof}
Consider the sum over $\A$ first. Recall that in the weight $W_\A$ we sum over $c \in [1,x^{\epsilon}]$ with $P^+(c) \leq w.$ Hence, by combining the variables $n$ and $c$ we get
\begin{align*}
\sum_{\substack{u,v,n \\ u \in [U^{1-\epsilon},U] \\ v \in [V^{1-\epsilon},V] \\ (n,P(w))=1 }} W_{\mathcal{A}}(uvn) a_ub_v =   \sum_{\substack{uvncc'q_1q_2q_3 p_1 \cdots p_J r_1 \cdots r_L \asymp x  \\ u \in [U^{1-\epsilon},U] \\ v \in [V^{1-\epsilon},V] \\ (n,P(w))=1}} a_ub_v f_{\mathcal{A}}(uvncc'q_1q_2q_3 p_1 \cdots p_J r_1 \cdots r_L/x) \\
= \sum_{\substack{u,v,n \\ u \in [U^{1-\epsilon},U] \\ v \in [V^{1-\epsilon},V] }} \widetilde{W}_{\mathcal{A}}(uvn) a_ub_v - \sum_{\substack{uvncc'q_1q_2q_3 p_1 \cdots p_J r_1 \cdots r_L \asymp x  \\ u \in [U^{1-\epsilon},U] \\ v \in [V^{1-\epsilon},V] \\ (n,P(w))=1 \\ c >x^{\epsilon}, \,P^+(c) \leq w } } \hspace{-20pt} a_ub_v f_{\mathcal{A}}(uvncc'q_1q_2q_3 p_1 \cdots p_J r_1 \cdots r_L/x)
\end{align*}
Let $\tau^{(k)}$ denote the $k$-fold divisor function (i.e. $\tau^{(k+1)}=1 \ast \tau^{(k)},$ $\tau^{(1)} \equiv 1$). Then (using the disjointness of the intervals $I_j,$ and combining the variables $u,$ $v,$ $n,$ $c',$ $q_1q_2q_3,$ $p_1 \cdots p_J,$ and $r_1 \cdots r_L$)
\begin{align}
&\left | \frac{1}{y}\sum_{\substack{u,v,n \\ u \in [U^{1-\epsilon},U], \, v \in [V^{1-\epsilon},V] }} \widetilde{W}_{\mathcal{A}}(uvn) a_ub_v - \frac{1}{y}\sum_{\substack{u,v,n \\ u \in [U^{1-\epsilon},U], \, v \in [V^{1-\epsilon},V]\\ (n,P(w))=1 }} W_{\mathcal{A}}(uvn) a_ub_v \right| \nonumber \\ \label{clarge}
&\hspace{80pt} \lesssim  \frac{1}{y}\sum_{\substack{uvncc'q_1q_2q_3p_1 \cdots p_J r_1 \cdots r_L \in [x-2y,x+2y] \\ c > x^{\epsilon}, \, P^+(c) \leq w }} 1 \, \leq    \frac{L!}{y}\sum_{\substack{cn \in [x-2y,x+2y] \\ c > x^{\epsilon} \, P^+(c) \leq w}} \tau^{(7)}(n).
\end{align}
Since $c$ in (\ref{clarge}) is $w$-smooth, it must have a divisor $d|c$ such that $x^{\epsilon/2} < d < x^{\epsilon}$ and $P^+(d) \leq w$ (since $c$ has a divisor on every interval of the form $[z,wz] \subset [1,c]$) Hence, (\ref{clarge}) is bounded by
\begin{align*}
 & \ll \, \frac{1}{y} \sum_{\substack{dcn \in [x-2y,x+2y]  \\ x^{\epsilon/2} < d < x^{\epsilon}, \, P^+(d) \leq w}} \tau^{(7)}(n) \,= \, \sum_{\substack{x^{\epsilon/2} < d < x^{\epsilon} \\ P^+(d) \leq w}} \sum_{n \in [(x-2y)/d,(x+2y)/d]} \tau^{(8)}(n) \\
& \, \ll (\log^C x) \sum_{\substack{x^{\epsilon/2} < d < x^{\epsilon} \\ P^+(d) \leq w}} d^{-1} \, \ll \, \log^{-C} x
\end{align*}
by Lemmata \ref{shiu} and  \ref{hrdd}. A similar argument yields
\begin{align*}
\left |  \frac{\log^{100}x}{x}\sum_{\substack{u,v,n \\ u \in [U^{1-\epsilon},U], \, v \in [V^{1-\epsilon},V] }} \widetilde{W}_{\mathcal{B}}(uvn) a_ub_v -  \frac{\log^{100}x}{x}\sum_{\substack{u,v,n \\ u \in [U^{1-\epsilon},U], \, v \in [V^{1-\epsilon},V]\\ (n,P(w))=1 }} W_{\mathcal{B}}(uvn) a_ub_v \right| \\ \lesssim \log^{-C} x.
\end{align*}
\end{proof}

Thus, it remains to prove Proposition \ref{typeiestimate2}.

\emph{Proof of Proposition \ref{typeiestimate2}.}
Let  $N$ be such that $N^{1+\delta/2}UV=x^{1-\gamma}$ (recall (\ref{hash})), and let $\phi_N(x)$ be as in (\ref{phi}). Then
\begin{align} \label{suma}
 \frac{1}{y}\sum_{\substack{u,v,n \\ u \in [U^{1-\epsilon},U], \, v \in [V^{1-\epsilon},V]}} \widetilde{W}_{\mathcal{A}}(uvn) a_ub_v  =  \frac{1}{y}\sum_{\substack{u,v,n  \\ u \in [U^{1-\epsilon},U], \, v \in [V^{1-\epsilon},V]}} \widetilde{W}_{\mathcal{A}}(uvn) a_ub_v  \phi_N(n),
\end{align}
if $\epsilon$ is small enough compared to $\delta.$ 

If $N\geq x^{1/2},$ then the claim is trivial, since in that case $UV \leq x^{1/2-\gamma}N^{-\delta/2}$ and (using (\ref{hash}) and the short-hand notation $\mathfrak{m}=uvc'q_1q_2q_3p_1 \cdots p_J r_1 \cdots r_L$)
\begin{align*}
 \frac{1}{y}\sum_{\substack{u,v,n \\ u \in [U^{1-\epsilon},U], \, v \in [V^{1-\epsilon},V]}} \widetilde{W}_{\mathcal{A}}(uvn) a_ub_v =  \frac{1}{y}\sum_{\substack{uvc'q_1q_2q_3p_1 \cdots p_J r_1 \cdots r_L \leq x^{1/2-\epsilon} \\ u \in [U^{1-\epsilon},U], \, v \in [V^{1-\epsilon},V]}}a_ub_v  \sum_{n} f_{\A}(n\mathfrak{m}/x),
\end{align*}
so that the smooth variable $n$ runs over an interval of length $y/\mathfrak{m} > x^{\epsilon}$ and we get an asymptotic formula. 

Hence, we may assume $N < x^{1/2}.$ If we write $S_{\mathcal{A}}$ for the quantity in (\ref{suma}), we have by Mellin inversion
\begin{align} \label{summellin}
S_{\mathcal{A}} = \frac{1}{y 2\pi i} \int_{1/2-i\infty}^{1/2+i\infty} x^{s} \hat{f}_{\mathcal{A}} (s) N(s)F(s)\, \frac{ds}{s},
\end{align}
where
\begin{align*}
N(s) &= \sum_n \phi_N(n) n^{-s}, \quad F(s)=A(s)B(s)C(s)\Omega(s),
\end{align*}
for
\begin{align*}
 A(s)&= \sum_{u \in [U^{1-\epsilon},U]} a_u u^{-s}, \quad B(s)= \sum_{v \in [V^{1-\epsilon},V]} b_v v^{-s} \quad
\\  C(s)& = \sum_{c \leq D, \, P^+(c) \leq w} c^{-s}, \quad
\Omega(s) = Q_{1}(s)Q_{2}(s)Q_{3}(s) R(s)^L \prod_{j=1}^J P_j(s),
\end{align*}
where the polynomials defining $\Omega(s)$ have the obvious definitions so that $\Omega(s)$ has length around $x^{\gamma + o(1)}$ (cf. (\ref{hash})) and $\Omega(1)C(1)^2=\Omega$ as in (\ref{omega}). Suppose then that $N < x/y$ (if $N \geq x/y,$ a similar but easier argument works). Split the integration in (\ref{summellin}) into three parts (writing $s=1/2+it$)
\begin{align*}
S_{\mathcal{A}}&= \frac{1}{y 2\pi i} \left(\int_{|t| \leq N} + \int_{N < |t| \leq x/y} + \int_{|t| > x/y} \right)x^{s} \hat{f}_{\mathcal{A}} (s) N(s)F(s)\, \frac{ds}{s} \\
& =: \mathcal{I}_1 +\mathcal{I}_2+\mathcal{I}_3,
\end{align*}
say. The main term will be recovered from the first integral, and the two other integrals will be bounded by an argument similar to that which was sketched at the beginning of this section.

\textbf{Integral $\mathcal{I}_1$:} For $|t| \leq N,$ $s=1/2+it,$  we have by Lemma \ref{smooth}
\begin{align*}
N(s)=\sum_n \phi_N(n) n^{-s} = \frac{\hat{\phi}_N(1-s)}{1-s} + \mathcal{O}(N^{-1/2 +  \mathcal{O}(\delta)}),
\end{align*}
where $\hat{\phi}_N$ is the Mellin transform of $\phi_N.$ We also have for $\sigma = 1/2$ by Lemma \ref{mellin}
\begin{align*}
x^s\hat{f}_{\mathcal{A}}(s)/s &= (2\eta_\A+\xi_\A) x^s + \mathcal{O}\left(x^{-1/2 + o(1)}(1+|s|)\right), \quad  \left|\hat{f}_{\mathcal{A}}(s)/s  \right| \ll x^{-1/2 + o(1)},
\end{align*}
where  $\eta_\A = y/x$ and $\xi_\A = (\log^{-\epsilon} x )y/x .$  Thus, $\mathcal{I}_1=  \mathcal{J}_1 +\mathcal{J}_2+\mathcal{J}_3,$ where
\begin{align*} 
\mathcal{J}_1 &=  \frac{2\eta_\A+\xi_\A}{y 2\pi i}  \int_{1/2-iN}^{1/2+iN} x^s  \frac{\hat{\phi}_N(1-s)}{1-s} F(s) \,ds, \\
\mathcal{J}_2 & \ll \frac{1}{y} N^{-1/2+ \mathcal{O}(\delta)} x^{1/2} \int_{1/2-iN}^{1/2+iN} |F(s)| |\hat{f}_{\mathcal{A}}(s)/s | |ds|, \quad \quad \text{and}\\
\mathcal{J}_3 & \ll \frac{1}{y}  \int_{1/2-iN}^{1/2+iN} |F(s)| \left|\frac{\hat{\phi}_N(1-s)}{1-s} \right| x^{-1/2+o(1)}(1+|s|) |ds|.
\end{align*}
If  we denote $F(s)= \sum_{m} c_m m^{-s},$ then by combining variables (using disjointness of the intervals $I_j$) we obtain $|c_m| \, \lesssim \tau^{(5)}(m),$ so that $\sum_{m} |c_m|^2 m^{-1} \ll \log^C x $. For $\mathcal{J}_3 $ note that by Lemma \ref{smooth} $|\hat{\phi}_N(1/2+it)| \, \ll N^{1/2+\mathcal{O}(\delta)}.$  Thus, by Cauchy-Schwarz and Lemma \ref{mvt1} 
\begin{align*}
| \mathcal{J}_2 |& \ll y^{-1}N^{-1/2+  \mathcal{O}(\delta)} x^{1/2}  \int_{1/2-iN}^{1/2+iN} |F(s)| |\hat{f}_{\mathcal{A}}(s)/s | |ds| \\ 
& \ll y^{-1}x^{ \mathcal{O}(\delta)}  \left(\int_{1/2-iN}^{1/2+iN} |F(s)|^2|ds| \right)^{1/2} 
\ll y^{-1} x^{ \mathcal{O}(\delta)}(N + x^{1+ \mathcal{O}(\delta)}/N)^{1/2} \ll x^{-\epsilon}
\end{align*}
and
\begin{align*}
| \mathcal{J}_3 |& \ll y^{-1} \int_{1/2-iN}^{1/2+iN} |F(s)| \left|\frac{\hat{\phi}_N(1-s)}{1-s} \right| x^{-1/2+o(1)}(1+|s|) |ds| \\
& \ll x^{-1+o(1)} N^{1/2 +\mathcal{O}(\delta) } \int_{1/2-iN}^{1/2+iN} |F(s)| |ds|  \, \ll x^{-1+\mathcal{O}(\delta)} N \left(\int_{1/2-iN}^{1/2+iN} |F(s)|^2|ds| \right)^{1/2} \\
& \ll  x^{-1+\mathcal{O}(\delta)} N (N + x^{1+ \mathcal{O}(\delta)}/N)^{1/2} \ll x^{-\epsilon},
\end{align*}
since by the assumptions on $U$ and $V$ we have $x^{1/4} \leq x^{1/2-2\delta}/U \leq N \leq x^{1/2}.$ For the main term we obtain by the change of variables $s \mapsto 1-s$ and Mellin inversion applied to $\phi_N$
\begin{align}
\mathcal{J}_1 &=  (2+\log^{-\epsilon} x) \frac{1}{2\pi i} \int_{1/2-iN}^{1/2+iN} x^{-s}  \hat{\phi}_N(s) F(1-s) \,\frac{ds}{s} \nonumber \\ \label{mainterm}
&=(2+\log^{-\epsilon} x) \sum_{\substack{u,v,c',q_1,q_2,q_3, \\ p_1, \dots p_J, r_1,\dots,r_L \\ u \in (U^{1-\epsilon}, U], \, \, v \in (V^{1-\epsilon},V]}} \frac{a_ub_v \phi_N(x/(uvc'q_1q_2q_3 p_1 \cdots p_J r_1 \cdots r_L))}{uvc'q_1q_2q_3 p_1 \cdots p_J r_1 \cdots r_L} + \mathcal{O}(E),
\end{align}
where the error term is
\begin{align*}
E = x^{-1/2} \int_{1/2+iN}^{1/2 + i\infty} |F(1-s)| |\hat{\phi}_N(s)|  \,\frac{|ds|}{|s|}. 
\end{align*}
The first term  in (\ref{mainterm}) can be evaluated asymptotically as the sum over $\mathcal{B}$ in the proposition, since $\phi_N \equiv 1$ in the sum, and $ \int f_\B (z) \, dz = (2+\log^{-\epsilon} x)\log^{-100}x .$  By Lemma \ref{smooth} we have $|\hat{\phi}_N(s)| \ll N^{\sigma + \mathcal{O}(\delta)} |t|^{-1}.$ We also have the trivial bound $|F(1/2 +it)| \ll (x/N)^{1/2 +\mathcal{O}(\delta)}.$ Thus,
\begin{align*}
E \ll x^{ \mathcal{O}(\delta)} \int_N^{\infty} \frac{dt}{t^2} \ll x^{\mathcal{O}(\delta)}/N \ll x^{-\epsilon},
\end{align*}
so that  from the integral $\mathcal{I}_1$ we obtain the main term with sufficient bounds for the error terms.

\textbf{Integral $\mathcal{I}_2$:}
We have $ |\hat{f}_{\mathcal{A}}(s)/s  | \ll y/x$ by Lemma \ref{mellin}. Thus,
\begin{align*}
\mathcal{I}_2 \ll x^{-1/2}  \int_{1/2+iN}^{1/2 + ix/y} |N(s)F(s)| |ds|.
\end{align*}
 Lemma \ref{smooth} yields 
\begin{align*}
N(s) = \chi(s)N_t(1-s) +  E(N,|t|),
\end{align*}
for 
\begin{align*}
N_t(1-s) &:= \sum_{n} \phi_N\left(\frac{|t|}{2\pi n}  \right)n^{s-1}, \quad  \quad \quad |\chi(1/2+it)|=1, \quad \quad \text{and}\\  
|E(N,|t|)| \, &\ll N^{-1/2 + \mathcal{O}(\delta)}1_{|t| \leq N^{1+\delta}}+ |t|^{-2}N^{1/2+\mathcal{O}(\delta)} +|t|^{-3/4  + o(1)}.
\end{align*}
We have
\begin{align*}
x^{-1/2}& \int_{1/2+iN}^{1/2 + ix/y} |F(s)E(N,|t|)| |ds| \\ & \hspace{-20pt}\ll \, x^{-1/2} N^{-1/2+\mathcal{O}(\delta)} \int_{1/2+iN}^{1/2 + iN^{1+\delta}} \hspace{-8pt}|F(s)| |ds|  + x^{-1/2}N^{1/2+\mathcal{O}(\delta)}\int_{1/2+iN}^{1/2 + ix/y} \frac{|F(s)| |ds|}{t^2}  \\&\hspace{150pt}+ x^{-1/2} \int_{1/2+iN}^{1/2 + ix/y} \frac{|F(s)| |ds| }{t^{3/4+o(1)}},  \\
& \ll x^{-\epsilon},
\end{align*}
where the last bound follows from applying Cauchy-Schwarz and Lemma 3. Thus,
\begin{align} \label{int2}
\mathcal{I}_2 \ll x^{-1/2}  \int_{1/2+iN}^{1/2 + ix/y} |N_t(s)||F(s)| |ds|\,  +\, x^{-\epsilon},
\end{align}
We now remove the cross-condition between $n$ and $t$: by Mellin inversion, we have 
\begin{align*}
\phi_N\left(\frac{|t|}{2\pi n}  \right )= \frac{1}{2\pi i} \int_{-\infty}^{\infty} \hat{\phi}_N(iz) (2\pi)^{iz} |t|^{-iz} n^{iz} \frac{dz}{z}.
\end{align*}
Hence, by the second part of Lemma \ref{smooth} the integral in (\ref{int2}) bounded by
\begin{align*} 
\int_{-\infty}^{\infty} | \hat{\phi}_N(iz)| \int_{1/2+iN}^{1/2 + ix/y} |M_z(s)||F(s)| |ds|\,  \frac{dz}{z} \lesssim \sup_{z \in \R} \int_{1/2+iN}^{1/2 + ix/y} |M_z(s)||F(s)||ds|,
\end{align*}
where
\begin{align*}
M_z(s) = \sum_{n \leq xN^{-1}y^{-1}} n^{iz} n^{-s}.
\end{align*}
Fix a $z$ such that the integral 
\begin{align*}
\int_{1/2+iN}^{1/2 + ix/y} |M_z(s)||F(s)||ds|
\end{align*}
is at least half of the supremum over all $z$, and write $M_z(s)=M(s).$ 

Recall now that $V\leq x^{1/2-\gamma+ \delta},$ $U \leq x^{1/4-10\delta},$ and by the definition of $N$
\begin{align} \label{usize}
U < \frac{x^{1-\gamma-10\delta}}{VU} = N^{1+\delta/2} x^{-10\delta}.
\end{align}
We now factor $\Omega(s)$ into a product $\Omega_1(s) \Omega_2(s)$ suitably: recall the definition of the intervals $I_j,$ and let $J_1 \leq J$ be the largest index such that $\omega^{r^{-1} + \cdots + r^{-J_1}} \leq x^{\gamma-2\delta},$ so that $\omega^{r^{-1} + \cdots + r^{-J_1-1}} > x^{\gamma-2\delta}.$ Then by (\ref{hash}) we have (since $r<2$ and $J_1 \ll_{\delta} 1$) 
\begin{align*}
\omega^{r^{-J_1-1}+ \cdots + r^{-J}} < \omega^{r^{-J_1-1}(1-\delta)/(r-1)}  \omega^{r^{-J_1-2}+\cdots +r^{-J}}  < \left( \omega^{r^{-J_1-2}+\cdots +r^{-J}}\right)^2 <  x^{6 \delta}.
\end{align*}
 Therefore, if $K$ large enough in terms of $\delta$, we have
 \begin{align*}
\omega^{r^{-1} + \cdots + r^{-J_1}}  \omega^{r^{-K-1} + \cdots + r^{-J}} \leq x^{\gamma-2\delta + o(\delta)}, \quad \quad \omega^{r^{-J_1-1} + \cdots + r^{-K}} < x^{6\delta}.
 \end{align*}
Hence, if we define
\begin{align*}
\Omega_1(s) &:= Q_1(s)Q_2(s)Q_3(s)R(s)^L \prod_{j=1}^{J_1} P_j(s)\prod_{j=K+1}^{J} P_j(s), \quad \quad \Omega_2(s) := \prod_{j=J_1+1}^{K} P_j(s) \\
 F_1(s) &:= B(s)C(s)\Omega_1(s), \quad \quad F_2(s):= A(s)\Omega_2(s),
\end{align*} 
then the length of $F_1(s)$ is less than $Vx^{o(1)} x^{\gamma-2\delta + o(\delta)} < x/y,$ and  the length of $F_2(s)$ is less than $Ux^{6\delta} < Nx^{-\delta}$ by (\ref{usize}), so that the length of $M(s)F_2(s)$ is  less than $x/y.$ Thus, $\mathcal{I}_2 \lesssim E  + x^{-\epsilon},$ where 
\begin{align*}
E &:=  x^{-1/2}\int_{1/2+iN}^{1/2 + ix/y} |M(s)F_{1}(s)F_2(s)| |ds|
\end{align*}
Applying Cauchy-Schwarz  and Lemma \ref{mvt1} we obtain
\begin{align*}
E \leq x^{-1/2} \left(\int_{1/2+iN}^{1/2 + ix/y} |F_{1}(s)|^2 |ds|\right)^{1/2} \left(\int_{1/2+iN}^{1/2 + ix/y} |M(s)F_2(s)|^2 |ds|\right)^{1/2} \\
\ll \frac{x^{1/2}}{y} \left(\sum_k |h_{1}(k)|^2k^{-1}\right)^{1/2}\left(\sum_k |h_2(k)|^2k^{-1}\right)^{1/2},
\end{align*}
where
\begin{align*}
h_{1}(k) = \sum_{k=vcq_1q_2q_3p_1\cdots p_Jr_1\cdots r_L} b_v , \quad \quad \quad h_2(n) =  \sum_{k=un}a_u n^{iz}.
\end{align*}
To simplify the notations, we have in the above written all of the primes coming from $\Omega(s)$ into the first term (if $\Omega(s) \neq \Omega_1(s),$ a similar argument as below works).  Then, since $(u,P(w))=1$ always, we have $|h_2(k)| \lesssim \tau_w(k).$ Thus, by Lemma \ref{shiu}
\begin{align*}
\sum_k |h_2(k)|^2k^{-1} \lesssim (\log x)\prod_{w < p \leq x}\left(1+\frac{\tau_w(p)^2-1}{p} \right) \, \lesssim \, \log x.
\end{align*}

For the sum over $h_{1},$  recall that the intervals $I_j$ are disjoint. On average an integer $n$ has $\log(1/\theta)$ prime factors from any given interval $I_j,$ so that the collisions between a smooth variable  and the primes $p_j$ are expected to contribute a factor  $(1+\log 1/\theta)^J.$ This is now made rigorous: $c$ is $w$-smooth, and $(v,P(w))=1$, so that by combining the variables $vc=n$ we get
\begin{align*}
|h_{1}(k)| \, \lesssim \sum_{k=nq_1q_2q_3p_1\cdots p_Jr_1\cdots r_L}  1.
\end{align*}
 Thus, for any $M < x^{1/2},$   we have to give a bound for (combining $n$ with $q_i,r_l$)
\begin{align} \label{colli}
\sum_{k\sim M} \left(\sum_{k= nq_1q_2q_3p_1\cdots p_Jr_1\cdots r_L} 1 \right)^2 \lesssim \sum_{np_1\cdots p_J \sim M} \sum_{np_1\cdots p_J= n'p'_1\cdots p'_J} g(n)g(n')   ,
\end{align}
where 
\begin{align*}
g(n):= 1+\sum_{n=mq_1q_2q_3r_1\cdots r_L} 1.
\end{align*}
 In \ref{colli} we have  $g(n') \lesssim g(n),$ since $g(n)-1$ counts the number of factors $q_1q_2q_3r_1\cdots r_L|n,$ and there exists only a bounded number of indices $j$ such that $p_j$ can be in the same range as one of the primes $q_1,q_2,q_3,r_l.$ Recall that the intervals $I_j$ are disjoint. We split the sum (\ref{colli}) into a sum over subsets $\rho \subseteq \{1,2,\dots,J\},$ where $\rho$ is the set of indices for which $p_j \neq p'_j.$ Note that $p_j \neq p'_j$ implies that $p'_j|n.$  Thus, (\ref{colli}) is bounded by
\begin{align} \nonumber
  \sum_{p_1,p_2,\dots, p_J } \sum_{\rho \subseteq \{1,2,\dots,J\}} \sum_{\substack{p'_j, \, j \in \rho \\ p'_j \neq p_j}}  \, &\sum_{\substack{n\sim M/(p_1\cdots p_J), \\ p'_j | n \, \forall \, j \in \rho}}  g(n)^2   \, \\ \label{colli2} & \lesssim  \sum_{p_1,p_2,\dots, p_J } \sum_{\rho \subseteq \{1,2,\dots,J\}}\,\sum_{\substack{p'_j, \, j \in \rho \\ p'_j \neq p_j}}  \sum_{\substack{n\sim M/(p_1\cdots p_J\prod_{j \in \rho}p'_j)}} g(n)^2,
\end{align}
since $g(n) \lesssim g(n/(\prod_{j \in \rho}p'_j)) .$ We have $g(n) \leq L! \tilde{g}(n),$ where $\tilde{g}(n)$ is the multiplicative function defined by
\begin{align*}
\tilde{g}(p^k) :=  \begin{cases}  (k+1), & p \sim \log^{\epsilon} x, \, p \sim Q_1,   \, \,  p \sim Q_2, \,\, \text{or} \,\, \,  Q_3^{1/2} \leq p \leq Q_3 \\
 1, &\text{otherwise},  \end{cases}.
\end{align*}
 Thus, by Lemma \ref{shiu} the sum (\ref{colli2}) is  bounded by
\begin{align*}
 & \sum_{p_1,p_2,\dots, p_J } \sum_{\rho \subseteq \{1,2,\dots,J\}} \sum_{p'_j, \, j \in \rho} \, \sum_{\substack{n\sim M/(p_1\cdots p_J\prod_{j \in \rho}p'_j)}} \tilde{g}(n)^2  \\ & \hspace{20pt}\lesssim M \prod_{p \leq x} \left( 1+\frac{\tilde{g}(p)^2-1}{p}\right)  \sum_{p_1,p_2,\dots, p_J } \sum_{\rho \subseteq \{1,2,\dots,J\}} \sum_{p'_j, \, j \in \rho} (p_1\cdots p_J)^{-1} \left( \prod_{p'_j, \, j \in \rho}  p'_j\right)^{-1} \\
& \hspace{20pt}\lesssim \, M \prod_{j=1}^J \left( \left( \sum_{p \in I_j}p^{-1}\right) \left(1+ \sum_{q \in I_j} q^{-1}\right)\right) \, \lesssim \, M (\log 1/\theta)^J(1+ \log 1/\theta)^J.
\end{align*}
Thus, the sum over $|h_{1}(k)|^2k^{-1}$ is bounded by $(\log x)(\log 1/\theta)^J(1+ \log 1/\theta)^J.$ 

 Combining the two above estimates we obtain
\begin{align*}
\mathcal{I}_2 \lesssim \frac{x^{1/2}}{y} (\log x) (\log 1/\theta)^{J/2}(1+ \log 1/\theta)^{J/2} = \log^{-0.7... +o(1)} x,
\end{align*}
while
\begin{align*}
\frac{\Omega}{\log x} = (\log^{1+o(1)} x)(\log 1/\theta)^J = \log^{-0.5...+o(1)} x,
\end{align*}
so that the estimate is sufficient for the proposition.

\textbf{Integral $\mathcal{I}_3$:} We still have to estimate the integral over $|t| \geq x/y.$ By Lemma \ref{mellin} we have the bound
\begin{align*}
| \hat{f}_{\mathcal{A}}(s)/s| \ll \xi^{-1} |s|^{-1}|1+s|^{-1} \, \lesssim x/y |s|^{-1}|1+s|^{-1}.
\end{align*}
Hence, by  a dyadic decomposition of the integral we obtain
\begin{align*}
\mathcal{I}_3 &\lesssim \frac{x^{1/2}}{y} \cdot \frac{x}{y} \int_{x/y}^{\infty} |N(1/2+it)||F(1/2+it)|\frac{dt}{t^2} \\
&\ll x^{-1/2} \cdot \frac{x}{y} \max_{T > x/y} \frac{1}{T} \int_T^{2T} |N(1/2+it)||F(1/2+it)| \,dt.
\end{align*}
Now a similar argument as for the integral over $\mathcal{I}_2$ gives the bound.
\qed

\subsection{Discussion of Type I estimates}
We note that in the above proof of Proposition \ref{typeiestimate2} we needed to factorize the polynomial $\Omega(s)$ only with an accuracy of $x^{\epsilon}.$ This can be done using just finitely many primes $p_j.$ Then we would be spared of the losses coming from the density. This means that the Type I/II estimate can be made to work even for the much shorter intervals, at least of type $[x,x+x^{1/2} (\log \log x)^B]$ for some constant $B$. We will need a stronger factorization property for the Type II sums below, which means that we require the number of primes $J \gg \log \log x,$ causing a loss of some power of $\log x.$ 

Unfortunately in our case we cannot make use of the more advanced mean value estimates such as Watt's Theorem or the Deshouillers-Iwaniec Theorem (cf. Lemma 3 of \cite{JL} for example). The reason for this is the $T^{o(1)}$ term in these estimates. For longer intervals $y=x^{1/2+\epsilon},$ we have that the critical range in the mean values is $T= x/y = x^{1/2-\epsilon},$ so that the $x^{-\epsilon}$ is sufficient to cancel $T^{o(1)}.$ A possible way one might try to implement these estimate is to try to `boost' these estimates by dividing the integration into $\mathcal{T} \cup \mathcal{U},$ where $\mathcal{T}$ is a range where some polynomial is small and $\mathcal{U}$ is the complement, and then use eg. Watt's Theorem only for the integral over  $\mathcal{U}.$ See for example Proposition 8 of Ter\"av\"ainen \cite{Ter} for such an argument. The difference compared to our case is that we do not have a square mean value of a Dirichlet polynomial to begin with, so that the same argument is not applicable.

In Chapter 5 of Harman's book \cite{Har}, one has also the so-called `two dimensional' Type I$_2$ estimate. In our case this corresponds to a sum of the form
\begin{align*}
\sum_{\substack{v,n,m \\(nm,P(w))=1}} b_v W_\A(vnm).
\end{align*}
It appears that a similar argument as with the Type I/II estimate works here also (at least for $v \leq x^{1/2-\gamma-\delta}$); by symmetry we may assume that $n > m,$ and then apply the argument with $u=m,$ $a_u= 1_{(u,P(w))=1}.$ Some complications occur when we want use this result (combined with the Type II estimate of the next section) to obtain a two dimensional version of Proposition \ref{ft} (cf. Section \ref{funprop} below). However, these problems are probably not too severe, requiring only some extra care.  In any case, we expect that the improvement in the value of $\gamma$ that this additional estimate would give is not very large (cf. Sum $S_8(\CC)$ in Section \ref{buchs}). We do not pursue this issue further here, since the most important aspect for us is the length of the interval $y$, and having the  Type I$_2$ estimate does not affect this.

\section{Type II estimate} \label{typeii}
Define $\tau_{w}(d) := (1_{(\cdot, P(w))=1} \ast 1)(d).$ Our goal in this section is to obtain 
\begin{prop} \emph{\textbf{(Type II estimate).}} \label{typeiiestimate}Let $a_u, b_v$ be complex coefficients supported for $u,v \in [x^{1/2-\gamma+\delta/2},x^{1/2-\delta/2}]$, such that $|a_u| \lesssim \tau_{w}(u)1_{(u,P(w))=1},$ $|b_v|   \lesssim \tau_{w}(v)1_{(v,P(w))=1}.$   Then
\begin{align*}
\left | \frac{1}{y} \sum_{u,v} W_{\mathcal{A}}(uv) a_ub_v  -  \frac{\log^{100} x}{x}\sum_{u,v} W_{\mathcal{B}}(uv)a_ub_v \right | \ll (\log^{-\delta} x ) \,\frac{\Omega}{\log x}
\end{align*}
\end{prop}

By symmetry and (\ref{hash}), we may assume that $a_u$ is supported for $u \leq x^{1/2-\gamma/2+\delta}.$ We will assume this systematically throughout this section.

For the proof we require the following technical lemma, which will be used to obtain a suitable factorization for Dirichlet polynomials. To motivate it, recall from Section 1 the argument described for the Type II sums; to obtain the last the bound (\ref{typebound}) we chose $L \leq K-1$ in such a way that $UQ^L$ and $VQ^{K-L-l}$ are roughly of the same size. In our situation we have primes $p_j$ in various different ranges $I_j$, which means that we need to consider all possible factorizations
\begin{align*}
p_1p_2\cdots p_J = \left(\prod_{j \in \pi} p_j \right) \left( \prod_{j \in \tau} p_j\right),
\end{align*}
 where $\pi \sqcup \tau = \{1,2,\dots,J\}$ is a partition of the indices.  In the lemma below the quantities $e^g,e^{g'}$ correspond to $e$-adic ranges so that $uc \sim e^g,$ and $vc' \sim e^{g'}.$ We also will divide the range for $q_3$ into $e$-adic parts of the form $q_3 \sim e^{\alpha}$ for  $\lfloor \log^{9/10} x \rfloor  \leq \alpha < 2 \lfloor \log^{9/10} x \rfloor.$  Giving the lemma below  in terms of the $e$-adic ranges gives us great flexibility which simplifies the calculations.  For any subset $\rho \subseteq \{1,2,\dots,J\},$ we denote by $(p_j)_{j\in \rho}$ the tuple of primes $p_j$ such that $j \in \rho.$ Recall that for $\theta := r-1+\epsilon$
\begin{align*}
p_{j} \in I_j  = \begin{cases} (\omega^{(1-2\epsilon) r^{-j}}, \omega^{(1-\epsilon)r^{-j}}], \, &j=1,2, \dots, K \\
  (\omega^{\theta r^{-j}}, \omega^{(1-\epsilon)r^{-j}}], \, & j= K+1, \dots, J. \end{cases}
\end{align*} Recall also that $r_1 \cdots r_L \asymp \log^{L\epsilon} x.$ 
\begin{lemma} \emph{\textbf{(Partition algorithm).}} \label{pa} Let $z:=x/\log^{L\epsilon} x,$ and let $\alpha,g,g'$ denote integers such that  $\lfloor \log^{9/10} x \rfloor  \leq \alpha < 2 \lfloor \log^{9/10} x \rfloor$, $e^g \in [x^{1/2-\gamma+\delta/4},x^{1/2-\gamma/2+2\delta}],$ and $e^{g'} \in [x^{1/2-\gamma+\delta/4},x^{1/2-\delta/4}].$ If the primes $p_j \in I_j$ are such that
\begin{align*}
e^{g+g'} Q_1 Q_2 e^{\alpha} p_1\cdots p_J \asymp z,
\end{align*}
then there exists a partition of the indices $\pi \sqcup \tau= \{1,2,\dots,J\}$ so that
\begin{align} \label{partition}
e^{g+\alpha} Q_1Q_2\prod_{j \in \pi} p_j \asymp z^{1/2}, \quad \quad e^{g'}\prod_{j \in \tau} p_j \asymp z^{1/2}.
\end{align}

Furthermore, there is an algorithm (which we call the partition algorithm)
\begin{align*}
 (g,g',\alpha,p_1,\dots,p_J) \mapsto \pi \sqcup \tau= \{1,2,\dots,J\}
\end{align*}
 such that the following holds:

For any $g,g',\alpha,p_1,\dots,p_J$ as above and any partition $\pi \sqcup \tau= \{1,2,\dots,J\},$ there are intervals $D((p_j)_{j\in \pi}),$ $D((p_j)_{j\in \tau})$ such that if  
\begin{align*}
e^{g+g'} Q_1 Q_2 e^{\alpha} p_1\cdots p_J \asymp z,
\end{align*}
 then
\begin{align} \label{partition2}
 \quad e^{g+\alpha} Q_1Q_2 \in D((p_j)_{j\in \pi}) \quad \text{and} \quad e^{g'} \in D((p_j)_{j\in \tau})
\end{align}
holds if and only if the partition algorithm produces the corresponding partition $\pi \sqcup \tau= \{1,2,\dots,J\}.$ The intervals $D((p_j)_{j\in \pi}),$ $D((p_j)_{j\in \tau})$ are always contained in some $C$-adic intervals around $z^{1/2} /(\prod_{j \in \pi} p_j) $, $z^{1/2} /(\prod_{j \in \tau} p_j) ,$ respectively, so that \emph{(\ref{partition2})} implies  \emph{(\ref{partition})}.
\end{lemma}

\begin{remark} We need the partition algorithm and the intervals $D( (p_j)_{j\in \pi}),$ $D( (p_j)_{j\in \tau})$ so that we know which partition to apply for each product; by uniqueness of the partition that the algorithm produces, for any given $(g,g',\alpha,p_1,\dots,p_J)$ there is exactly one partition such that (\ref{partition2}) holds.  It will be crucial for us that the interval $D((p_j)_{j\in \pi})$ (resp. $D((p_j)_{j\in \tau})$) depends only on those primes $p_j$ such that $j \in \pi$ (resp. $p_j$ such that $j \in \tau$). 

Note that the above lemma does not include the primes $r_l.$ This is because we want to reserve a possibility to skew each partition one way or another by some power of $\log x.$ That is, for any partition $\pi\sqcup \tau= \{1,2,\dots,J\},$ we will later choose some suitable $L(\pi,\tau) \leq L$ and write $r_1 \cdots r_L = (r_1\cdots r_{L(\pi,\tau)})\cdot (r_{L(\pi, \tau)+1}\cdots r_L).$
\end{remark}

\begin{proof}
We first construct a suitable partition by using an iterative algorithm, and afterwards recover the intervals  $D((p_j)_{j\in \pi}),$ $D( (p_j)_{j\in \tau}).$ Let $\lfloor \log^{9/10} x \rfloor  \leq \alpha < 2 \lfloor \log^{9/10} x \rfloor$ and  $e^g,e^{g'} \in [x^{1/2-\gamma+\delta/4},x^{1/2-\delta/4}]$ be given, and let $p_1, \dots, p_J$ be such that
\begin{align*}
 e^{g+g'} Q_1 Q_2 e^{\alpha} p_1\cdots p_J \asymp z.
\end{align*}
Since $e^g \leq x^{1/2 -\gamma/2+2\delta},$ we can choose $1 \in \pi,$ since by definitions for any $p_1 \in I_1$ (using $r <2$)
\begin{align*}
e^{g+\alpha} Q_1Q_2 p_1 \leq  x^{1/2-\gamma/2+\delta +\gamma(1-1/r)} < x^{1/2-\delta} < z^{1/2}.
\end{align*}
Let $j_1 \leq J$ be the smallest index such that
\begin{align*}
e^{g+\alpha} Q_1Q_2 p_1 \cdots p_{j_1} \omega^{\epsilon r^{-j_1}} \leq z^{1/2} < e^{g+\alpha} Q_1Q_2 p_1 \cdots p_{j_1} \omega^{r^{-j_1-1}}.
\end{align*}
There must exist such an index since 
\begin{align*}
e^{g+\alpha} Q_1Q_2 p_1p_2\dots p_{J} > x^{\epsilon}z^{1/2} ,
\end{align*}
and if
\begin{align*}
e^{g+\alpha} Q_1Q_2 p_1p_2\dots p_{j}\omega^{r^{-j-1}} \leq z^{1/2} ,
\end{align*}
we can then multiply by $p_{j+1}$ also; we choose $1,2, \dots, j_1 \in \pi.$  Then there is some large $C=C(\epsilon)$ such that if $j_1 \leq J-C,$ then $e^{g'}p_{j_1+1} < z^{1/2},$ because otherwise (using $\theta= r-1+\epsilon$)
\begin{align*}
z & \gg (e^{g+\alpha} Q_1Q_2p_1\cdots p_{j_1} p_{j_1+2}\cdots p_J)\cdot(e^{g'}p_{j_1+1}) \\
&> e^{g+\alpha} Q_1Q_2p_1\cdots p_{j_1} \omega^{\theta(r^{-j_1-2} + r^{-j_1-3}\dots + r^{-J}) } z^{1/2}\\
& \gg e^{g+\alpha} Q_1Q_2p_1\cdots p_{j_1} \omega^{r^{-j_1-1}} \omega^{\epsilon r^{-j_1-2}  } z^{1/2} >  \omega^{\epsilon r^{-j_1-2} } z,
\end{align*}
which is clearly impossible if $C$ is large enough. If $j_1 >  J-C,$ then $e^g Q(\alpha,\pi)p_1\cdots p_{j_1} \asymp z^{1/2},$ which implies that also  $e^{g'}p_{j_1+1}\cdots p_J \asymp z^{1/2}$ so that we are done. Hence, we can assume $j_1 \leq J-C$ and in that case there exists $j_2>j_1,$ which is the smallest index such that
\begin{align*}
e^{g'}p_{j_1+1}p_{j_1+2}\dots p_{j_2}\omega^{\epsilon r^{-j_2}} \leq z^{1/2} < e^{g'}p_{j_1+1}p_{j_1+2}\dots p_{j_2}\omega^{r^{-j_2-1}}.
\end{align*}
We may now iterate the above process  to get $j_1 < j_2 < \cdots < j_R$ where $j_{R-1} > J-C,$ for some suitably large $C,$  and $j_{R} := J,$ so that (\ref{partition}) is satisfied. 

Remark that since we have included the factor $\omega^{\epsilon r^{-j}}$ in the above, we  have $j_{i+1}-j_i \ll_{\epsilon} 1$ for all $i.$ For example, to see that $j_2-j_1$ is bounded, we just have to note that
\begin{align*}
e^{g'}p_{j_1+1} \cdots p_J \gg z/(e^{g+\alpha} Q_1Q_2 p_1 \cdots p_{j_1}) \geq \omega^{\epsilon r^{-j_1}}z^{1/2}.
\end{align*}

We can now recover the intervals  $D((p_j)_{j\in \pi}),$ $D((p_j)_{j\in \tau})$ from the above procedure.   Fix $(g,g',\alpha,p_1,\dots,p_J)$ and $\pi,\tau$ such that the above algorithm produces the partition $\pi,\tau.$   Denote
\begin{align*}
\pi_0 := \{j \in \pi: \, j+1 \notin \pi\} \quad \text{and} \quad \tau_0 := \{j \in \tau: \, j+1 \notin \tau\}.
\end{align*}
For any subset $\rho \subset \{ 1,2,\dots, J \}$ and  $i \leq J$ define $\Pi(\rho,i):= \prod_{j\in \rho, \, j \leq i} p_j.$ Then  for all $i \in \pi_0$ we obtain from the above algorithm
\begin{align} \label{cc1}
e^{g+\alpha} Q_1Q_2 \Pi(\pi,i) \omega^{\epsilon r^{-i}} \leq z^{1/2} < e^{g+\alpha} Q_1Q_2 \Pi(\pi,i)\omega^{r^{-i-1}},
\end{align}
and if $i -1 \in \pi,$ we also have (since at each stage we chose the smallest $j_i$)
\begin{align} \label{cc2}
e^{g+\alpha} Q_1Q_2 \Pi(\pi,i) p_{i}^{-1}\omega^{r^{-i}}  \leq z^{1/2}.
\end{align}
In this case, the above is stricter than the left-hand side inequality in (\ref{cc1}). If also $i -2 \in \pi,$ then 
\begin{align*}
e^{g+\alpha} Q_1Q_2 \Pi(\pi,i)p_{i}^{-1}p_{i-1}^{-1}\omega^{r^{-i+1}} \leq z^{1/2},
\end{align*}
but this is already implied by the inequality (\ref{cc2}) since (using $1.5 < r < 2$ for the last inequality below)
\begin{align*}
p_{i-1}^{-1}\omega^{r^{-i+1}} \leq \omega^{r^{-i+1}(1-\theta)} < \omega^{r^{-i}(2-r)r} < \omega^{r^{-i}}.
\end{align*}
Thus, the inequalities corresponding to each $i \in \pi_0$ are
\begin{align*}
e^{g+\alpha} Q_1Q_2 \Pi(\pi,i) \omega^{\epsilon r^{-i}} \leq z^{1/2} < e^{g+\alpha} Q_1Q_2 \Pi(\pi,i)\omega^{r^{-i-1}} , \quad \text{if}\quad & i -1 \notin \pi, \\
e^{g+\alpha} Q_1Q_2 \Pi(\pi,i) p_{i}^{-1}\omega^{r^{-i}} \leq z^{1/2} < e^{g+\alpha} Q_1Q_2 \Pi(\pi,i)\omega^{r^{-i-1}}, \quad \text{if}\quad & i -1 \in \pi,
\end{align*}
which can be written in the form $e^{g+\alpha} Q_1Q_2  \in \mathcal{J}(\pi,i),$ where
\begin{align*}
 \mathcal{J}(\pi,i) := \begin{cases} \left( \frac{z^{1/2}}{\Pi(\pi,i)\omega^{r^{-i-1}} }, \frac{z^{1/2}}{\Pi(\pi,i)\omega^{\epsilon r^{-i}}}\right], \quad &\text{if}\quad i -1 \notin \pi \\[10pt]
\left( \frac{z^{1/2}}{\Pi(\pi,i)\omega^{r^{-i-1}} }, \frac{p_{i}z^{1/2}}{\Pi(\pi,i)\omega^{r^{-i}} }\right], \quad & \text{if}\quad i-1 \in \pi.
 \end{cases}
\end{align*}
Similarly, we obtain from the algorithm conditions $e^{g'} \in \mathcal{J}(\tau,i) $ for each $i \in \tau_0.$ 
Hence, we can set
\begin{align*}
D((p_j)_{j\in \pi}):=  \bigcap_{i \in \pi_0}  \mathcal{J}(\pi,i), \quad \quad D((p_j)_{j\in \tau}):=  \bigcap_{i \in \tau_0}  \mathcal{J}(\tau,i).
\end{align*}
Then (\ref{partition2}) is satisfied if and only if the partition algorithm produces the partition $\pi \sqcup \tau = \{1,2,\dots, J\}.$ Since $j_{i+1}-j_{i} \ll_{\epsilon} 1,$ the intervals are always contained in some $C$-adic intervals around $z^{1/2} /(\prod_{j \in \pi} p_j) $ and $z^{1/2} /(\prod_{j \in \tau} p_j) ,$ respectively.
\end{proof}

We now apply the above lemma to obtain a suitable factorization: let $\CC=\A$ or $\CC=\B,$ and suppose that the assumptions of Proposition \ref{typeiiestimate} are satisfied. Then by dividing $e$-adically the ranges for $uc,$ $vc',$ and $q_3,$ we obtain
\begin{align*}
\sum_{u,v}  W_{\mathcal{C}}(uv) a_u b_v = \sum_{g,g',\alpha} \sum_{uc \sim e^g} \sum_{vc' \sim e^{g'}} \sum_{\substack{q_1\sim Q_1 \\ q_2 \sim Q_2 \\ q_3 \sim e^{\alpha}}}  \, \, \sum_{\substack{p_1,\dots,p_J,r_1,\dots, r_L \\ e^{g+g'}Q_1Q_2e^{\alpha}p_1\cdots p_Jr_1\cdots r_L \asymp x}} a_u b_v f_{\CC},
\end{align*} 
where $f_{\mathcal{C}}:= f_{\mathcal{C}}(uvcc'q_1q_2q_3p_1 \cdots p_J r_1 \dots r_{L}  /x).$ By applying Lemma \ref{pa} and the remark following the lemma, we can partition this sum into
\begin{align*}
& \sum_{\substack{g,g',\alpha}} \sum_{uc \sim e^g} \sum_{vc' \sim e^{g'}} \sum_{\substack{q_1\sim Q_1 \\ q_2 \sim Q_2 \\ q_3 \sim e^{\alpha}}} \,\sum_{\pi \sqcup \tau = \{1,\dots,J\}} \, \sum_{\substack{p_1,\dots,p_J,r_1,\dots, r_L \\ e^{g+g'}Q_1Q_2e^{\alpha}p_1\cdots p_Jr_1\cdots r_L \asymp x \\  e^{g+\alpha} Q_1Q_2  \in D((p_j)_{j\in \pi}) \\ e^{g'} \in D((p_j)_{j\in \tau})}} a_u b_v f_{\CC}  \\
&=\sum_{\pi \sqcup \tau = \{1,\dots,J\}} \sum_{r_1,\dots, r_L}\sum_{\alpha}  \sum_{\substack{q_1\sim Q_1 \\ q_2 \sim Q_2 \\ q_3 \sim e^{\alpha}}} \sum_{\substack{p_j, \, j \in \pi}} \sum_{\substack{g \\ e^{g+\alpha} Q_1Q_2  \in D((p_j)_{j\in \pi}) }} \, \sum_{\substack{uc \sim e^g}}  \sum_{\substack{p_j, \, j \in \tau}} \sum_{\substack{g'\\ e^{g'} \hspace{-2pt}  \in D((p_j)_{j\in \tau}) }}  \sum_{\substack{vc' \sim e^{g'} }}\hspace{-5pt} a_ub_v f_{\CC}
\end{align*}
Let us now define 
\begin{align}
 \mathcal{I}(\alpha,(p_j)_{j\in \pi}):= \bigcup_{\substack{g \\ e^{g+\alpha} Q_1Q_2 \in D((p_j)_{j\in \pi})}} (e^g, e^{g+1}], \\ \mathcal{I}'((p_j)_{j\in \tau}):= \bigcup_{\substack{ g' \\ e^{g'} \in D((p_j)_{j\in \tau})}} (e^{g'}, e^{g'+1}]
\end{align}
Note that by Lemma \ref{pa} these are either empty or $C$-adic intervals. By using the Mellin inversion formula for $f_{\CC},$ we obtain
\begin{lemma} \label{mellinlemma} Let $\CC = \A$ or $\CC=\B,$ and suppose that  the assumptions of Proposition \ref{typeiiestimate} hold.  Then
\begin{align} \label{mellinintegral}
 \sum_{u,v}  W_{\mathcal{C}}(uv) a_ub_v =  \frac{1}{2\pi i} \int_{1-i\infty}^{1+i\infty} F(s) x^{s} \hat{f}_{\mathcal{C}}(s)\, \frac{ds}{s},
\end{align}
where 
\begin{align*}
F(s) := \sum_{\pi \sqcup \tau = \{1,2,\dots J\}}  R(s)^LG_{\tau}(s) Q_1(s)Q_2(s) \sum_{\alpha} \sum_{q \sim e^{\alpha}} q^{-s}F_{\pi,\alpha}(s) ,
\end{align*}
where the sum over $\alpha$ runs over $\lfloor \log^{9/10} x \rfloor  \leq \alpha < 2 \lfloor \log^{9/10} x \rfloor,$ $Q_i(s):=\sum_{q \sim Q_i} q^{-s},$ $R(s):= \sum_{r \sim \log^{\epsilon} x} r^{-s},$ and
\begin{align*}
F_{\pi,\alpha}(s)  &:= \sum_{\substack{p_j, \, j \in \pi}}  \sum_{\substack{uc\in \mathcal{I}(\alpha,(p_j)_{j\in \pi}) }} a_u (uc \prod_{j \in \pi} p_j)^{-s}\\
G_{\tau}(s) &:= \sum_{\substack{p_j, \, j \in \tau}}   \sum_{\substack{vc' \in \mathcal{I}'((p_j)_{j\in \tau}) }} b_v (vc'\prod_{j \in \tau} p_j)^{-s}.
\end{align*}
\end{lemma}

\begin{remark} We have separated the sums over $q_i$ to make the use of the method of Matom\"aki-Radzivi\l\l \, \cite{MR} easier; this spares us of the Lemma 12 of \cite{MR} in our situation. Dividing the ranges for $uc,$ $vc'$ into $e$-adic has the benefit that the sums over $uc$, $vc'$ always run over either a $C$-adic interval, or no interval at all (cf. proof of Lemma \ref{trivial} for why this is helpful). 
\end{remark}

We now divide the integration in (\ref{mellinintegral}) into two parts: let $T_0 := \log^{50} x,$ and define
\begin{align*}
 U_{\mathcal{C}} :=  \frac{1}{2\pi i} \int_{|t| \leq T_0} F(s) x^{s} \hat{f}_{\mathcal{C}}(s)\, \frac{ds}{s} \quad  \quad \quad V_{\mathcal{C}} :=   \frac{1}{2\pi i} \int_{|t| > T_0} F(s) x^{s} \hat{f}_{\mathcal{C}}(s)\, \frac{ds}{s}.
\end{align*}
Then by the same argument as in Section 10 of \cite{MR} (or by using Lemma \ref{mellin} similarly as in the proof of the Type I/II estimate) we find that
\begin{align*}
\left | \frac{1}{y}U_{\mathcal{A}}-\frac{\log^{100} x}{x}U_{\mathcal{B}} \right | \ll \log^{-20} x,
\end{align*}
and 
\begin{align*}
\left | \frac{1}{y}V_{\mathcal{A}}-\frac{\log^{100} x}{x}V_{\mathcal{B}} \right | \lesssim \int_{T_0}^{x/y} |F(1+it)| \, dt +& \frac{x}{y} \max_{T > x/y} \frac{1}{T} \int_{1+iT}^{1+i2T} |F(1+it)| \, dt \, +\\
& + (\log^{100} x) \max_{T > \log^{100} x} \frac{1}{T} \int_{1+iT}^{1+i2T} |F(1+it)| \, dt.
\end{align*}
The second term is handled by a similar argument as the first, and the third term is trivially bounded by the sum of the first two. The difficult part is to prove
\begin{prop} \label{mvtii} Let  $T_0 = \log^{50}x,$ and let $F(s)$ be as in Lemma \ref{mellinlemma}.
 Then
\begin{align*}
 \int_{T_0}^{x/y} |F(1+it)| \, dt \lesssim (\log^{-\delta} x ) \frac{\Omega}{\log x}.
\end{align*}
\end{prop}

For the proof need the following lemma, which states that we can obtain an estimate for the integral, which is of the correct order up to a factor of $\log^{o(1)}x,$ that is,
\begin{align*}
 \int_{T_0}^{x/y} |F(1+it)| \, dt \lesssim  \frac{\Omega}{\log x}.
\end{align*}
 After proving the lemma we can use the method of Matom\"aki and Radziwi\l\l \, \cite{MR} to get a saving $\log^{-\delta} x$ over this.  For all partitions $\pi \sqcup \tau,$ we will choose suitable $L(\pi,\tau) \leq L$ and write $R(s)^L= R(s)^{L(\pi,\tau)}R(s)^{L-L(\pi,\tau)}$ (cf. last case in the proof of Lemma \ref{trivial} below).  Then, once we use the triangle inequality to bring out the sum over the partitions, we can use Cauchy-Schwarz to the integral to obtain a sum over products of mean squares of Dirichlet polynomials (writing $s=1+it$)
\begin{align*}
 \sum_{\pi \sqcup \tau = \{1,2,\dots J\}}& \left( \int_{T_0}^{x/y} \left| R(s)^{L(\pi,\tau)}Q_1(s)Q_2(s) \sum_{\alpha} \sum_{q \sim e^{\alpha}} q^{-s}F_{\pi,\alpha}(s)  \right|^2  dt \right)^{1/2} \\ & \hspace{170pt}\cdot \left( \int_{T_0}^{x/y}\left| R(s)^{L-L(\pi,\tau)} G_{\tau}(s) \right|^2 dt\right)^{1/2}\hspace{-10pt}.
\end{align*}
 The quantity estimated in Lemma \ref{trivial} is then the result of applying the Improved mean value theorem (Lemma \ref{mvt2}). We note that obtaining this lemma is what determines the length of the interval $y$ (cf. line (\ref{length}) in the proof below). For the lemma, define the range for the combined variable $ucq_3$
\begin{align} \label{interval}
\mathcal{I}((p_j)_{j\in \pi}):= \bigcup_{\substack{g,\alpha \\ e^{g+\alpha} Q_1Q_2 \in D((p_j)_{j\in \pi})}} (e^{g+\alpha}, e^{g+\alpha+2}]
\end{align}
Clearly $\mathcal{I}((p_j)_{j\in \pi})$ is a $C$-adic interval and for any $\alpha$ we have $\mathcal{I}(\alpha,(p_j)_{j\in \pi}) \neq \emptyset$ if and only if $\mathcal{I}((p_j)_{j\in \pi}) \neq \emptyset$.
\begin{lemma} \emph{\textbf{(Correct-order estimate).}} \label{trivial}
Suppose that the assumptions of Proposition \ref{typeiiestimate} hold. There exists a function $L(\cdot,\cdot):\{(\pi,\tau): \, \pi \sqcup \tau = \{1,2,\dots,J\}\}  \to \{1,2,\dots, L\}$ such that the following holds: 

Let 
\begin{align*}
E &:= \sum_{\pi \sqcup \tau =\{1,\dots,J\}} \left( S_{\pi,1} +S_{\pi,2}\right)^{1/2}\left( S_{\tau,1} +S_{\tau,2}\right)^{1/2},  \\
W(\pi,\tau) &:= (\log x)^{\epsilon (L(\pi,\tau)-L/2) }
\end{align*}
where
\begin{align*}
S_{\pi,1} &:= \frac{ W(\pi,\tau)^2 }{y} \sum_{m \asymp W(\pi,\tau)^{-1} \sqrt{x}} \left| \quad \sum _{\substack{m=ucq_1q_2q_3 r_1 \dots r_{L(\pi,\tau)}\prod_{j \in \pi} p_j \\ ucq_3 \in \mathcal{I}((p_j)_{j \in \pi}) }} a_u \right|^2, \\
S_{\pi,2} &:= \frac{W(\pi,\tau)^2}{y}\sum_{h=1}^{\lfloor yx^{-1/2} W(\pi,\tau)^{-1}  \rfloor} \hspace{-12pt} \sum_{\substack{ucq_1q_2q_3 r_1 \dots r_{L(\pi,\tau)}\prod_{j \in \pi} p_j -u'c'q'_1q'_2q'_3 r'_1 \dots r'_{L(\pi,\tau)}\prod_{j \in \pi} p'_j = h \\ ucq_3 \in \mathcal{I}((p_j)_{j \in \pi}) \\ u'c'q'_3 \in \mathcal{I}((p'_j)_{j \in \pi}) }} |a_u||a_{u'}|, \\
S_{\tau,1} &:= \frac{ 1 }{yW(\pi,\tau)^2} \sum_{m \asymp W(\pi,\tau) \sqrt{x}} \left| \quad \sum _{\substack{m=vcr_1 \dots r_{L-L(\pi,\tau)}\prod_{j \in \tau} p_j \\  vc \in \mathcal{I}'((p_j)_{j \in \tau}) }} b_v \right|^2, \\
S_{\tau,2} &:= \frac{1}{yW(\pi,\tau)^2}\sum_{h=1}^{\lfloor yx^{-1/2} W(\pi,\tau)  \rfloor}  \sum_{\substack{ vcr_1 \dots r_{L-L(\pi,\tau)}\prod_{j \in \tau} p_j -v'c'r'_1 \dots r'_{L-L(\pi,\tau)}\prod_{j \in \tau} p'_j = h \\ vc \in \mathcal{I}'((p_j)_{j \in \tau}), \quad v'c' \in \mathcal{I}'((p'_j)_{j \in \tau}) }} |b_v||b_{v'}|.
\end{align*}
 Then  $E \, \lesssim \, \frac{\Omega}{\log x}.$ 
 
 Furthermore, the same bound holds in both of the following modified cases (with the same choice of $L(\cdot,\cdot)$)
 
 \textbf{\emph{(i):}} In the definitions of  $S_{\pi,1},S_{\pi,2}$, the primes  $q_1$ and $q'_1$ are replaced by 1, and  the factor $W(\pi,\tau)$ is replaced by $W(\pi,\tau) Q_1$.
 
  \textbf{\emph{(ii):}} In the definitions of  $S_{\pi,1},S_{\pi,2}$, the primes  $q_2$ and $q'_2$ are respectively replaced by the product of $H$ primes $q_{1,1},\dots,q_{1,H} \sim Q_1$, and $q'_{1,1},\dots,q'_{1,H} \sim Q_1$ (recall that $Q_2=Q_1^H$).
\end{lemma}
\begin{remark} The first modified case corresponds to a situation where the polynomial $Q_1(s)$ has been removed from $F(s)$. The second modified case corresponds to a situation where the polynomial $Q_2(s)$ has been replaced by $Q_1(s)^H$ in $F(s)$ (cf. definition of $F(s)$ in Lemma \ref{mellinlemma}, and the proof of Proposition \ref{mvtii} below).
\end{remark}
\emph{Proof of Lemma \ref{trivial}}.
By using $\sqrt{a+b} \leq \sqrt{a} + \sqrt{b},$ we get 
\begin{align*}
E \leq   \sum_{\pi \sqcup \tau =\{1,\dots,J\}} (S_{\pi,1}S_{\tau,1})^{1/2}   +  (S_{\pi,2}S_{\tau,2})^{1/2}   +   (S_{\pi,1}S_{\tau,2})^{1/2}   +  (S_{\pi,2}S_{\tau,1})^{1/2}.  
\end{align*}
We now estimate the four sums separately. In each case we first do the unmodified case and then explain how to get the estimate for the cases  \textbf{(i)} and   \textbf{(ii)}.

\textbf{Sum over $(S_{\pi,1}S_{\tau,1})^{1/2}:$} We need to estimate $S_{\pi,1}$ and $S_{\tau,1};$ our aim is to eventually use Cauchy-Schwarz to the sum $\sum_{\pi \sqcup \tau =\{1,\dots,J\}},$ and then regroup using Lemma \ref{pa}.   

Consider first $S_{\pi,1}.$ Since we assume that $|a_u |\lesssim \tau_w(u) 1_{(u,P(w))=1}$ and $c$ is $w$-smooth, we may estimate $uc$ from above by one smooth variable $n.$ Let
\begin{align*}
g(n):= 1+\sum_{n=m q_1 q_2q_3r_1 \cdots r_{L(\pi,\tau)}} 1.
\end{align*}
Then 
\begin{align*}
&\sum_{m \asymp  W(\pi,\tau)^{-1}\sqrt{x}} \left| \sum _{\substack{m=ucq_1q_2q_3 r_1 \dots r_{L(\pi,\tau)}\prod_{j \in \pi} p_j \\ ucq_3 \in \mathcal{I}((p_j)_{j \in \pi}) }} a_u \right|^2 \\
& \hspace{150pt}\lesssim \sum_{m \asymp  W(\pi,\tau)^{-1}\sqrt{x}} \left(\sum _{\substack{m=nq_1q_2q_3 r_1 \dots r_{L(\pi,\tau)}\prod_{j \in \pi} p_j \\ nq_3 \in \mathcal{I}((p_j)_{j \in \pi})}} \tau_w(n) \right)^2 \\
& \hspace{150pt} \lesssim  \sum_{m \asymp  W(\pi,\tau)^{-1}\sqrt{x}} \left(\sum _{\substack{m=n\prod_{j \in \pi} p_j \\   \mathcal{I}((p_j)_{j \in \pi}) \neq \emptyset  }} g(n) \tau_w(n) \right)^2, 
\end{align*}

Hence,  we need to bound
\begin{align}
\label{swapsq} \sum _{\substack{n\prod_{j \in \pi} p_j \asymp  W(\pi,\tau)^{-1}\sqrt{x} \\\mathcal{I}((p_j)_{j \in \pi}) \neq \emptyset }} \, \, \, \sum _{\substack{n\prod_{j \in \pi} p_j =n'\prod_{j \in \pi} p'_j \\ \mathcal{I}((p'_j)_{j \in \pi}) \neq \emptyset }} g(n)g(n')\tau_w(n)\tau_w(n').
\end{align}
We have $\tau_w(n') \lesssim \tau_w(n),$ since by the definition of the intervals $I_j$ if $p_j \geq w,$ then $j \ll \log \log \log x,$ so that if e.g. $p_1 \cdots p_{j} | n'$ and $p_i \nmid n'$ for other indices $i,$ then 
\begin{align*}
\tau_w(n') \leq \tau_w(p_1 \cdots p_{J_0}) \tau_w(n'/(p_1 \cdots p_{J_0})) \leq 2^{C \log \log \log x} \tau_w(n) \lesssim \tau_w(n).
\end{align*}
We also have $g(n') \lesssim g(n),$ since $g(n)-1$ counts the number of factors $q_1q_2q_3r_1\cdots r_L|n,$ and there exists only a bounded number of indices $j$ such that $p_j$ can be in the same range as one of the primes $q_1,q_2,q_3,r_l.$   Thus, (\ref{swapsq}) is bounded by
   \begin{align} \label{loss}
   \sum _{\substack{n \prod_{j \in \pi} p_j \asymp  W(\pi,\tau)^{-1}\sqrt{x} \\\mathcal{I}((p_j)_{j \in \pi}) \neq \emptyset}}  g(n)^2\tau_w(n)^2 \left( \sum _{\substack{n\prod_{j \in \pi} p_j=n'\prod_{j \in \pi} p'_j \\\mathcal{I}((p'_j)_{j \in \pi}) \neq \emptyset }}  1 \right).
   \end{align}

 To make progress, we drop the condition $\mathcal{I}((p'_j)_{j \in \pi}) \neq \emptyset.$ This causes a loss of some small power of $\log x$ but we do not know how to avoid this.  We then divide the sum into a sum over subsets $\rho \subseteq \pi,$ where $\rho$ is the set of indices $j$ such that $p_j \neq p'_j.$ Notice that this implies that $p'_j | n.$  That is, we have to give a bound for
\begin{align} \nonumber
\sum_{\substack{p_j, \, j \in \pi \\ \mathcal{I}((p_j)_{j \in \pi}) \neq \emptyset }} \sum_{\rho \subseteq \pi} \sum_{\substack{p'_j,\, j \in \rho \\ p'_j \neq p_j} } \, \,& \sum_{\substack{n \asymp  W(\pi,\tau)^{-1}\sqrt{x}/(\prod_{j\in \pi}p_j)\\  p'_j | n, \, j\in \rho }} g(n)^2 \tau_w(n)^2 \\ \label{gsum}
& \, \lesssim \, \sum_{\substack{p_j, \, j \in \pi \\ \mathcal{I}((p_j)_{j \in \pi}) \neq \emptyset }} \sum_{\rho \subseteq \pi} \sum_{p'_j, \, j \in \rho} \, \, \sum_{\substack{n \asymp  W(\pi,\tau)^{-1}\sqrt{x}/(\prod_{j\in \pi}p_j \prod_{j\in  \rho}p'_j) }} g(n)^2 \tau_w(n)^2, 
\end{align}
since $g(n)  \lesssim g(n/\prod_{j\in  \rho}p'_j)$ and $\tau_w(n) \lesssim \tau_w(n/\prod_{j\in  \rho}p'_j).$ We have $g(n) \leq L! \tilde{g}(n),$ where $\tilde{g}(n)$ is the multiplicative function defined by
\begin{align*}
\tilde{g}(p^k) :=  \begin{cases}  (k+1), & p \sim \log^{\epsilon} x, \, p \sim Q_1,   \, \,  p \sim Q_2, \,\, \text{or} \,\, \,  Q_3^{1/2} \leq p \leq Q_3 \\
 1, &\text{otherwise},  \end{cases}.
\end{align*}
 Thus, by Lemma  \ref{shiu} the sum (\ref{gsum}) is bounded by
\begin{align} \nonumber
&  (L!)^2 \sum_{\substack{p_j, \, j \in \pi \\ \mathcal{I}((p_j)_{j \in \pi}) \neq \emptyset }} \sum_{\rho \subseteq \pi} \sum_{p'_j, \, j \in \rho} \, \, \sum_{\substack{n \asymp  W(\pi,\tau)^{-1}\sqrt{x}/(\prod_{j\in \pi}p_j\prod_{j\in  \rho}p'_j)\\   }} \tilde{g}(n)^2 \tau_w(n)^2 \\ \nonumber
 & \quad  \ll  W(\pi,\tau)^{-1}\sqrt{x} \prod_{p \leq x} \left( 1+ \frac{\tilde{g}(p)^2\tau_w(p)^2-1}{p}\right)  \hspace{-10pt} \sum_{\substack{p_j, \, j \in \pi \\ \mathcal{I}((p_j)_{j \in \pi}) \neq \emptyset }}\hspace{-12pt} \sum_{\rho \subseteq \pi} \sum_{p'_j, \, j \in \rho} \left(\prod_{j \in \pi} p_j^{-1} \right)\left(\prod_{j \in \rho} (p_j')^{-1} \right) \\ \label{sumw}
 & \hspace{20pt}  \lesssim  W(\pi,\tau)^{-1}\sqrt{x} \prod_{j\in \pi} \left( 1+ \sum_{p \in I_j} p^{-1} \right)  \sum_{\substack{p_j, \, j \in \pi \\ \mathcal{I}((p_j)_{j \in \pi}) \neq \emptyset  }}  \left(\prod_{j \in \pi} p_j^{-1} \right). 
\end{align}
Recall now (by definitions) that for any $\alpha$  the set $ \mathcal{I}(\alpha,(p_j)_{j \in \pi})$ is always either empty or a $C$-adic interval around  $\sqrt{z}/(Q_1Q_2e^{\alpha}\prod_{j\in \pi} p_j)$ for some $C \geq e,$ and that $\mathcal{I}(\alpha,(p_j)_{j \in \pi})= \emptyset$ if and only if $\mathcal{I}((p_j)_{j \in \pi})= \emptyset.$  Recall also (\ref{smallprimes}) for the contributions of the small primes $q_1,q_2,q_3.$ Then (\ref{sumw}) is bounded by (using the definition (\ref{interval}) of  $\mathcal{I}((p_j)_{j \in \pi})$)
\begin{align*}
 \frac{W(\pi,\tau)^{-1}\sqrt{x}}{\sqrt{z}}(1+ \log 1/\theta)^{|\pi|} \hspace{-5pt} \sum_{\substack{n q_1q_2q_3 \prod_{j\in \pi} p_j \asymp  \sqrt{z} \\ n \in\mathcal{I}(\lfloor \log q_3 \rfloor,(p_j)_{j \in \pi})}} 1,
\end{align*}
so that
\begin{align} \label{case11}
S_{\pi,1} \, \lesssim \, \frac{W(\pi,\tau)\sqrt{x}}{y\sqrt{z}}(1+ \log 1/\theta)^{|\pi|} \hspace{-5pt} \sum_{\substack{n q_1q_2q_3 \prod_{j\in \pi} p_j \asymp  \sqrt{z} \\ n \in \mathcal{I}(\lfloor \log q_3 \rfloor,(p_j)_{j \in \pi})}} 1.
\end{align}
Similarly, we obtain 
\begin{align} \label{case11t}
S_{\tau,1} \, \lesssim \frac{ \sqrt{x} }{yW(\pi,\tau)\sqrt{z}} (1+ \log 1/\theta)^{|\tau|} \hspace{-5pt} \sum_{\substack{m \prod_{j\in \tau} p_j \asymp  \sqrt{z} \\ m \in \mathcal{I}((p_j)_{j \in \tau}) }} 1.
\end{align}
Hence, by Cauchy-Schwarz, and by applying Lemma \ref{pa}  to regroup
\begin{align}
\nonumber  \sum_{\substack{\pi \sqcup \tau =\{1,\dots,J\} }} & (S_{\pi,1}S_{\tau,1})^{1/2}  \\ \nonumber &  \lesssim \, \frac{\sqrt{x}}{y\sqrt{z}} (1+ \log 1/\theta)^{J/2}2^{J/2}   \left(\sum_{\substack{\pi \sqcup \tau =\{1,\dots,J\}}}\sum_{\substack{n q_1q_2q_3  \prod_{j\in \pi} p_j \asymp \sqrt{z} \\ n \in \mathcal{I}(\lfloor \log q_3 \rfloor,(p_j)_{j \in \pi})}} \, \,  \sum_{\substack{m \prod_{j\in \tau} p_j \asymp  \sqrt{z}  \\ m \in \mathcal{I}((p_j)_{j \in \tau}) }} 1 \right)^{1/2} \\ \nonumber 
&\leq  \frac{\sqrt{x}}{y\sqrt{z}} (1+ \log 1/\theta)^{J/2}2^{J/2} \left( \sum_{nmq_1q_2q_3 p_1 \cdots p_J \asymp z}  1\right)^{1/2}  \\  \nonumber 
& \lesssim \frac{\sqrt{x}}{y} (1+ \log 1/\theta)^{J/2}2^{J/2}   \left(\frac{\Omega}{\log x} \right)^{1/2} \\ \label{length}
&  \lesssim  \frac{\sqrt{x}}{y} (1+ \log 1/\theta)^{J/2}2^{J/2} (\log1/\theta)^{J/2} \log^{1/2} x \lesssim \frac{\Omega}{\log x},
\end{align}
where the last bound follows from the definition (\ref{y}) of $y$ and the definition (\ref{omega}) of $\Omega.$  We now discuss the modified cases:

  \textbf{Modified case (i):} There are two changes: Firstly, the  function $g(n)$ gets replaced by 
\begin{align*}
f(n) :=  1+\sum_{n=m q_2q_3r_1 \cdots r_{L(\pi,\tau)}} 1,
\end{align*}
which clearly satisfies $f(n) \ll \tilde{g}(n).$ Secondly, since the factor $W(\pi,\tau)$ is replaced by $W(\pi,\tau)Q_1,$ the final bound we obtain is
\begin{align*}
\sum_{\substack{\pi \sqcup \tau =\{1,\dots,J\}}}  (S_{\pi,1}S_{\tau,1})^{1/2} \, \lesssim \,   \frac{Q_1 x^{1/2}}{y} (1+ \log 1/\theta)^{J/2}2^{J/2} (\log1/\theta)^{J/2} \log^{1/2} x. 
\end{align*}
Since $Q_1=\log^{10\delta} x$ and $y= x^{1/2}\log^{\beta +10\delta + o(1)} x,$ the claim follows.

\textbf{Modified case (ii):} Here the only change is that the function $g(n)$ is replaced by 
\begin{align*}
f(n)  :=1+\sum_{n=n' q_1q_{1,1} \cdots q_{1,H}q_3 r_1 \dots r_{L(\pi,\tau)}} 1  \, \leq \, 2(H+1)! L! \tilde{g}(n) \lesssim \tilde{g}(n),
\end{align*}
since $H \ll (\log \log x)^{1/2}.$ Thus, the claim follows.

\textbf{Sum over $(S_{\pi,2}S_{\tau,2})^{1/2}:$} We need to estimate $S_{\pi,2}$ and $S_{\tau,2}.$ These sums are essentially averages over $h$ of correlations of some sequence $g_n$, of the form $N^{-1}\sum_{n\sim N} g_n g_{n+h}.$ The aim is to show that these correlations behave as expected, so that we get (on average over $h$) 
\begin{align*}
N^{-1} \sum_{n \sim N} g_n g_{n+h} \, \lesssim  \left(N^{-1} \sum_{n \sim N} g_n\right)^2.
\end{align*}
Then the square cancels the square root in $(S_{\pi,2}S_{\tau,2})^{1/2}$, so that we may use the partition algorithm to regroup the sums, giving us an estimate of the correct order.

Consider first $S_{\pi,2}$.  Note that if $ yx^{-1/2} W(\pi,\tau)^{-1}  < 1,$ then the bound is trivial, so assume the opposite. Recall that for any integers $a,b,h$ such that $(a,b) | h$, the number of integer solutions $(c,c')$ to the diophantine equation $ac+bc'=h$ with $ac \sim M$ is bounded by
\begin{align*}
\sum_{\substack{c \sim M/a\\ x \equiv c_0 \, (\text{mod} \,b/(a,b))}} 1 \, \ll 1+\frac{M (a,b)}{ab},
\end{align*}
 where $(c_0,c'_0)$ is any given solution to the equation. Hence, for a fixed $h$ we have (here we drop the usual restrictions for $c,c',$ and count the number of solutions $(c,c')$)
\begin{align}
 \nonumber &\sum_{\substack{ucq_1q_2q_3 r_1 \dots r_{L(\pi,\tau)}\prod_{j \in \pi} p_j -u'c'q'_1q'_2q'_3 r'_1 \dots r'_{L(\pi,\tau)}\prod_{j \in \pi} p'_j = h \\ ucq_3 \in \mathcal{I}((p_j)_{j \in \pi}) \\ u'c'q'_3 \in \mathcal{I}((p'_j)_{j \in \pi})}} |a_u||a_{u'}| \\
 & \nonumber \hspace{10pt} \leq \sum_{u,u'} |a_u||a_{u'}|\sum_{q_i,q_i'} \sum_{r_l,r_l'} \sum_{\substack{p_j, p'_j \, j \in \pi \\ \mathcal{I}((p_j)_{j \in \pi}) \neq \emptyset \\ \mathcal{I}((p'_j)_{j \in \pi}) \neq \emptyset }} \, \, \sum_{\substack{c,c' \\  cuq_1q_2q_3 r_1 \dots r_{L(\pi,\tau)}\prod_{j \in \pi} p_j \asymp W(\pi,\tau)^{-1} x^{1/2}  \\ cuq_1q_2q_3 r_1 \dots r_{L(\pi,\tau)}\prod_{j \in \pi} p_j  -c'u'q'_1q'_2q'_3 r'_1 \dots r'_{L(\pi,\tau)}\prod_{j \in \pi} p'_j = h}} \hspace{-10pt}1 \\
 &\label{s2} \hspace{10pt} \lesssim 1\,+\, W(\pi,\tau)^{-1} x^{1/2} \sum_{u,u'} \frac{|a_u||a_{u'}|}{uu'}\sum_{q_i,q_i'} (q_1q_2q_3q'_1q'_2q'_3)^{-1}  \\ 
 & \nonumber\hspace{110pt} \sum_{r_l,r_l'} (r_1 \cdots r_{L(\pi,\tau)}r'_1 \cdots r'_{L(\pi,\tau)})^{-1}\sum_{\substack{p_j, \, j \in \pi \\ \mathcal{I}((p_j)_{j \in \pi}) \neq \emptyset }}\sum_{\substack{p'_j, \, j \in \pi \\ \mathcal{I}((p'_j)_{j \in \pi}) \neq \emptyset}} 
  \frac{G \cdot 1_{G|h} }{\prod_{j\in \pi} p_jp'_j}
\end{align}
where $G:=\gcd\left(uq_1q_2q_3r_1 \cdots r_{L(\pi,\tau)} \prod_{j \in \pi}p_j, \, \, u'q'_1q'_2q'_3r'_1 \cdots r'_{L(\pi,\tau)} \prod_{j \in \pi}p'_j\right).$ The contribution from the `$+1$' is trivially small enough, so that we may ignore it. Averaging over $h$ we have
\begin{align*}
\frac{W(\pi,\tau) x^{1/2}}{y}\sum_{h=1}^{\lfloor yx^{-1/2} W(\pi,\tau)^{-1}  \rfloor} G \cdot 1_{G|h}  \, \ll 1.
\end{align*}
Since $|a_u| \lesssim 1_{(u,P(w))=1} \tau_w(u),$ we have
\begin{align*}
 \sum_{u,u'} \frac{|a_u||a_{u'}|}{uu'}\, \lesssim \, 1.
\end{align*}
Hence,
\begin{align} \label{case22}
S_{\pi,2} \, & \lesssim \left(  \sum_{\substack{p_j, \, j \in \pi \\\mathcal{I}((p_j)_{j \in \pi}) \neq \emptyset}} \prod_{j \in \pi} p_j^{-1} \right)^2 \, \lesssim \, \left(  \frac{1}{ \sqrt{z}} \sum_{\substack{n q_1q_2q_3 \prod_{j\in \pi} p_j \asymp \sqrt{z} \\ n \in \mathcal{I}(\lfloor \log q_3 \rfloor,(p_j)_{j \in \pi})}} 1  \right)^2,
\end{align}
where the last bound again follows  from the facts that $ \mathcal{I}(\alpha,(p_j)_{j \in \pi})$ is always either empty or a $C$-adic interval around  $\asymp W(\pi,\tau)^{-1}\sqrt{x}/( Q_1Q_2e^{\alpha}\prod_{j\in \pi} p_j)$ for some $C \geq e,$ and that for all any $\alpha$ we have $\mathcal{I}(\alpha,(p_j)_{j \in \pi})= \emptyset$ if and only if $\mathcal{I}((p_j)_{j \in \pi})= \emptyset.$

Similarly, we obtain
\begin{align} \label{case22t}
S_{\tau,2} \, \lesssim  \left(  \frac{1}{ \sqrt{z}} \sum_{\substack{m  \prod_{j\in \tau} p_j \asymp  \sqrt{z} \\ m \in \mathcal{I}'((p_j)_{j \in \tau}) }} 1  \right)^2.
\end{align}
 Thus, by using Lemma \ref{pa} to regroup we obtain
\begin{align*}
\sum_{\pi \sqcup \tau =\{1,\dots,J\}} (S_{\pi,2}S_{\tau,2})^{1/2} \lesssim  \frac{1}{z} \sum_{nmq_1q_2q_3 p_1 \cdots p_J  \asymp z} 1 \lesssim \frac{\Omega}{\log x}.
\end{align*}
The modified cases \textbf{(i)} and \textbf{(ii)} clearly follow by a similar argument, since in \textbf{(i)} every relevant factor is scaled similarly throughout by factor $Q_1$ or $Q_1^{-1}$ (note the averaging over $h$), and the modification in \textbf{(ii)} is harmless since $H \ll (\log \log x)^{1/2}$ so that $\left( \sum_{q \sim Q_1} q^{-1}\right)^H = \log^{o(1)} x.$

\textbf{Sum over $(S_{\pi,1}S_{\tau,2})^{1/2} +(S_{\pi,2}S_{\tau,1})^{1/2} :$}  Let us denote 
\begin{align*}
E_{\pi,1} &:= \frac{ W(\pi,\tau)\sqrt{x}}{y} (1+\log 1/\theta)^{|\pi |} \sum_{\substack{p_j, \, j \in \pi \\\mathcal{I}((p_j)_{j \in \pi}) \neq \emptyset}} \prod_{j \in \pi} p_j^{-1} , \quad  E_{\pi,2} := \left(  \sum_{\substack{p_j, \, j \in \pi \\\mathcal{I}((p_j)_{j \in \pi}) \neq \emptyset}} \prod_{j \in \pi} p_j^{-1} \right)^2, \\
E_{\tau,1}&:=  \frac{\sqrt{x}}{ W(\pi,\tau) y} (1+\log 1/\theta)^{|\tau |} \sum_{\substack{p_j, \, j \in \tau \\\mathcal{I}'((p_j)_{j \in \tau}) \neq \emptyset}} \prod_{j \in \tau} p_j^{-1}, \quad E_{\tau,2}:=\left(  \sum_{\substack{p_j, \, j \in \tau \\\mathcal{I}'((p_j)_{j \in \tau}) \neq \emptyset}} \prod_{j \in \tau} p_j^{-1} \right)^2
\end{align*}
so that by (\ref{case11}), (\ref{case11t}), (\ref{case22}), and (\ref{case22t}) we have 
\begin{align*}
S_{\pi,1} \lesssim E_{\pi,1}, \quad \quad S_{\pi,2} \lesssim E_{\pi,2}, \quad \quad S_{\tau,1} \lesssim E_{\tau,1} \quad \quad S_{\tau,2} \lesssim E_{\tau,2}.
\end{align*}
Our strategy here is to choose $L(\pi,\tau)$ so that $E_{\pi,1}E_{\tau,2} \approx E_{\pi,2}E_{\tau,1},$ and then use Cauchy-Schwarz to reduce the estimate to the previous cases.

We must first deal separately with partitions $\pi\sqcup \tau$ such that one of $E_{\pi,i}, E_{\tau,i}$ is exceptionally small. We note that trivially (since $(\log 1/\theta)(1+\log 1/\theta) <1$)
\begin{align*}
 E_{\pi,1}W(\pi, \tau)^{-1} \lesssim ((\log 1/\theta)(1+\log 1/\theta))^{|\pi|} \leq 1,
\end{align*}
and similarly $E_{\tau,1} W(\pi,\tau) \lesssim 1.$  Note also that $E_{\pi,2} \lesssim 1, \, E_{\tau,2} \lesssim 1.$ Suppose then that $\pi \sqcup \tau$ is such that 
\begin{align*}
\min \left \{E_{\pi,1} E_{\tau,2} W(\pi, \tau)^{-1}, \quad E_{\pi,2}E_{\tau,1}W(\pi, \tau)  \right \} \leq \log^{-100}x
\end{align*}
Recall that $W(\pi,\tau)=  (\log x)^{\epsilon (L(\pi,\tau)-L/2) }.$ If  $L > 100/\epsilon,$ then we may clearly choose $L(\pi,\tau)$ so that  $E_{\pi,1}E_{\tau,2}$ and $ E_{\pi,2}E_{\tau,1}$ are both bounded by $\log^{-40}x.$  Thus, the sum over such $\pi\sqcup \tau$ is trivially bounded by $2^J \log^{-20} x < \log^{-10} x,$ which is clearly sufficient.

Hence, we may assume that 
\begin{align*}
 E_{\pi,1}E_{\tau,2}W(\pi,\tau)^{-1}  \quad \text{and} \quad E_{\pi,2}E_{\tau,1}W(\pi,\tau)
\end{align*}
 are within a factor of $\log^{100} x$ of each other.  Choose $L$ large enough so that $L>200/\epsilon.$ We may then choose $L(\pi,\tau)$ so that $E_{\pi,1}E_{\tau,2}$ and $ E_{\pi,2}E_{\tau,1}$ are equal up to a factor bounded by $\log^{\epsilon} x = \log^{o(1)} x.$ We then obtain by Cauchy-Schwarz
\begin{align*}
\sum_{\pi \sqcup \tau =\{1,\dots,J\}}   (E_{\pi,1}E_{\tau,2})^{1/2}   + & (E_{\pi,2}E_{\tau,1})^{1/2}  \lesssim \sum_{\pi \sqcup \tau =\{1,\dots,J\}}   (E_{\pi,1}E_{\tau,1} E_{\pi,2}E_{\tau,2})^{1/4} \\
& \leq \left( \sum_{\pi \sqcup \tau =\{1,\dots,J\}}  ( E_{\pi,1}E_{\tau,1})^{1/2}\right)^{1/2} \left( \sum_{\pi \sqcup \tau =\{1,\dots,J\}}   (E_{\pi,2}E_{\tau,2})^{1/2}\right)^{1/2},
\end{align*}
which reduces the bound to the previous cases. The modified cases \textbf{(i)} and \textbf{(ii)} again follow by a similar argument.
\qed

We also require the following variant of the above lemma, which will be used after applying Lemma \ref{halasz}. Here we care about the partition only up to an accuracy of $x^{o(1)},$ since we will apply Lemma \ref{halasz} with $T^{1/2}|\mathcal{T}| \ll (x/y)^{1-\epsilon}.$
\begin{lemma} \label{trivial2} For any partition $\pi \sqcup \tau = \{1,2,\dots, J\},$ and any  $\lfloor \log^{9/10} x \rfloor  \leq \alpha < 2 \lfloor \log^{9/10} x \rfloor,$ define
\begin{align*}
S_{\pi} (\alpha)&:= \sum_{m \asymp e^{-\alpha}(\log^{L\epsilon/2} x) \sqrt{x}} \left| \quad \sum _{\substack{m=ucq_1q_2 r_1 \dots r_{L}\prod_{j \in \pi} p_j \\ uc \in \mathcal{I}(\alpha,(p_j)_{j \in \pi})}} a_u \right|^2, \\
S_{\tau} &:=  \sum_{m \asymp  (\log^{-L\epsilon/2} x)\sqrt{x}} \left| \quad \sum _{\substack{m=vc\prod_{j \in \tau} p_j \\  vc \in \mathcal{I}'((p_j)_{j \in \tau}) }} b_v \right|^2.
\end{align*}
Then, for any $\lfloor \log^{9/10} x \rfloor  \leq \alpha < 2 \lfloor \log^{9/10} x \rfloor$ and partition $\pi \sqcup \tau = \{1,2,\dots, J\},$
\begin{align*}
\frac{e^{\alpha}}{x}S_{\pi} (\alpha) S_{\tau} \, \lesssim \, 1.
\end{align*}
\end{lemma}
\begin{proof}
We use the same argument as with the first case in the proof of Lemma \ref{trivial} (compare with (\ref{case11}) and (\ref{case11t})) to obtain
\begin{align*}
 S_{\pi} (\alpha)  \, &\lesssim \, e^{-\alpha}(\log^{L\epsilon/2} x) \sqrt{x} \,  (1+\log 1/\theta)^{|\pi |}    \sum_{\substack{p_j, \, j \in \pi }} \prod_{j \in \pi} p_j^{-1} \, \lesssim \, e^{-\alpha} (\log^{L\epsilon/2} x)\sqrt{x} \\
S_{\tau}  &\lesssim \,(\log^{-L\epsilon/2} x)\sqrt{x} \,(1+ \log 1/\theta)^{|\tau|} \sum_{\substack{p_j, \, j \in \tau}} \prod_{j \in \tau} p_j^{-1} \, \lesssim \, (\log^{-L\epsilon/2} x)\sqrt{x},
\end{align*}
since $(\log 1/\theta)(1+\log 1/\theta) < 1$.
\end{proof}

\emph{Proof of Proposition \ref{mvtii}.}
Write 
\begin{align*}
[T_0,x/y] = \mathcal{T}_1 \cup \mathcal{T}_2 \cup \mathcal{U},
\end{align*}
where
\begin{align*}
\mathcal{T}_1 &:= \{t \in [T_0,x/y]: |Q_1(1+it)| \leq Q_1^{-1/4+2\epsilon} \}, \\
\mathcal{T}_2 &:= \{t \in [T_0,x/y]: |Q_2(1+it)| \leq Q_2^{-1/4+\epsilon} \} \setminus \mathcal{T}_1,
\end{align*}
and  $\mathcal{U} := [T_0,x/y]\setminus (\mathcal{T}_1 \cup \mathcal{T}_2).$ We estimate the integral over each region separately.

\textbf{Integral over $\mathcal{T}_1 $:} We have (for $s=1+it$)
\begin{align*}
\int_{\mathcal{T}_1} |F(s)| dt \leq Q_1^{-1/4+2\epsilon}  \sum_{\pi \sqcup \tau = \{1,2,\dots J\}}  \int_{T_0}^{x/y}   \left| R(s)^L G_{\tau}(s) Q_2(s) \sum_{\alpha} \sum_{q \sim e^{\alpha}} q^{-s}F_{\pi,\alpha}(s)  \right | \, dt.
\end{align*}
 Choose $L(\pi,\tau)$ as in Lemma \ref{trivial}, and use Cauchy-Schwarz to obtain
\begin{align*}
\int_{T_0}^{x/y}    & \left| R(s)^{L-L(\pi,\tau)}G_{\tau}(s)R(s)^{L(\pi,\tau)}  Q_2(s)  \sum_{\alpha} \sum_{q \sim e^{\alpha}}  q^{-s}F_{\pi,\alpha}(s)  \right | \, dt \leq \\
& \left( \int_{T_0}^{x/y}    \left| R(s)^{L(\pi,\tau)} Q_2(s) \sum_{\alpha} \sum_{q \sim e^{\alpha}} q^{-s}F_{\pi,\alpha}(s)  \right |^2 \, dt\right)^{1/2} \left( \int_{T_0}^{x/y}    \left| R(s)^{L-L(\pi,\tau)}G_{\tau}(s) \right |^2 \, dt \right)^{1/2}.
\end{align*}
We now apply  Lemma \ref{mvt2}  and the modified case \textbf{(i)} of Lemma \ref{trivial} to obtain
\begin{align*}
\int_{\mathcal{T}_1} |F(s)| \, dt \lesssim   Q_1^{-1/4+2\epsilon} \frac{\Omega}{\log x} \ll (\log^{-\delta} x ) \frac{\Omega}{\log x}.
\end{align*}

\textbf{Integral over $\mathcal{T}_2$:} Since $|Q_1(1+it)| > Q_1^{-1/4+2\epsilon}$ on $\mathcal{T}_2,$ we obtain (for $s=1+it$)
\begin{align*}
\int_{\mathcal{T}_2}|F(s)| dt  \leq  Q_2^{-1/4+\epsilon} Q_1^{H(1/4-2\epsilon)} \sum_{\pi \sqcup \tau = \{1,2,\dots J\}}  \int_{T_0}^{x/y}    \left| R(s)^L G_{\tau}(s) Q_1(s)^{H+1} \sum_{\alpha} \sum_{q \sim e^{\alpha}} q^{-s}F_{\pi,\alpha}(s) \right | \, dt,
\end{align*}
where $Q_2^{-1/4+\epsilon} Q_1^{H(1/4-2\epsilon)} = Q_2^{-\epsilon} \leq \log^{-2\delta} x.$ By the same argument as with the integral over $\mathcal{T}_1$, applying the modified case \textbf{(ii)} of Lemma \ref{trivial}, we obtain
 \begin{align*}
\int_{\mathcal{T}_2} |F(s)| \, dt \lesssim  (\log^{-\delta} x ) \frac{\Omega}{\log x}.
\end{align*}

\textbf{Integral over $\mathcal{U}$:} Let $\mathcal{T} \subset \mathcal{U}$ be a set of well-spaced points such that
\begin{align*}
\int_{\mathcal{U}} |F(1+it)| \, dt \ll \sum_{t\in \mathcal{T}} |F(1+it)|.
\end{align*}
By Lemma \ref{large} applied to $Q_2(s)$ we have
\begin{align*}
|\mathcal{T}| \, \ll (x/y)^{1/2 - 2\epsilon }Q_2^{1/2}\exp \left( 2\log (x/y) \frac{\log \log (x/y)}{\log Q_2}\right) \leq (x/y)^{1/2-\epsilon},
\end{align*}
since $Q_2 \sim \exp((\log\log x)^{3/2}).$ By Lemma \ref{vino} we have for any  $\lfloor \log^{9/10} x \rfloor  \leq \alpha < 2 \lfloor \log^{9/10} x \rfloor$ and any $t \in [\log^{50} x,x]$
\begin{align*}
\left | \sum_{q\sim e^{\alpha}} q^{-1-it} \right | \ll  \exp \left( -\frac{\alpha}{\log^{7/10} t}\right) + \frac{ \alpha^3 }{t} \ll \log^{-40} x.
\end{align*}
Hence, by Cauchy-Schwarz (for $s=1+it$)
\begin{align*}
\sum_{t\in \mathcal{T}} |F(1+it)| &\ll \log^{-39} x  \sum_{\pi \sqcup \tau = \{1,2,\dots J\}} \max_{\alpha} \sum_{t\in \mathcal{T}}  \left| R(s)^LQ_1(s)Q_2(s) F_{\pi,\alpha}(s) G_{\tau}(s) \right | \\
& \leq \log^{-30} x  \max_{\pi \sqcup \tau } \max_{\alpha} \left( \sum_{t\in \mathcal{T}}  \left|R(s)^L Q_1(s)Q_2(s) F_{\pi,\alpha}(s)  \right |^2 \right)^{1/2}\left( \sum_{t\in \mathcal{T}}  \left|  G_{\tau}(s) \right |^2 \right)^{1/2}.\\
\end{align*}
Thus, applying Lemma \ref{halasz} (with $T=x/y, \, |\mathcal{T}| \, \ll (x/y)^{1/2-\epsilon} $) we obtain 
\begin{align*}
\sum_{t\in \mathcal{T}} |F(1+it)| &\ll  (\log^{-30} x) \max_{\pi \sqcup \tau } \max_{\alpha} \left( \left( \frac{e^{\alpha} \log^{L\epsilon/2} x}{\sqrt{x}} +  e^{2\alpha} (\log^{L\epsilon} x) x^{-1/2-\epsilon} \right) S_{\pi} (\alpha) \right)^{1/2}  \\
& \hspace{180pt} \cdot  \left( \left( \frac{ 1}{(\log^{L\epsilon/2} x)\sqrt{x}} +  \frac{x^{-1/2-\epsilon}}{\log^{L\epsilon} x}  \right) S_{\tau}  \right)^{1/2} ,
\end{align*}
where $S_{\pi} (\alpha)$ and $ S_{\tau}$ are as in Lemma \ref{trivial2}. Thus, by Lemma \ref{trivial2} 
\begin{align*}
\sum_{t\in \mathcal{T}} |F(1+it)| \ll (\log^{-30} x) \max_{\pi \sqcup \tau } \max_{\alpha}  \left(\frac{e^{\alpha}}{x} S_{\pi} (\alpha) S_{\tau} \right)^{1/2} \ll \log^{-20} x,
\end{align*}
which is sufficient for the proposition.
\qed

\emph{Conclusion of the proof of Proposition \ref{typeiiestimate}.}
Recall that
\begin{align*}
\left | \frac{1}{y}V_{\mathcal{A}}-\frac{\log^{100} x}{x}V_{\mathcal{B}} \right | \lesssim \int_{T_0}^{x/y}  |F(1+it)| \, dt +& \frac{x}{y} \max_{T > x/y} \frac{1}{T} \int_{1+iT}^{1+i2T} |F(1+it)| \, dt \, +\\
& + (\log^{100} x) \max_{T > \log^{100} x} \frac{1}{T} \int_{1+iT}^{1+i2T} |F(1+it)| \, dt.
\end{align*}
The first integral can be bounded using Proposition \ref{mvtii}. The second term is bounded by using the same argument as for the first, since the integral is multiplied by the factor $T^{-1}x/y$, and the sum over $h$ in the off-diagonal term from applying Lemma \ref{mvt2} is now shorter than with the first integral. The third term is trivially bounded by  the sum of the first and second terms.
\qed

\subsection{Discussion of the loss}
As can be seen from the above, the reason we cannot make the interval shorter than $(\log^{\beta + \delta} x )\sqrt{x}$ is due to losses in the correct-order estimate Lemma \ref{trivial}. To see how this loss occurs, consider a sequence $a_n, \, n\sim N,$  which is the characteristic function of some well-behaved set of density $\rho$ around $N.$ Then we expect
\begin{align*}
\frac{1}{N} \sum_{\substack{ nm \in [N^2,N^2+N] \\ n,m\sim N}} a_na_m \asymp \rho^2,
\end{align*}
but estimate for the corresponding mean value is
\begin{align*}
\int_{0}^{N} \left | \sum_{n\sim N} a_n n^{-1-it} \right|^2 \, dt \, \ll \frac{1}{N} \sum_{n\sim N} |a_n|^2 \,  = \rho.
\end{align*}
 Hence, we have already lost a square root of the density. This is of course because the diagonal term in the mean value theorem corresponds to square root cancellation on average.

At first it may seem that including the variables $c,c'$ causes a loss of a factor $(1+\log 1/\theta)^{J/2} = \log^{0.39\dots+o(1)} x.$ However, without these variables we would lose a factor $\log x$ due to a smaller density, so that it is beneficial to have them in the mix (not to mention that we needed one of them in the proof of the Type I/II estimate).

As was noted in the proof of Lemma \ref{trivial}, some of our losses come from our inability to handle the cross-conditions in the sum (\ref{loss}), but this inaccuracy contributes definitely no more than $(1+\log 1/\theta)^{J/2} = \log^{0.39\dots+o(1)} x.$ Another potential loss is the use of Cauchy-Schwarz in the case of the sums $(S_{\pi,1}S_{\tau,1})^{1/2}$ in the proof; the Cauchy-Schwarz is optimal if most of the terms $S_{\pi,1}S_{\tau,1}$ are of the same size, but this may not be the case (depending on the partition $\pi \sqcup \tau,$ some of the cross-conditions are expected to be more strict than others). We do not pursue these issues here, as they would require a significant effort with relatively small improvements.

An alternative construction one might consider is to let the primes $p_j$ vary more freely by installing some cross-conditions, eg. of the form $p_{j+1} \cdots p_J \gg p_j.$ This would indeed increase the density of our sequence. However, to be able to use Cauchy-Schwarz, we would need to remove the cross-conditions going between $\pi$ and $\tau.$ At best (using smooth cross-conditions), removing one cross-condition causes a loss of a constant $C >1,$ and there are typically $\gg \log\log x$ cross-conditions to be removed, causing an additional loss of $\log^{C'} x.$ We expect that more is lost than gained in this approach.

Yet another set-up would be to make the intervals $I_j$ narrower, so that we could get a better control over the number of partitions needed (e.g. we could get $j_{l+1}-j_l \leq N$ for some fixed $N$ in the partition algorithm). This improves the factor $2^{J/2},$ but the losses from the narrowness of $I_j$ grow faster, making the bound worse.

\section{Fundamental proposition} \label{funprop}
From here on we shall not need the precise structure of the weights $W_{\mathcal{A}}, W_{\mathcal{B}}.$ Hence, we can freely use summation variables of type $p_j,q_j,r_j$ in the sieve decompositions below without risk of confusion. The aim of this section is to prove a proposition, which combines the previous estimates Type I/II and Type II, by using Harman's iterative argument (cf. Chapter 3 of Harman's book \cite{Har}). The proposition plays the same role as Lemma 5.3 in Chapter 5 of Harman's book.

For any natural number $d$ and any $U \geq 1,$ define
\begin{align*}
S(\mathcal{A}_d,U) := \sum_{n } 1_{(n,P(U))=1} W_{\mathcal{A}}(dn), \quad \quad S(\mathcal{B}_d,U) := \sum_{n } 1_{(n,P(U))=1} W_{\mathcal{B}}(dn).
\end{align*}
The basic idea is as follows: Suppose that we want to estimate $S(\A, x^{\gamma-2\delta}).$ By using the elementary identity $(\mu \ast 1)(n) = 1_{n=1},$ we have
\begin{align*}
S(\A, x^{\gamma-2\delta}) &= \sum_{d | P(x^{\gamma-2\delta})} \mu(d) S(\A_d, 1) \\ &= \sum_{\substack{n \\ d | P(x^{\gamma-2\delta})\\ d  < x^{1/2-\gamma+\delta} }} \mu(d) W_\A(nd) + \sum_{\substack{n, \\ d | P(x^{\gamma-2\delta})\\ d  \geq x^{1/2-\gamma+\delta} }} \mu(d) W_\A (nd)=:\Sigma_I + \Sigma_{II},
\end{align*}
say. In $\Sigma_I$ we have obtained a long smooth variable $n$ so that we have a Type I sum (cf. Proposition \ref{typeiestimate}). On the other hand, in $\Sigma_{II}$ we can write $d=p_1 \cdots p_k$ with $p_j \leq x^{\gamma-2\delta}$ for $j=1,2,\dots,k.$ Then there is some $j \leq k$ such that $p_1 \cdots p_j \in [x^{1/2-\gamma +\delta}, x^{1/2-\delta}],$ which means that we have a Type II sum (cf. Proposition \ref{typeiiestimate}). Here we come across a problem: in the Type II sum we also have a smooth variable $n,$ which means that the sum could be at least one factor of $\log x$ larger than the original sum (if we ignore the cancellations from $\mu(d)$). To overcome this problem, we must add a cut-off to the Buchstab's identity from below, so that we write 
\begin{align*}
S(\A, x^{\gamma-2\delta}) &= \sum_{d | P(x^{\gamma-2\delta})/P(w)} \mu(d) S(\A_d, w)
\end{align*}
with $w=x^{1/(\log \log x)^2}.$ This solves our problems, except that now in the Type I/II sum we also have $(n,P(w))=1.$ However, as was noted in Section \ref{typei}, this is not a problem since the weight $W_\A$ contains a $w$-smooth variable $c$ which can be combined with $n$ to form a long smooth variable.

In practice, we need a result of a more general from:
\begin{prop} \emph{\textbf{(Fundamental proposition).}} \label{ft} Let $Z=x^{\gamma-2\delta},$ $X=x^{1/2-\gamma/2+\delta},$ and let $U,V \geq 1,$ $U \leq x^{1/4-10\delta},$ $V \leq x^{1/2-\gamma+\delta}.$ Let $a_u, b_v \ll 1$ be some non-negative coefficients, supported for $(u,P(Z))=1,(v,P(Z))=1.$ Define  $\lambda:= y \log^{100} x /x.$ Then
\begin{align*}
\sum_{\substack{u \in (U^{1-\epsilon},U] \\
v \in (V^{1-\epsilon},V]}} a_u b_v S(\A_{uv},Z) = \lambda \sum_{\substack{u \in (U^{1-\epsilon},U] \\
v \in (V^{1-\epsilon},V]}} a_u b_v S(\B_{uv},Z) + \mathcal{O}\left( \lambda S(\B,X)\log^{-\delta/2} x \right).
\end{align*}
\end{prop}
\begin{proof}
 Define $W=x^{1/2-\gamma+\delta},$ and let $\CC= \A$ or $\CC=\B.$  Using Buchstab's identity we obtain
\begin{align*}
&\sum_{\substack{u \in (U^{1-\epsilon},U] \\
v \in (V^{1-\epsilon},V]}} a_u b_v S(\CC_{uv}, Z) \\
&= \sum_{\substack{u \in (U^{1-\epsilon},U] \\
v \in (V^{1-\epsilon},V]}} a_u b_v \sum_{ \substack{d | P(Z)/P(w) \\ vd < W}} \mu(d) S(\CC_{uvd},w) + \sum_{\substack{u \in (U^{1-\epsilon},U] \\
v \in (V^{1-\epsilon},V]}} a_u b_v \sum_{ \substack{d | P(Z)/P(w) \\ vd \geq W}} \mu(d) S(\CC_{uvd},w) \\[10pt]
&=: \Sigma_I(\CC) + \Sigma_{II}(\CC),
\end{align*}
say. We will apply the Type I/II estimate (Proposition \ref{typeiestimate}) to  $\Sigma_I(\CC)$ and the Type II estimate (Proposition \ref{typeiiestimate}) to $\Sigma_{II}(\CC).$

\textbf{Sums $\Sigma_I(\CC)$:} We let $v'=vd,$ and 
\begin{align*}
b'_{v'}=1_{v' < W} \sum_{\substack{v'=vd \\ d | P(Z)/P(w)}} \mu(d) b_v.
\end{align*}
Since $ b_v$ is supported on $(v,P(Z))=1,$ we have $|b'_{v'}| \, \ll \, 1,$ and $b'_{v'}$ is supported on $(v',P(w))=1.$ Thus, by Proposition \ref{typeiestimate} we have 
\begin{align*}
\Sigma_I(\A)= \lambda \Sigma_I(\B) + \mathcal{O}\left( \lambda S(\B,X)\log^{-\delta/2} x \right).
\end{align*}
 
\textbf{Sums $\Sigma_{II}(\CC)$:}  We write
\begin{align*}
\Sigma_{II}(\CC) &= \sum_{\substack{u \in (U^{1-\epsilon},U] \\
v \in (V^{1-\epsilon},V]}} a_u b_v \sum_{k} (-1)^k \sum_{\substack{ w \leq p_1 < p_2 < \cdots <p_k < Z \\ vp_1 \cdots p_k \geq W}} S(\CC_{uvp_1 \cdots p_k},w) \\
&=\sum_{k} \frac{(-1)^k}{k!} \sum_{\substack{u \in (U^{1-\epsilon},U] \\
v \in (V^{1-\epsilon},V]}} \sum_{\substack{w \leq p_1,p_2, \cdots ,p_k < Z  \\ vp_1 \cdots p_k \geq W}} a_u b_v  S(\CC_{uvp_1 \cdots p_k},w) + \mathcal{O}(E(\CC)),
\end{align*}
where the sum over $k$ runs over $k \ll (\log \log x )^2,$ and the error term is for $\CC=\A$ bounded by (using the notation $\tau^{(4)}(n):= (1\ast1\ast1\ast1)(n)$)
\begin{align*}
k^2 \sum_{\substack{u \in (U^{1-\epsilon},U] \\
v \in (V^{1-\epsilon},V]}} \sum_{\substack{w \leq p_1,p_2, \cdots ,p_k < Z, \\ p_1 = p_2}}& a_u b_v  S(\A_{uvp_1 \cdots p_k},w) \\
 &\lesssim \sum_{\substack{u \in (U^{1-\epsilon},U] \\
v \in (V^{1-\epsilon},V]}}\sum_{\substack{w \leq p_1,p_2, \cdots ,p_k < Z, \\ p_1 = p_2}}\, \, \sum_{n, \, (n,P(w))=1} W_{\A}(uvp_1\cdots p_k n) \\
& \ll k! \sum_{n, w \leq p < Z} \tau^{(4)}(n) W_{\A}(p^2n) \\
& \ll k! y\log^C x \sum_{w \leq p < Z} p^{-2} \ll y \log^{-C} x,
\end{align*}
where we have applied Shiu's bound (Lemma \ref{shiu}) for the penultimate inequality. Similarly, we obtain a sufficient error term if $\CC=\B.$ Hence, we need to handle the sums
\begin{align*}
 \frac{1}{k!}\sum_{\substack{u \in (U^{1-\epsilon},U] \\
v \in (V^{1-\epsilon},V]}} \sum_{\substack{w \leq p_1,p_2, \cdots ,p_k < Z  \\ vp_1 \cdots p_k \geq W}} a_u b_v   S(\CC_{uvp_1 \cdots p_k},w).
\end{align*}
 To this end we note that for all $vp_1 \cdots p_k \geq W$ in the above sum, there exists exactly one $j \leq k$ such that
 \begin{align*}
  W \leq vp_1 \cdots p_j \leq x^{1/2-\delta}  \quad \text{and} \quad  vp_1 \cdots p_{j-1} < W.
\end{align*}  
 By (\ref{hash}) this implies that $unp_{j+1} \cdots p_k \in  [x^{1/2 - \gamma + \delta/2}, x^{1/2 -\delta/2}],$ where $n$ is the implicit variable in  $S(\CC_{uvp_1 \cdots p_k},w).$   Let $c_{j,k}, \, j= 1,\dots, k$  be any positive constants such that $c_{j,k} \leq 1/j!,$ $c_{j,k}c_{k-j,k} = 1/k!.$ Define then for any $j \leq k$
\begin{align*}
b'_{v',j,k} &:=   1_{v' \in [x^{1/2 - \gamma + \delta/2}, x^{1/2 -\delta/2}]} \,c_{j,k} \sum_{\substack{v'=vp_1 \cdots p_j \\  v \in (V^{1-\epsilon},V]\\ w \leq p_1, \dots, p_j < Z \\  W \leq vp_1 \cdots p_j \leq x^{1/2-\delta} , \\
vp_1 \cdots p_{j-1} < W }} b_v, \\
a'_{u',j,k} & :=   1_{u' \in [x^{1/2 - \gamma + \delta/2}, x^{1/2 -\delta/2}]} \sum_{\substack{u'=unr\\ u \in (U^{1-\epsilon},U]}} a_u 1_{(n,P(w))=1} \, c_{k-j,k} \sum_{\substack{r=p_1 \dots p_{k-j}, \\ w \leq p_1, \dots, p_{k-j }<Z}} 1.
\end{align*}
 Then (by uniqueness of the choice of $j$)
\begin{align*}
 \frac{1}{k!}\sum_{\substack{u \in (U^{1-\epsilon},U] \\
v \in (V^{1-\epsilon},V]}}  \sum_{\substack{w \leq p_1,p_2, \dots ,p_k < Z  \\ vp_1 \cdots p_k \geq W}} a_u b_v S(\CC_{uvp_1 \cdots p_k},w) = \sum_{j \leq k} \sum_{u',v'} a'_{u',j,k} b'_{v',j,k} W_{\CC}(u'v').
\end{align*}
Since $c_{j,k} \leq 1/j!,$ a trivial bound yields (using $(uv,P(Z))=1$)
\begin{align*}
|a'_{u',j,k}| \, \lesssim \, \tau_{w}(u')1_{(u',P(w))=1} \quad \text{and} \quad |b'_{v',j,k}| \, \lesssim \, \tau_{w}(v')1_{(v',P(w))=1}. 
\end{align*}
Hence, by Proposition \ref{typeiiestimate} we get
\begin{align*}
\sum_{u',v'} a'_{u',j,k} b'_{v',j,k} W_{\A}(u'v') = \lambda \sum_{u',v'} a'_{u',j,k} b'_{v',j,k} W_{\B}(u'v') + \mathcal{O}\left( \lambda S(\B,X)\log^{-\delta} x \right),
\end{align*}
which suffices since we sum over $j,k \lesssim 1.$
\end{proof}

\section{Buchstab decompositions} \label{buchs}
The aim of this section is to prove Proposition \ref{mainp} by obtaining a lower bound for $S(\A,x^{1/2-\gamma/2+\delta}).$  Our argument follows  similar lines as that described in Chapter 5 of Harman's book \cite{Har}. There are also many similarities with Jia's and Liu's decompositions in \cite{JL}. 

The general idea of Harman's sieve is to use Buchstab's identity to decompose the sum $S(\CC,x^{1/2-\gamma/2+\delta})$  (in parallel for $\CC=\A$ and $\CC=\B$) into a sum of the form $\sum_k \epsilon_k S_k(\CC),$ where $\epsilon_k \in \{-1,1\},$ and  $S_k(\CC) \geq 0$ are sums over almost-primes.  Since we are interested in a lower bound, for $\CC=\A$ we can insert the trivial estimate $S_k(\A) \geq 0$ for any $k$ such that the sign $\epsilon_k =1;$ these sums are said to be discarded. For the remaining $k$ we will obtain an asymptotic formula by using Propositions \ref{typeiiestimate} and \ref{ft}. That is, if $\mathcal{K}$ is the set of indices that are discarded, then (for  $\lambda:= y \log^{100} x /x$) 
\begin{align*}
S(\A,x^{1/2-\gamma/2+\delta})&= \sum_k \epsilon_k S_k(\A) \geq \sum_{k \notin \mathcal{K}} \epsilon_k S_k(\A)  \\
&\sim \sum_{k \notin \mathcal{K}} \epsilon_k \lambda S_k(\B) =  \lambda S(\B,x^{1/2-\gamma/2+\delta}) -  \lambda \sum_{k \in \mathcal{K}}  S_k(\B).
\end{align*} 
We are successful if we can then obtain a bound $\sum_{k \in \mathcal{K}} S_k(\B) \leq (1-\mathfrak{C}(\gamma)) S(\B,x^{1/2-\gamma/2+\delta})$ for some $\mathfrak{C}(\gamma) > 0;$ obtaining this ultimately determines the exponent $\gamma$ in Proposition \ref{mainp}. 

For this last step we require two lemmata, which allow us to transform sums over $W_\B$ into so-called Buchstab integrals that can be estimated numerically. Let $\omega(u)$ denote the Buchstab function (cf. Chapter I of Harman's book \cite{Har}, for instance), so that by the Prime Number Theorem for $X^{\epsilon} < Z < X,$ $X \log^{-C} X \ll Y \ll X$
\begin{align*}
\sum_{X< n \leq X+Y} 1_{(n,P(z))=1} = (1+o(1)) \omega \left(\frac{\log X}{\log Z} \right) \frac{Y}{\log Z}
\end{align*}
(the same argument as in  Chapter I of \cite{Har} gives the result for the slightly shorter intervals of length $Y$). Note that for $1< u \leq 2$ we have $\omega(u)=1/u.$ In the numerical computations we will use the following standard upper bound (cf. Lemma 4 and the discussion below that in \cite{JL}, for instance) for the Buchstab function 
\begin{align*}
\omega(u) \, \leq \begin{cases} 0, &u < 1 \\
1/u, & 1 \leq u < 2 \\
(1+\log(u-1))/u, &2 \leq u < 3 \\
0.5644, &  3 \leq u < 4 \\
0.5617, & u \geq 4.
\end{cases}
\end{align*}
In the two lemmata below we assume that the range $\mathcal{U}\subset [x^{\epsilon},x]^{k}$ is sufficiently well-behaved, e.g. an intersection of sets of the type $\{ \boldsymbol{u}: u_i < u_j \}$ or $\{\boldsymbol{u}: V <  f(u_1, \dots,u_k) < W\}$ for some polynomial $f$ and some fixed $V,W.$

\begin{lemma} \label{bi} Let $X=x^{1/2-\gamma/2 +\delta}$ and let $\mathcal{U} \subset [x^{\epsilon},x]^{k}.$ Then
\begin{align*}
\frac{\log^{100} x}{x}\sum_{(p_1, \dots , p_k) \in \mathcal{U}} S(\B_{p_1, \dots, p_k},p_k) = S(\B,X)(1+o(1))(1-\gamma)\int \omega (\boldsymbol{\alpha }) \frac{d\alpha_1 \cdots d\alpha_k}{\alpha_1\cdots\alpha_{k-1}\alpha_k^2},
\end{align*}
where the integral is over the range 
\begin{align*}
\{(\boldsymbol{\alpha}: \, (x^{\alpha_1}, \dots, x^{\alpha_k}) \in \mathcal{U})\}
\end{align*}
 and $\omega(\boldsymbol{\alpha})= \omega(\alpha_1,\dots,\alpha_k):= \omega((1-\gamma-\alpha_1-\cdots -\alpha_k)/\alpha_k)$
\end{lemma}
\begin{proof}
The left-hand side is by the Prime Number Theorem
\begin{align*}
&\frac{\log^{100} x}{x}\sum_{(p_1, \dots , p_k) \in \mathcal{U}} \sum_q 1_{(q,P(p_k))=1} W_{\B}(p_1 \cdots  p_k q ) \\
&= (2+o(1))\Omega  \sum_{(p_1, \dots , p_k) \in \mathcal{U}} \frac{1}{p_1\cdots p_k \log p_k} \omega \left( \frac{\log(x^{1-\gamma+o(1)}/(p_1\cdots p_k))}{\log p_k} \right) \\
&= (2+o(1))\Omega\hspace{-5pt}  \sum_{(n_1,\dots,n_k ) \in \mathcal{U}} \frac{1}{n_1\cdots n_k (\log n_1) \dots (\log n_{k-1} )\log^2 n_k} \omega \left( \frac{\log(x^{1-\gamma+o(1)}/(n_1\cdots n_k))}{\log n_k} \right) \\
&= (2+o(1))\Omega  \int_{\mathcal{U}}  \omega \left( \frac{\log(x^{1-\gamma}/(u_1\cdots u_k))}{\log u_k} \right)   \frac{du_1\cdots du_k}{u_1\cdots u_k (\log u_1) \dots (\log u_{k-1} )\log^2 u_k}\\
&= \frac{(2+o(1))\Omega}{\log x}  \int \omega (\boldsymbol{\alpha }) \frac{d\alpha_1 \cdots d\alpha_k}{\alpha_1\cdots\alpha_{k-1}\alpha_k^2} 
\end{align*}
by the change of variables $u_j=x^{\alpha_j}$ The claim now follows by the definition (\ref{omega}) of $\Omega$.
\end{proof}
\begin{remark}
Similarly as in \cite{JL}, we call the factor $(1-\gamma) \int d \boldsymbol{\alpha}$ the deficiency of the corresponding sum. By the lemma it is up to the factor $1+o(1)$ the ratio of the sum and $S(\B,X)$.
\end{remark}
We also need the following variant of the above lemma, which will occur as the result of using role reversals. 
\begin{lemma}\label{roleb} Let $X=x^{1/2-\gamma/2 +\delta},$ $Z=x^{\gamma -2\delta},$ and let $\mathcal{U} \subset [x^{\epsilon},x)^4.$ Then 
\begin{align*}
\frac{\log^{100} x}{x} \hspace{-10pt} \sum_{\substack{(q,m,p_2,p_3) \in \mathcal{U} \\  (m,P(q))=1 }}  \hspace{-10pt} S(\B_{qmp_2p_3},Z) = S(\B,X)(1+o(1)) \cdot \frac{1-\gamma}{\gamma}\int \omega_2 (\boldsymbol{\alpha }) \frac{d\alpha_0 d\alpha_1d\alpha_2d\alpha_3 }{\alpha_0^2 \alpha_2 \alpha_3},
\end{align*}
where the integral is over $\{\boldsymbol{\alpha}: \, (x^{\alpha_0},x^{\alpha_1},x^{\alpha_2},x^{\alpha_3}) \in \mathcal{U} \},$ and 
\begin{align*}
\omega_2 (\boldsymbol{\alpha }):=\omega((1-\gamma-\alpha_0-\alpha_1-\alpha_2-\alpha_3)/\gamma)\omega(\alpha_1/\alpha_0).
\end{align*}
\end{lemma}
\begin{proof} The left-hand side equals by the Prime Number Theorem
\begin{align*}
&\frac{\log^{100} x}{x}\sum_{\substack{q,m,p_2,p_3,n \\ (q,m,p_2,p_3) \in \mathcal{U} \\ (n,P(Z))=1, (m,P(q))=1 }} W_{\B}(qmp_2p_3n)\\
&= \frac{(2+o(1))\Omega}{\log Z}\sum_{(q,m,p_2,p_3) \in \mathcal{U}} \frac{1_{(m,P(q))=1}}{qp_2p_3m} \omega \left( \frac{\log(x^{1-\gamma+o(1)}/(qmp_2p_3))}{\log Z} \right) \\
&= \frac{(2+o(1))\Omega}{\gamma \log x}\sum_{(q,m,p_2,p_3) \in \mathcal{U}} \frac{1}{qp_2p_3m\log q} \omega \left( \frac{\log m}{\log q} \right) \omega \left( \frac{\log(x^{1-\gamma}/(qmp_2p_3))}{\log Z} \right) \\
& \frac{(2+o(1))\Omega}{(1-\gamma)\log x} \cdot \frac{1-\gamma}{\gamma}\int \omega_2 (\boldsymbol{\alpha }) \frac{d\alpha_1d\alpha_2d\alpha_3 d\alpha_4}{\alpha_1 \alpha_2\alpha_4^2}.
\end{align*}
The claim now follows by the definition  (\ref{omega}) of $\Omega$.
\end{proof}
\begin{remark}
In this instance we call the factor $\frac{1-\gamma}{\gamma} \int d \boldsymbol{\alpha}$ the deficiency of the corresponding sum. By the lemma it is the ratio of the sum and $S(\B,X)$, up to the factor $1+o(1)$. 
\end{remark}

We are now ready for the Buchstab decompositions. We fix $\gamma=1/19,$ and define 
\begin{align*}
X:= x^{1/2-\gamma/2 + \delta}, \quad \quad Z:=x^{\gamma -2 \delta}, \quad \quad W:= x^{1/2-\gamma + \delta} ,
\end{align*}
 and write by Buchstab's identity (for $\CC = \A$ or $\CC = \B$)
\begin{align*}
S(\CC, X) &= S(\mathcal{C}, Z) - \sum_{Z \leq p < X} S(\CC_p, p) \\
& =  S(\CC, Z) - \sum_{W \leq p < X} S(\CC_p, p)  - \sum_{Z \leq p < W} S(\CC_p, Z) + \sum_{\substack{Z < p_2 < p_1 < W \\ p_1p_2^2 < X^2 }} S(\CC_{p_1p_2}, p_2) \\
&=: S_1(\CC) -S_2(\CC)-S_3(\CC)+S_4(\CC).
\end{align*}

\subsection{Sum $S_1(\CC)$}
Applying Proposition \ref{ft} with $U=V=1,$ $a_u=b_u=1_{u=1}$ we obtain $S_1(\A) = \lambda S_1(\B) +  \mathcal{O}\left( \lambda S(\B,X)\log^{-\delta/2} x \right).$

 \subsection{Sum $S_2(\CC)$}
We have (since $S(\mathcal{A}_p, p) \equiv 0$ for $p \geq X$)
\begin{align*}
S_2(\A) &= \sum_{ p \geq W } S(\mathcal{A}_p, p) = \sum_{\substack{p,q \\
q > p \geq W }} W_{\A}(pq) = \frac{1}{2} \sum_{\substack{p,q \\
 p, q \in [W,x^{1/2-\delta/2}]}} W_{\A}(pq) + \mathcal{O}(y\log^{-C} x) \\ & = \lambda S_2(\B) +  \mathcal{O}\left( \lambda S(\B,X)\log^{-\delta/2} x \right),
\end{align*}
by Proposition \ref{typeiiestimate}.

\subsection{Sum $S_3(\CC)$}
Dividing the sums $S_3(\CC)$ into $\mathcal{O}_{\epsilon}(1)$ sums such that $p=v \in (V^{1-\epsilon},V],$ and applying Proposition \ref{ft} with $U=1,$ $a_u=1_{u=1},$ and $b_v=1_{v \in \mathbb{P}}1_{Z \leq v < W},$ we obtain  $S_3(\A) = \lambda S_3(\B) +  \mathcal{O}\left( \lambda S(\B,X)\log^{-\delta/2} x \right).$

\subsection{Sum $S_4(\CC)$}
We write $S_4(\CC) = S_5(\CC) + S_6(\CC)+S_7(\CC)+S_8(\CC),$ where for $V := x^{1/2-2\gamma+ 3\delta}$ 
\begin{align*}
S_5(\CC) :=  \sum_{\substack{Z < p_2 < p_1 < V \\ p_1p_2 < W }} S(\CC_{p_1p_2}, p_2), \quad \quad
S_6 (\CC):= \sum_{\substack{Z \leq p_2 < p_1 <W \\
W \leq p_1p_2 \leq x^{1/2-\delta}}}  S(\CC_{p_1p_2}, p_2) \\
S_7(\CC) := \sum_{\substack{x^{1/4-\delta/2} \leq p_1 < W \\ Z \leq p_2 < x^{1/4-10\delta} \\  p_1p_2 > x^{1/2-\delta} }}  S(\CC_{p_1p_2}, p_2) , \quad \quad S_8(\CC) := \sum_{\substack{x^{1/4-\delta/2} \leq p_1 < W \\  x^{1/4-10\delta} \leq p_2 < p_1 \\  x^{1/2-\delta}p_2 < p_1p_2^2 \leq X^2  }}  S(\CC_{p_1p_2}, p_2).
\end{align*}
We estimate each sum separately.

\subsection{Sum $S_5(\CC)$}
Two applications of Buchstab's identity yields
\begin{align*}
S_5(\CC)& =   \sum_{\substack{Z < p_2 < p_1 < V \\ p_1p_2 < W }} S(\CC_{p_1p_2}, Z)
- \sum_{\substack{Z < p_3 < p_2 < p_1 < V \\ p_1p_2 < W, \, \, p_1p_2p_3^2 < X^2}} S(\CC_{p_1p_2p_3}, Z) \\
 & \hspace{80pt}+  \sum_{\substack{Z <p_4< p_3 < p_2 < p_1 < V \\ p_1p_2 < W, \, \, p_1p_2p_3^2 < X^2,  \, \, p_1p_2p_3p_4^2 < X^2}} S(\CC_{p_1p_2p_3p_4}, p_4) \\
 &=:S_{5,1}(\CC)- S_{5,2}(\CC)+ S_{5,3}(\CC),
\end{align*}
say.
\subsubsection{Sum $S_{5,1}(\CC)$}
Using Proposition \ref{ft} with $u=1$ and $v=p_1p_2,$ we obtain $S_{5,1}(\A) = \lambda S_{5,1}(\B) +  \mathcal{O}\left( \lambda S(\B,X)\log^{-\delta/2} x \right).$

\subsubsection{Sum $S_{5,2}(\CC)$}
Note that $p_1p_2 \leq W,$ $p_3<p_2<p_1$  implies $p_3 \leq W^{1/2} < x^{1/4-20\delta}.$  Thus, we wish to apply Proposition \ref{ft} with $v=p_1p_2,$ $u=p_3$ but we have cross-conditions $p_3 < p_2,$ $p_1p_2p_3^2 < X^2$ that need to be removed. We do this by dividing the ranges into shorter ones, that is,
\begin{align*}
S_{5,2}(\CC) = \sum_{V_1,V_2,V_3}\sum_{\substack{Z < p_3 < p_2 < p_1 < V \\ p_1p_2 < W, \, \, p_1p_2p_3^2 < X^2 \\ p_j \in (V_j^{1-\epsilon},V_j], \, \, j \in \{1,2,3\}}} S(\CC_{p_1p_2p_3}, Z),
\end{align*}
where the sum over $V_j$ runs over $V_j$ of the form $V^{(1-\epsilon)^g},$ $g \in \Z$ such that $V_3 \leq V_2,$ and $(V_1V_2V_3^2)^{1-\epsilon} \leq X^2$ (that is, each condition is loosened if necessary but at most by a factor of $x^{\mathcal{O}(\delta)}$) Note that the overall sign of sums $S_{5,2}(\CC)$ is negative, so that we only require an upper bound for $S_{5,2}(\A).$ Thus, we can drop the unwanted cross-conditions for $\CC=\A$ so that by Proposition \ref{ft} (since the inner sum is non-empty only if $V_3 \leq   x^{1/4-10\delta}$)
 \begin{align*}
 S_{5,2}(\A) & \leq \sum_{\substack{V_1,V_2,V_3 \\ V_3 \leq V_2 \\ (V_1V_2V_3^2)^{1-\epsilon} \leq X^2}}\sum_{\substack{Z < p_3 < V, \, Z< p_2 < p_1 < V \\ p_1p_2 < W  \\ p_j \in (V_j^{1-\epsilon},V_j], \, \, j \in \{1,2,3\}}} S(\A_{p_1p_2p_3}, Z) \\
 & = \lambda\sum_{\substack{V_1,V_2,V_3  \\ V_3 \leq V_2 \\ (V_1V_2V_3^2)^{1-\epsilon} \leq X^2}}\sum_{\substack{Z < p_3 < V, \, Z< p_2 < p_1 < V \\ p_1p_2 < W  \\ p_j \in (V_j^{1-\epsilon},V_j], \, \, j \in \{1,2,3\}}} S(\B_{p_1p_2p_3}, Z) + \mathcal{O}\left( \lambda S(\B,X)\log^{-\delta/2} x \right)  \\
 &=  \lambda S_{5,2}(\B)  + E(\B) + \mathcal{O}\left( \lambda S(\B,X)\log^{-\delta/2} x \right).
 \end{align*}
Here
\begin{align*}
E(\B) =  \lambda \sum_{\substack{V_1,V_2,V_3 \\ V_3 \leq V_2 \\ (V_1V_2V_3^2)^{1-\epsilon} \leq X^2}} \sum_{\substack{Z < p_3 < V, \, Z < p_2 < p_1 < V \\ p_1p_2 < W, \\ p_3 \geq p_2  \,\, \text{or} \, \, p_1p_2p_3^2 \geq X^2 \\ p_j \in (V_j^{1-\epsilon},V_j], \, \, j \in \{1,2,3\}}} S(\B_{p_1p_2p_3}, Z) \ll \delta \lambda S(\B,X),
\end{align*} 
since the conditions in the sum over $V_1,V_2,V_3$ imply that in the inner sum always either $p_3=p_2x^{o(1)}$ or $p_1p_2p_3^2 = X^{2+o(1)},$ so that by Lemma \ref{bi} the  Buchstab integral correspondig to the sum $E(\B)$ is of size $\ll \delta$  (thus, the deficiency of $E(\B)$ is $\ll \delta$). Hence, $ S_{5,2}(\A) \leq (1+\mathcal{O}(\delta)) \lambda S_{5,2}(\B).$
 
We can also obtain a lower bound by a similar argument; instead of dropping the cross-conditions $p_3 < p_2,$ $p_1p_2p_3^2 < X^2$ for $S_{5,2}(\A),$ we divide the sum over $V_j$ into two parts $\sum_{V_1,V_2,V_3} =\sum_{V_1,V_2,V_3}^{(1)} + \sum_{V_1,V_2,V_3}^{(2)},$ where the first sum runs over $V_j$ such that $V_3 < V_2^{1-\epsilon}$ and $V_1V_2V_3^2 < X^2,$ and the second sum over the complement. In the range of $\sum_{V_1,V_2,V_3}^{(1)}$ we always have
\begin{align*}
\sum_{\substack{Z < p_3 < p_2 < p_1 < V \\ p_1p_2 < W, \, \, p_1p_2p_3^2 < X^2 \\ p_j \in (V_j^{1-\epsilon},V_j], \, \, j \in \{1,2,3\}}} S(\CC_{p_1p_2p_3}, Z) = \sum_{\substack{Z < p_3 < V, \, Z< p_2 < p_1 < V \\ p_1p_2 < W  \\ p_j \in (V_j^{1-\epsilon},V_j], \, \, j \in \{1,2,3\}}} S(\CC_{p_1p_2p_3}, Z),
\end{align*}
so that we may apply Proposition \ref{ft}. In the second sum $\sum_{V_1,V_2,V_3}^{(2)}$ we estimate $S(\A_{p_1p_2p_3}, Z)$ trivially by $0$ from below, so that we obtain
\begin{align*}
 S_{5,2}(\A) \geq  \lambda S_{5,2}(\B) - E'(\B) + \mathcal{O}\left( \lambda S(\B,X)\log^{-\delta/2} x \right).
\end{align*}
 The error term 
\begin{align*}
 E'(\B) = \lambda \sideset{}{^{(2)}} \sum_{V_1,V_2,V_3} \sum_{\substack{Z < p_3 < p_2 < p_1 < V \\ p_1p_2 < W, \, \, p_1p_2p_3^2 < X^2 \\ p_j \in (V_j^{1-\epsilon},V_j], \, \, j \in \{1,2,3\}}} S(\B_{p_1p_2p_3}, Z)
\end{align*} 
 is again  a sum with $p_3=p_2x^{o(1)}$ or $p_1p_2p_3^2 = X^{2+o(1)},$ corresponding to a Buchstab integral of size $\ll \delta$, which yields $S_{5,2}(\A) \geq (1-\mathcal{O}(\delta)) \lambda S_{5,2}(\B).$  We will need this version later, when we have to remove cross-conditions in a sum with a positive overall sign.
 
By combining the above, we have  $ S_{5,2}(\A) =(1+\mathcal{O}(\delta)) \lambda S_{5,2}(\B).$
 
\subsubsection{Sum $S_{5,3}(\CC)$}
We split the sum into three parts depending on the size of $p_1p_2p_3p_4$
\begin{align*}
S_{5,3}(\CC) = \hspace{-10pt} \sum_{\substack{(p_1,p_2,p_3,p_4)\in \mathcal{U}(5,3,1)} } \hspace{-6pt}  S(\CC_{p_1p_2p_3p_4}, p_4) &+\hspace{-10pt} \sum_{\substack{(p_1,p_2,p_3,p_4)\in \mathcal{U}(5,3,2)} } \hspace{-6pt} S(\CC_{p_1p_2p_3p_4}, p_4) \\
 &\hspace{30pt}+\hspace{-10pt}  \sum_{\substack{(p_1,p_2,p_3,p_4)\in \mathcal{U}(5,3,3)} } \hspace{-6pt} S(\CC_{p_1p_2p_3p_4}, p_4)
\end{align*}
where
\begin{align*}
\mathcal{U}(5,3)&:= \{(p_1,p_2,p_3,p_4):  \, \, Z <p_4< p_3 < p_2 < p_1 < V,  \,\, p_1p_2 < W,  \\  &\hspace{230pt} p_1p_2p_3^2 < X^2,  \, \,  p_1p_2p_3p_4^2 < X^2 \} \\
\mathcal{U}(5,3,1)&:= \mathcal{U}(5,3) \cap \{(p_1,p_2,p_3,p_4): \, \, p_1p_2p_3p_4 < W\} \\
\mathcal{U}(5,3,2)&:= \mathcal{U}(5,3) \cap \{(p_1,p_2,p_3,p_4): \, \,W \leq p_1p_2p_3p_4 \leq x^{1/2-\delta} \}  \\
\mathcal{U}(5,3,3)&:= \mathcal{U}(5,3) \cap \{(p_1,p_2,p_3,p_4): \, \, p_1p_2p_3p_4 > x^{1/2 -\delta}\}.
\end{align*}

\textbf{Sum over $\mathcal{U}(5,3,2)$:} We have the variable $p_1p_2p_3p_4$ in the Type II range $[W,x^{1/2-\delta}]$, so that we may apply Proposition \ref{typeiiestimate}; we just need to remove the cross-condition $(n,P(p_4))=1$ for the implicit variable in $S(\CC_{p_1p_2p_3p_4}, p_4).$ To this end, write
 $n=q_1\cdots q_k$ so that
\begin{align*}
\sum_{(p_1,p_2,p_3,p_4) \in \mathcal{U}(5,3,2)} \hspace{-5pt} S(\CC_{p_1p_2p_3p_4}, p_4) = \sum_{k \ll 1}   \sum_{(p_1,p_2,p_3,p_4) \in \mathcal{U}(5,3,2)} \sum_{\substack{ p_4 < q_1 \leq  q_2 \leq \dots \leq  q_k }} \hspace{-5pt} W_\CC(p_1p_2p_3p_4q_1 \cdots q_k).
\end{align*}
Similarly as in the above for the sum $S_{5,2}(\CC),$ we divide sums over $p_4$ and $q_1$ into shorter sums, which yields
\begin{align*}
\sum_{k \ll 1} \sum_{\substack{V,U \\ V^{1-\epsilon} \leq U}} \sum_{\substack{(p_1,p_2,p_3,p_4) \in \mathcal{U}(5,3,2)\\
p_4 \in (V^{1-\epsilon},V]}}  \sum_{\substack{ p_4 < q_1 \leq  q_2 \leq \dots \leq  q_k  \\ q_1 \in (U^{1-\epsilon}, U]}} W_\CC(p_1p_2p_3p_4q_1 \cdots q_k).
\end{align*}
Applying the argument used with the sum $S_{5,2}(\CC)$ to handle the cross-conditions, combined with Proposition \ref{typeiiestimate}, we obtain
\begin{align*}
\sum_{(p_1,p_2,p_3,p_4) \in \mathcal{U}(5,3,2)} S(\A_{p_1p_2p_3p_4}, p_4) = \lambda \sum_{(p_1,p_2,p_3,p_4) \in \mathcal{U}(5,3,2)} S(\B_{p_1p_2p_3p_4}, p_4) + \mathcal{O}\left( \delta \lambda S(\B,X) \right).
\end{align*}

\textbf{Sum over $\mathcal{U}(5,3,1):$} Here we apply Buchstab's identity twice, which yields
\begin{align*}
\sum_{\substack{(p_1,p_2,p_3,p_4)\in \mathcal{U}(5,3,1)} } S(\A_{p_1p_2p_3p_4}, Z) - \sum_{\substack{(p_1,p_2,p_3,p_4)\in \mathcal{U}(5,3,1), \\ Z \leq p_5 < p_4, \, p_1p_2p_3p_4p_5^2 < X^2} } S(\A_{p_1p_2p_3p_4p_5}, Z)  \\
+ \sum_{\substack{(p_1,p_2,p_3,p_4)\in \mathcal{U}(5,3,1), \\ Z \leq p_6 < p_5 < p_4, \, p_1p_2p_3p_4p_5^2 < X^2, \\p_1p_2p_3p_4p_5p_6^2 < X^2} } S(\A_{p_1p_2p_3p_4p_5p_6}, p_6).
\end{align*}
The first two sums have asymptotic formulas by  Proposition \ref{ft}, since $p_1p_2p_3p_4 < W$ and $p_5 < (p_1p_2p_3p_4)^{1/4} < x^{1/4-20\delta}$ (the cross-conditions can be handled by the discussion of the sum $S_{5,2}(\CC)$ in the above). In the third sum we take out the range where at least one of the products $\prod_{j \in I}p_j$ (where $I \subseteq \{1,2,\dots,6\}$) is in the Type II range  $[W,x^{1/2-\delta}]$ (these can be dealt with by a similar argument as for the sum over $\mathcal{U}(5,3,2)$). We must discard the rest of the sum, giving us a deficiency (cf. Lemma \ref{bi})
\begin{align*}
\mathcal{O} (\delta) +(1-\gamma)\int f_{5,3,1} (\boldsymbol{\alpha})\omega(\boldsymbol{\alpha})\frac{d\alpha_1d\alpha_2d\alpha_3d\alpha_4d\alpha_5d\alpha_6}{\alpha_1\alpha_2\alpha_3\alpha_4\alpha_5\alpha_6^2} < 0.006493,
\end{align*}
where $\omega(\boldsymbol{\alpha})= \omega(\alpha_1,\dots,\alpha_6):= \omega((1-\gamma-\alpha_1-\cdots -\alpha_k)/\alpha_6),$ and $f_{5,3,1} (\boldsymbol{\alpha})$ is the characteristic function of the six-dimensional set (the various $\delta$'s can be dropped, with an error $\ll \delta$)
\begin{align*}
\{\boldsymbol{\alpha}: & \, (x^{\alpha_1}, x^{\alpha_2},x^{\alpha_3},x^{\alpha_4}) \in \mathcal{U}(5,3,1), \, \gamma \leq \alpha_6 < \alpha_5 < \alpha_4, \,  \alpha_1 + \cdots +\alpha_4 +2 \alpha_5 < 1-\gamma ,  \\& \alpha_1 + \cdots + \alpha_5 +2 \alpha_6 < 1-\gamma, \, \sum_{j \in I} \alpha_j \notin [1/2-\gamma, 1/2]  \quad \text{for every} \, \, I \subseteq \{1,2,\dots,6\} \}.
\end{align*}
For all of the codes for computing upper bounds for the numerical integrals, see the codepad links at the end of this section.

\textbf{Sum over $\mathcal{U}(5,3,3):$}  We divide the range $\mathcal{U}(5,3,3)$ into three parts  $\mathcal{U}(5,3,3,1) \cup \mathcal{U}(5,3,3,2) \cup \mathcal{U}(5,3,3,3),$  where $\mathcal{U}(5,3,3,1) := \mathcal{U}(5,3,3) \cap \{p_2p_3p_4 < W\},$ and $\mathcal{U}(5,3,3,3):=\mathcal{U}(5,3,3) \cap \{p_2p_3p_4 > x^{1/2-\delta}\},$ and $\mathcal{U}(5,3,3,2)$ is the remaining part which can be handled as a Type II sum, since $p_2p_3p_4 \in [W,x^{1/2-\delta}]$ (the cross-condition $(n,P(p_4))=1$ is again dealt with by a similar argument as with the sum over $\mathcal{U}(5,3,2)$).

For $\mathcal{U}(5,3,3,1)$ we use Buchstab's identity in the `upwards' direction (this is called Process II in Jia and Liu \cite{JL}, page 27)
\begin{align*}
 S(\mathcal{A}_{p_1p_2p_3p_4}, p_4) =  S\left(\mathcal{A}_{p_1p_2p_3p_4}, \frac{X}{(p_1p_2p_3p_4)^{1/2} } \right) + \sum_{p_4 \leq p_5 < \frac{X}{(p_1p_2p_3p_4)^{1/2} }  }  S(\mathcal{A}_{p_1p_2p_3p_4p_5}, p_5).
\end{align*}
By (\ref{hash}) the implicit variable in the first sum is of size $x^{1-\gamma +o(1)}/(p_1p_2p_3p_4).$ Thus, we have a four dimensional sum over primes and a five dimensional sum over $p_5$-almost-primes. In each sum  we take out ranges with a Type II variable, and discard the rest. This gives us deficiencies
\begin{align*}
 \mathcal{O}(\delta) +(1-\gamma)\int f_{5,3,3,1}(\boldsymbol{\alpha})\frac{d\alpha_1d\alpha_2d\alpha_3d\alpha_4}{(1-\gamma-\alpha_1-\alpha_2-\alpha_3-\alpha_4)\alpha_1\alpha_2\alpha_3\alpha_4}<0.1139225,
\end{align*}
and
\begin{align*}
\mathcal{O}(\delta)+ (1-\gamma)\int g_{5,3,3,1}(\boldsymbol{\alpha}) \omega(\boldsymbol{\alpha}) \frac{d\alpha_1d\alpha_2d\alpha_3d\alpha_4d\alpha_5}{\alpha_1\alpha_2\alpha_3\alpha_4\alpha_5^2}<0.0450231.
\end{align*}
Here $f_{5,3,3,1}$ is the characteristic function of the four-dimensional set
\begin{align*}
\mathcal{V}(5,3,3,1)=\{\boldsymbol{\alpha}: & \, (x^{\alpha_1}, x^{\alpha_2},x^{\alpha_3},x^{\alpha_4}) \in \mathcal{U}(5,3,3,1), \\ &\hspace{60pt} \sum_{j \in I} \alpha_j \notin [1/2-\gamma, 1/2]  \quad \text{for every} \, \, I \subseteq \{1,2,3,4\} \},
\end{align*}
and $g_{5,3,3,1}$ is the characteristic function of the five-dimensional set
\begin{align*}
\{\boldsymbol{\alpha}: & \, (x^{\alpha_1}, x^{\alpha_2},x^{\alpha_3},x^{\alpha_4}) \in \mathcal{V}(5,3,3,1), \\ &\hspace{60pt} \sum_{j \in I} \alpha_j \notin [1/2-\gamma, 1/2]  \quad \text{for every} \, \, I \subseteq \{1,2,3,4,5\} \}.
\end{align*}

We discard the sum over $\mathcal{U}(5,3,3,3)$ (no combination of the variables is in the Type II range),  which gives a deficiency
\begin{align*}
(1-\gamma)\int f_{\mathcal{V}_2} (\boldsymbol{\alpha})\omega(\boldsymbol{\alpha})\frac{d\alpha_1d\alpha_2d\alpha_3d\alpha_4}{\alpha_1\alpha_2\alpha_3\alpha_4^2}< 0.014837,
\end{align*}
where $f_{\mathcal{V}_2}$ is the characteristic function of $\{\boldsymbol{\alpha}: \, (x^{\alpha_1}, x^{\alpha_2},x^{\alpha_3},x^{\alpha_4}) \in \mathcal{U}(5,3,3,3) \}.$

\subsubsection{Deficiency of $S_5(\CC)$} Combining the above, the deficiency of $S_5(\CC)$ is $< 0.1802756.$

\subsection{Sum $S_6(\CC)$}
This is almost already a Type II sum, we just need to deal with the cross-condition $(n,P(p_2))=1.$ Applying the argument used with the sum over $\mathcal{U}(5,3,2)$, we obtain 
\begin{align*}
S_6(\A) = \lambda S_6(\B) +  \mathcal{O}\left( \delta \lambda S(\B,X) \right).
\end{align*}

\subsection{Sum $S_7(\CC)$} We first divide $S_7(\CC)$ into two parts (the exponent 0.36 is optimized by computer for using role reversal in the first sum)

\begin{align*}
S_7(\CC)&= \sum_{\substack{x^{1/4-\delta/2} \leq p_1 < W, \\ Z \leq p_2 < x^{1/4-10\delta} \\  p_1p_2 > x^{1/2-\delta}, \, \, \sqrt{p_1}p_2 < x^{0.36} }}  S(\CC_{p_1p_2}, p_2)  + \sum_{\substack{x^{1/4-\delta/2} \leq p_1 < W, \\ Z \leq p_2 < x^{1/4-10\delta} \\  p_1p_2 > x^{1/2-\delta}, \, \, \sqrt{p_1}p_2 \geq x^{0.36} }}  S(\CC_{p_1p_2}, p_2) \\[10pt]
&=: S_{7,1}(\CC) + S_{7,2}(\CC)
\end{align*}

\subsubsection{Sum $S_{7,1}(\CC)$}  We apply Buchstab's identity twice to obtain
\begin{align*}
S_{7,1}(\CC)=\sum_{\substack{x^{1/4-\delta/2} \leq p_1 < W, \\ Z \leq p_2 < x^{1/4-10\delta} \\  p_1p_2 > x^{1/2-\delta},\, \, \sqrt{p_1}p_2 < x^{0.36} }}  S(\CC_{p_1p_2}, Z) - \sum_{\substack{x^{1/4-\delta/2} \leq p_1 < W, \\ Z \leq p_3< p_2 < x^{1/4-10\delta} \\  p_1p_2 > x^{1/2-\delta}, \, \, \sqrt{p_1}p_2 < x^{0.36}, \, \,  p_1p_2p_3^2 < X^2 }} \hspace{-20pt} S(\CC_{p_1p_2p_3}, Z) \\
+  \sum_{\substack{x^{1/4-\delta/2} \leq p_1 < W, \\ Z \leq p_4< p_3< p_2 < x^{1/4-10\delta} \\  p_1p_2 > x^{1/2-\delta}, \, \, \sqrt{p_1}p_2 < x^{0.36}, \\  p_1p_2p_3^2 < X^2, \, \,  p_1p_2p_3p_4^2 < X^2 }} \hspace{-20pt}  S(\CC_{p_1p_2p_3p_4}, p_4)=: S_{7,1,1}(\CC) - S_{7,1,2}(\CC)+S_{7,1,3}(\CC).
\end{align*}

\textbf{Sum $S_{7,1,1}(\CC)$:} By Proposition \ref{ft} with $u=p_2$, $v=p_1$ (the cross-conditions between $p_1 $ and $p_2$ can be removed by the same argument as with the sum $S_{5,2}(\CC)$)
\begin{align*}
S_{7,1,1}(\A) = \lambda S_{7,1,1}(\B)  +  \mathcal{O}\left( \delta \lambda S(\B,X) \right).
\end{align*}

\textbf{Sum $S_{7,1,2}(\CC)$:} The parts with $p_1p_3 \leq x^{1/2-\delta/2}$ or $p_2p_3 \leq x^{1/4-20\delta}$ have an asymptotic formula by Proposition \ref{ft} (again using the discussion of  $S_{5,2}(\CC)$ to remove cross-conditions). Write 
\begin{align*}
\mathcal{U}(7,1,2):= \{(p_1,p_2,p_3): x^{1/4-\delta/2} \leq p_1 < W, Z \leq p_3< p_2 < x^{1/4-10\delta}, \, \, p_1p_3 > x^{1/2-\delta/2}, \, \, \\ \sqrt{p_1}p_2 < x^{0.36}, \, \,  p_1p_2p_3^2 < X^2   , \, p_2p_3 > x^{1/4-20\delta}\}
\end{align*}
 for the complementing region. Here we apply the role reversal device; we write out the implicit sum and apply Buchstab's identity to the sum over $p_1,$ that is,
\begin{align*}
\sum_{(p_1,p_2,p_3) \in \mathcal{U}(7,1,2)} S(\A_{p_1p_2p_3},Z)&= \sum_{\substack{p_1,p_2,p_3, n \\(p_1,p_2,p_3) \in \mathcal{U}(7,1,2) \\
(n,P(Z))=1}} W_{\CC}(p_1p_2p_3n) \\
&= \hspace{-20pt} \sum_{\substack{m,p_2,p_3, n  \\(m,p_2,p_3) \in \mathcal{U}(7,1,2) \\
(n,P(Z))=1, \, \, (m,P(Z))=1}}\hspace{-20pt} W_{\CC}(mp_2p_3n) - \hspace{-20pt}\sum_{\substack{q,m,p_2,p_3,n  \\(qm,p_2,p_3) \in \mathcal{U}(7,1,2), \, \, Z \leq q < m \\
(n,P(Z))=1, \, \, (m,P(q))=1}}\hspace{-20pt}  W_{\CC}(qmp_2p_3n) \\
&=: S_{7,1,2,1}(\CC) - S_{7,1,2,2}(\CC),
\end{align*}
say.

\textbf{Sum $S_{7,1,2,1}(\CC)$:} Note that  $mp_3 > x^{1/2-\delta/2}$ implies by (\ref{hash}) that $p_2n < x^{1/2-\gamma+\delta},$ and we have $p_3<p_2<x^{1/4-20\delta}.$ Thus, we will apply Proposition \ref{ft} with $u=p_3,$ $v=p_2n,$ and $m$ as the implicit variable. To justify this properly, we need to remove the cross-conditions between $m$ and the other variables in such a way, that we use Proposition \ref{ft} only to sums where $m$ is not restricted. Similarly as with $S_{5,2},$ we write
\begin{align*}
\sum_{\substack{m,p_2,p_3, n  \\(m,p_2,p_3) \in \mathcal{U}(7,1,2) \\
(n,P(Z))=1, \, \, (m,P(Z))=1}} W_{\CC}(mp_2p_3n)  = \sum_{U,V_2,V_3} \sum_{\substack{m,p_2,p_3, n  \\(m,p_2,p_3) \in \mathcal{U}(7,1,2) \\
(n,P(Z))=1, \, \, (m,P(Z))=1\\
n \in (U^{1-\epsilon},U] \\
p_j \in (V_j^{1-\epsilon},V_j], \, \, j \in \{2,3\} } } W_{\CC}(mp_2p_3n),
\end{align*}
where the sum is over $U,V_2,V_3$ of the form $ x^{(1-\epsilon)^g}, $ $g\in \N,$ such that (note that $m=x^{1-\gamma+o(1)}/(UV_2V_3)$ by (\ref{hash}))
\begin{align*}
x^{1/4-\delta/2}  \leq  \frac{x^{1-\gamma+\delta}}{UV_2V_3} \leq W x^{2\delta},\quad Z \leq V_3 \leq V_2,  \quad V_2^{1-\epsilon} \leq x^{1/4 -10\delta}, \quad
\frac{x^{1-\gamma+\delta}}{UV_2V_3}V_3 \geq x^{1/2-\delta/2}, \\ \sqrt{\frac{x^{1-\gamma-\delta}}{UV_2V_3}}V_2^{1-\epsilon} < x^{0.36}, \quad \frac{x^{1-\gamma-\delta}}{UV_2V_3}(V_2V_3^2)^{1-\epsilon} < X^2, \quad  V_2V_3 \geq x^{1/4-20\delta}
\end{align*}
(that is, each condition in the definition of $\mathcal{U}(7,1,2)$ is loosened appropriately but at most by a factor of $x^{\mathcal{O}(\delta)}$). Since $S_{7,1,2}(\CC)$ has overall a negative sign, we only require an upper bound. Thus, we remove the cross-condition for $\CC=\A$ so that by Proposition \ref{ft}
\begin{align*}
S_{7,1,2,1}(\A) \, &\leq \, \sum_{U,V_2,V_3} \sum_{\substack{m,p_2,p_3, n \\
(n,P(Z))=1, \, \, (m,P(Z))=1\\
n \in (U^{1-\epsilon},U]  \\
p_j \in (V_j^{1-\epsilon},V_j], \, \, j \in \{2,3\} } } W_{\A}(mp_2p_3n) \\
&= \sum_{U,V_2,V_3} \sum_{\substack{p_2,p_3, n \\
(n,P(Z))=1\\
n \in (U^{1-\epsilon},U]  \\
p_j \in (V_j^{1-\epsilon},V_j], \, \, j \in \{2,3\} } } S(\A_{np_2p_3},Z) \\
&= \lambda S_{7,1,2,1}(\B) + E(\B) + \mathcal{O}\left( \lambda S(\B,X)\log^{-\delta/2} x \right),
\end{align*}
where the error term $E(\B)$ is again a sum where some combination of the variables is fixed up to a factor $x^{\mathcal{O}(\delta)},$ so that the sum has a deficiency $\ll \delta.$ Therefore, 
\begin{align*}
S_{7,1,2,1}(\A) \leq \lambda S_{7,1,2,1}(\B) +  \mathcal{O}\left( \delta \lambda S(\B,X) \right).
\end{align*}

\textbf{Sum $S_{7,1,2,2}(\CC)$:} Write
\begin{align*}
\sum_{\substack{q,m,p_2,p_3,n  \\(qm,p_2,p_3) \in \mathcal{U}(7,1,2), \, \, Z \leq q < m \\
(n,P(Z))=1, \, \, (m,P(q))=1}}  W_{\CC}(qmp_2p_3n) = \sum_{\substack{q,m,p_2,p_3  \\(qm,p_2,p_3) \in \mathcal{U}(7,1,2), \, \, Z \leq q < m \\
 (m,P(q))=1}} S(\CC_{qmp_2p_3},Z).
\end{align*}
 We first take out the part which has an asymptotic formula by Proposition \ref{ft} applied with $n$ as the implicit variable (cross-conditions again handled by the discussion of $S_{5,2}$); we are left with the range $\{(q,m,p_2,p_3): (qm,p_2,p_3) \in \mathcal{U}(7,1,2)\} \setminus \mathcal{V}  ,$ where (note that always $q < x^{1/4-20\delta}$)
\begin{align*}
\mathcal{V} = &\{mp_2p_3 < x^{1/2-\gamma}\} \cup \{m<x^{1/4-20\delta}, \, qp_2p_3 < x^{1/2-\gamma} \} \\ &\cup \{mp_1 <x^{1/4-20\delta}, \, qp_2 < x^{1/2-\gamma}\}\cup \{mp_2 <x^{1/4-20\delta}, \, qp_1 < x^{1/2-\gamma} \} \\
&\cup \{ qp_1 <x^{1/4-20\delta}, \,  mp_2 < x^{1/2-\gamma} \} \cup \{ qp_2 <x^{1/4-20\delta}, \,  mp_1 < x^{1/2-\gamma} \} .
\end{align*}
We also take out the parts where we have a Type II variable; thus, we are left with
\begin{align*}
\mathcal{W}(7,1,2)= \{(q,m,p_2,p_3): \, \, (qm,p_2,p_3) \in \mathcal{U}(7,1,2), &\, \, \,  Z \leq q  < m, \\ & \hspace{-60pt} qp_2p_3, mp_2p_3, \, qp_j, mp_j \notin [W,x^{1/2-\delta}] \} \setminus \mathcal{V} 
\end{align*}
 This remaining sum has the right sign so that it can be dropped, with a deficiency (cf. Lemma \ref{roleb}, $q=x^{\alpha_0},  m=x^{\alpha_1}, p_2=x^{\alpha_2}, p_3=x^{\alpha_3} $)
\begin{align*}
\mathcal{O}(\delta)+ \frac{1-\gamma}{\gamma} \int f_{7,1,2}(\boldsymbol{\alpha})\omega_2 (\boldsymbol{\alpha }) \frac{d\alpha_0 d\alpha_1d\alpha_2d\alpha_3 }{\alpha_0^2\alpha_2 \alpha_3} <  0.054317,
\end{align*}
where $\omega_2 (\boldsymbol{\alpha })=\omega((1-\gamma-\alpha_0-\alpha_1-\alpha_2-\alpha_3)/\gamma)\omega(\alpha_1/\alpha_0),$ and $f_{7,1,2}$ is the characteristic function of $\{(x^{\alpha_0}, x^{\alpha_1},x^{\alpha_2},x^{\alpha_3}) \in \mathcal{W}(7,1,2) \}$.

\textbf{Sum $S_{7,1,3}(\CC)$:} We take out the range with Type II variables and discard the rest to find a deficiency 
\begin{align*}
\mathcal{O}(\delta)+ (1-\gamma)\int f_{7,1,3}(\boldsymbol{\alpha})\omega(\boldsymbol{\alpha})\frac{d\alpha_1d\alpha_2d\alpha_3d\alpha_4}{\alpha_1\alpha_2\alpha_3\alpha_4^2} <  0.113006,
\end{align*}
where $f_{7,1,3}$ is the characteristic function of the four-dimensional set
\begin{align*}
\{\boldsymbol{\alpha}:& \, 1/4 \leq \alpha_1 < 1/2-\gamma, \,\, \gamma  \leq \alpha_4< \alpha_3< \alpha_2 < 1/4, \\ &\alpha_1 + \alpha_2 > 1/2, \,\, \alpha_1/2+ \alpha_2 < 0.36, \,\,  \alpha_1+\alpha_2+2\alpha_3 < 1-\gamma, \\ & \alpha_1+\alpha_2+\alpha_3 +2\alpha_4< 1-\gamma, \, \,\sum_{j \in I} \alpha_j \notin [1/2-\gamma, 1/2]  \quad \text{for every} \, \, I \subseteq \{1,2,3,4\}  \}.
\end{align*}

\subsubsection{Sum $S_{7,2}(\CC)$} We discard the sum $S_{7,2}(\CC),$ which gives a deficiency
\begin{align*}
\mathcal{O}(\delta)+ (1-\gamma)\int f_{7,2}(\boldsymbol{\alpha}) \omega(\boldsymbol{\alpha}) \frac{d\alpha_1 d\alpha_2}{\alpha_1\alpha_2^2} < 0.4425785
\end{align*}
where $f_{7,2}$ is the characteristic function of
\begin{align*}
\{\boldsymbol{\alpha}: \, 1/4  \leq \alpha_1 < 1/2-\gamma, \,\, \gamma  \leq \alpha_2 < 1/4, \,
\,   \alpha_1 +\alpha_2 > 1/2, \, \, \alpha_1/2+ \alpha_2 \geq 0.36 \}.
\end{align*}
\subsubsection{Deficiency of $S_7(\CC)$} The total deficiency of $S_7(\CC)$ is  < 0.6099015.

\subsection{Sum $S_8(\CC)$}
 This corresponds to the part where some ranges can be handled by the Type I$_2$ information in Chapter 5 of Harman's book \cite{Har}.  In our case, we have not obtained the Type I$_2$ information (cf. discussion after Proposition \ref{typeiestimate2}), so that we have to discard all of the sum. 
 The sum $S(\B_{p_1p_2},p_2)$ counts primes of size $x^{1-\gamma + o(1)}/(p_1p_2),$ since $p_2 > x^{1/4+o(1)}.$ Thus, the deficiency is 
\begin{align*}
\mathcal{O}(\delta)+ (1-\gamma)\int_{1/4}^{1/2-\gamma} \frac{d\alpha_1}{\alpha_1} \int_{1/4}^{\min\{ \alpha_1,(1-\gamma-\alpha_1)/2 \}} \frac{d\alpha_2}{(1-\gamma-\alpha_1-\alpha_2)\alpha_2} < 0.2021922
\end{align*}
\begin{remark} Since $p_2$ is large here, the deficiency from this range grows very slowly as gamma decreases.
\end{remark}
\subsection{Conclusion of the proof} Combining the above estimates we obtain
\begin{align*}
S(\A,X )  &=  S_1(\A) -S_2(\A) -S_3(\A) + S_5(\A) +S_6(\A) + S_7(\A) +S_8(\A) \\ 
&\geq  \lambda S_1(\B) -\lambda S_2(\B) -\lambda S_3(\B) + \lambda (S_5(\B) -0.1802756 \cdot S(\B,X))  \\
& \hspace{20pt} + \lambda S_6(\B) + \lambda(S_7(\B)-0.6099015 \cdot S(\B,X) ) \\
& \hspace{40pt}+\lambda (S_8(\B)-0.2021922\cdot S(\B,X)) -\mathcal{O}(\delta)\lambda S(\B,X) \\ 
&=(1-0.1802756-0.6099015-0.2021922-\mathcal{O}(\delta))\lambda S(\B,X)\\
& > 0.007 \cdot \lambda  S(\B,X)
\end{align*}
which completes the proof of Proposition \ref{mainp} with $\mathfrak{C}(1/19) = 0.007.$ For $\gamma > 1/19$ all of the deficiencies are strictly smaller, so that $\mathfrak{C}(\gamma) > 0.007$ for $\gamma > 1/19.$
\qed
\begin{remark} For $\gamma \geq 1/4$ the sum $S_4(\CC)$ is essentially empty, so that we actually get an asymptotic formula for $\gamma \geq 1/4$.
\end{remark}

The Python codes for computing the Buchstab integrals are available at (in the order of appearance)

\begin{tabular}{ c c }
$\mathcal{U}(5,3,1)$ & \url{http://codepad.org/rxR2O7Is} \\
$\mathcal{V}(5,3,3,1)$, four dimensional prime part & \url{http://codepad.org/fQKYi7hg} \\
$\mathcal{V}(5,3,3,1)$, five dimensional almost-prime part & \url{http://codepad.org/1SaVNuBy} \\
$\mathcal{U}(5,3,3,3)$ & \url{http://codepad.org/ZiNV3AuH} \\
$\mathcal{W}(7,1,2)$, with role reversal & \url{http://codepad.org/fVJHM3az} \\
Sum $S_{7,1,3}(\CC)$ & \url{http://codepad.org/G6Kx7IMg} \\
Sum $S_{7,2}(\CC)$ & \url{http://codepad.org/4RTwoAPk} \\
Sum $S_{8}(\CC)$ & \url{http://codepad.org/L4n8cLtY}
\end{tabular}

\bibliography{largefactor}
\bibliographystyle{abbrv}

\end{document}